\documentclass[reqno,final,oneside]{amsart}
\usepackage[
  a4paper,
  textwidth=16.3cm,
]{geometry}
\usepackage[dvipsnames]{xcolor}
\usepackage{mathrsfs}
\usepackage{enumitem}
\usepackage{mathtools}
\usepackage[final]{hyperref}
\usepackage{esint}
\usepackage{bm}
\usepackage{subcaption}
\usepackage{pgfplots}
\pgfplotsset{compat = newest}
\usepackage{soul}
\usepackage[english]{babel}

\newtheorem{theorem}{Theorem}[section]
\newtheorem{lemma}[theorem]{Lemma}
\newtheorem{proposition}[theorem]{Proposition}

\theoremstyle{definition}
  \newtheorem{definition}[theorem]{Definition}

\theoremstyle{remark}
  \newtheorem{remark}[theorem]{Remark}

\newcommand{\lp}{\varepsilon}
\newcommand{\N}{\mathbb{N}}

\newcommand{\R}{\mathbb{R}}
\newcommand{\C}{\mathbb{C}}
\newcommand{\Id}{\mathbf{Id}}
\newcommand{\id}{\mathbf{id}}
\newcommand{\eps}{\varepsilon}
\newcommand{\vphi}{\varphi}
\newcommand{\weakly}{\rightharpoonup}

\newcommand{\defas}{\coloneqq}
\newcommand{\sym}{\mathrm{sym}}

\newcommand{\cplen}{\mathcal{W}^{\mathrm{cpl}}}

\newcommand{\mechen}{\mathcal{M}}
\newcommand{\toten}{\mathcal{E}}
\newcommand{\totenalpha}{\toten_{\alpha,\theta_c}}
\newcommand{\inten}{W^{\mathrm{in}}}

\newcommand{\diss}{\mathcal{R}}

\usepackage{bbm}
\usepackage{dsfont}
\newcommand{\indic}{\mathds{1}}
\newcommand{\elpot}{W^{\mathrm{el}}}
\newcommand{\cplpot}{W^{\mathrm{cpl}}}
\newcommand{\hypot}{H}
\newcommand{\felpot}{W}
\newcommand{\disspot}{R}
\newcommand{\xiregnu}{\xi_{\nu, \alpha}^{\rm reg}}
\newcommand{\Wid}{\mathcal{Y}_{\id}}

\newcommand{\pl}{\partial}

\newcommand{\bt}{\theta_\flat}
\newcommand{\thetaepsnu}{\theta_{\eps,\nu}}
\newcommand{\thetaepsnuu}{\theta}
\newcommand{\thetanu}{\theta_{\nu}}
\newcommand{\yepsnu}{y_{\eps, \nu}}
\newcommand{\yepsnuu}{y}
\newcommand{\ynu}{y_{\nu}}
\newcommand{\wepsnu}{w_{\eps, \nu}}
\newcommand{\wnu}{w_{\nu}}
\newcommand{\meps}{m}
\newcommand{\dotyepsnu}{\partial_t{y}_{\eps,\nu}}
\newcommand{\dotnablayepsnu}{\partial_t \nabla {y}_{\eps,\nu}}
\newcommand{\dotnablayepsnuu}{\partial_t \nabla {y}}
\newcommand{\dotnablaynu}{\partial_t \nabla {y}_{\nu}}

\newcommand{\thetaeps}{\theta_{\eps}}
\newcommand{\yeps}{y_{\eps}}

\newcommand{\hc}{\mathbb{K}}
\newcommand{\hcm}{\mathcal{K}}

\newcommand{\drate}{\xi}

\newcommand{\haus}{\mathcal{H}}

\newcommand{\aC}{C_0}
\newcommand{\ac}{c_0}

\newcommand{\intQ}{\int_I\int_\Omega}

\newcommand{\intSN}{\int_I\int_{\Gamma_N}}

\newcommand{\CW}{\mathbb{C}_W}
\newcommand{\CD}{\mathbb{C}_D}
\newcommand{\mechenl}{\overline{\mechen}_0}
\newcommand{\rdrate}{\drate_\alpha^{\rm{reg}}}
\usepackage{stackengine}
\newcommand{\cdddot}{\mathrel{\Shortstack{{.} {.} {.}}}}
\newcommand*{\di}{\mathop{}\!\mathrm{d}}

\DeclareMathOperator{\dist}{dist}

\DeclareMathOperator{\trace}{tr}
\DeclareMathOperator{\diver}{div}
\DeclareMathOperator*{\esssup}{ess\,sup}
\DeclareMathOperator*{\essinf}{ess\,inf}

\DeclarePairedDelimiterX\setof[1]\{\}{#1}
\DeclarePairedDelimiterX\abs[1]\lvert\rvert{#1}
\DeclarePairedDelimiterX\norm[1]\lVert\rVert{#1}
\DeclarePairedDelimiterX\sprod[2]\langle\rangle{#1, #2}

\newcommand{\ZZZ}{\color{black}}
\newcommand{\III}{\color{black}}
\newcommand{\lll}{\color{black}}
\newcommand{\TTT}{\color{black}}

\newcommand{\BBB}{\color{black}}
\newcommand{\MMM}{\color{black}}
\newcommand{\MMMMM}{\color{black}}
\newcommand{\lenni}{\color{black}}
\newcommand{\NNN}{\color{black}}

\newcommand{\EEE}{\color{black}}
\newcommand{\LLL}{\color{black}}
\newcommand{\MKK}{\color{black}} 

\newcommand{\LMM}{\color{black}}
\newcommand{\AAA}{\color{black}}
\newcommand{\rb}{\color{black}}
\newcommand{\rbb}{\color{black}}
\newcommand{\ee}{\color{black}}
\newcommand{\martin}{\color{black}}
\newcommand{\rufb}{\color{black}}

\newcommand{\asdf}{\color{black}}

\usepackage{xargs}
\usepackage[textsize=tiny,textwidth=3cm]{todonotes}
\newcommandx{\task}[2][1=]{\todo[linecolor=blue,backgroundcolor=blue!30,#1]{#2}}
\newcommandx{\note}[2][1=]{\todo[linecolor=green,backgroundcolor=green!30,#1]{#2}}
\newcommandx{\noterb}[2][1=]{\todo[linecolor=orange,backgroundcolor=orange!30,#1]{#2}}
\newcommandx{\error}[2][1=]{\todo[linecolor=red,backgroundcolor=red!30,#1]{#2}}
\newcommandx{\attention}[2][1=]{\todo[linecolor=magenta,backgroundcolor=magenta!30,#1]{#2}}

\numberwithin{equation}{section}
\begin{document}
\title[Positive temperature in nonlinear thermoviscoelasticity]{Positive temperature in nonlinear thermoviscoelasticity and the derivation of linearized models}

\subjclass[2020]{35A15, \III 35Q74,  74A15, 74D05, \EEE  74D10}
\keywords{Thermoviscoelasticity, frame-indifferent viscous stresses, third law of thermodynamics, linearization.}

\author[R.~Badal]{Rufat Badal}
\address[Rufat Badal]{
  Department of Mathematics \\
  Friedrich-Alexander Universit\"at Erlangen-N\"urnberg \\
  Cauerstr.~11, D-91058 Erlangen, Germany
}
\email{rufat.badal@fau.de}

\author[M.~Friedrich]{Manuel Friedrich} 
\address[Manuel Friedrich]{%
  Department of Mathematics \\
  Friedrich-Alexander Universit\"at Erlangen-N\"urnberg \\
  Cauerstr.~11, D-91058 Erlangen, Germany 
}
\email{manuel.friedrich@fau.de}

\author[M.~Kru\v{z}\'ik]{Martin Kru\v{z}\'ik}
\address[Martin Kru\v{z}\'ik]{
  Czech Academy of Sciences \\
  Institute of Information Theory and Automation \\
  Pod vod\'arenskou v\v{e}\v{z}\'i 4, CZ-182 00 Praha 8, Czechia  \\
  \& Faculty of Civil Engineering \\
  Czech Technical University \\
  Th\'akurova 7, CZ-166 29 Praha 6, Czechia}
\email{kruzik@utia.cas.cz}

\author[L. Machill]{Lennart Machill}
\address[Lennart Machill]{
  Applied Mathematics M\"{u}nster \\
  University of M\"{u}nster \\
  Einsteinstr.~62, D-48149 M\"{u}nster, Germany (corresponding address)
}
\email{lennart.machill@uni-muenster.de}

\maketitle

\begin{abstract}
According to the Nernst theorem or, equivalently, the \AAA {third} \EEE law of thermodynamics, the absolute zero \III temperature \EEE is not attainable. Starting with an initial \MMM  positive \EEE temperature, we show that there exist solutions to a Kelvin-Voigt model for quasi-static nonlinear thermoviscoelasticity at a finite-strain setting \cite{MielkeRoubicek2020}, \rb obeying \ee an exponential-in-time lower bound on the temperature. 
 Afterwards, we focus on the case of deformations near the identity and temperatures near a critical \MMM {positive} \EEE temperature, and we show that weak solutions of the nonlinear system converge in a suitable sense to solutions of a system in linearized thermoviscoelasticity. Our result \rb extends \ee
the recent linearization result \AAA in \EEE \cite{BFK}\MMM, \III as it allows  \EEE the critical temperature to be positive.
\end{abstract}

\maketitle

\section{Introduction}

\ZZZ The \rb rheological \ZZZ Kelvin-Voigt model tracing back to  Lord {\sc Kelvin} (1824--1907) and Woldemar {\sc Voigt} (1850--1919) is a fundamental concept in engineering science. It serves as a tool for describing \rb the evolution \lenni of \ZZZ viscoelastic solids, where slow continuous deformations are observed, tending to a recoverable configuration of a material. The simplified schematic description in its linearized form involves an elastic and a viscous element (spring and dashpot), which are coupled in parallel, i.e., while both elements undergo the same deformation, they may cause different stresses. \martin  Here, the elastic element depends only on the displacement gradient, whereas the viscous element encounters its change in time. \EEE

\AAA The \EEE standard linear Kelvin-Voigt model is only valid \III for sufficiently small deformations \EEE and may break down if the undeformed and deformed configurations are significantly different. \EEE
The so-called large-strain deformation theory \MMM addresses \ZZZ this effect, leading to nonlinear stress-strain relations.  \III
%
In particular, by respecting the fundamental concept of frame indifference in nonlinear continuum mechanics, potentials of the first Piola-Kirchhoff and viscosity stress tensors must be written \III in terms of the right Cauchy-Green tensor and its time derivative, respectively, see \cite{Antmann98Physically}. \MMM In particular, \ZZZ the viscous stress is influenced by strain and strain rate.

\MMM 
\III As a time-dependent deformation of a body may generate heat due to viscosity (internal friction) and hence may influence the material properties, it is reasonable \EEE to couple the mechanical equations with a heat-transfer equation. Although the study of such models in thermoviscoelasticity has a long history dating back to pioneering work of {\sc Dafermos} \cite{dafermos}, only recently there have been advances in the investigation of nonlinear models respecting frame indifference  \cite{BFK,MielkeRoubicek2020}. \EEE In \AAA such  \EEE highly nonlinear and coupled situations, \martin essential \ee features are not yet well understood. In this article, we address the issue of positive temperature for the nonlinear system in \III\cite{BFK, MielkeRoubicek2020}, \EEE and discuss \III its \EEE relation to linearized models in thermoviscoelasticity. \ZZZ

 \MMMMM  We start by giving \EEE an overview \III on the existence theory \EEE for the underlying equations of motion, \AAA see \eqref{strong_formulation}--\eqref{strong_formulation_boundary_conditions} for their exact formulation. \EEE \lenni Already in the isothermal case, \MMM
 the nonlinear nature of the problem \ZZZ leads to the loss of monotonicity in the strain rate and makes the problem highly nontrivial.
\rb Existence of  global-in-time \rb weak solutions given initial data appropriately close \ZZZ to a smooth equilibrium was \rb first \ZZZ proven by {\sc Potier-Ferry} in \cite{potier-ferry-1,potier-ferry-2}, whereas subsequent articles provided a local\rb-\ZZZ in-time existence result \cite{Lewick}   and an existence result in the space of measure-valued solutions \cite{demoulini}. \EEE
The quasi-static version of the equations, i.e., without inertia, can be tackled through gradient flows in metric  \III spaces, \EEE as proposed in \cite{MielkeOrtnerSenguel14Anapproach}\rb, where the authors focus on the one-dimensional case, while also highlighting challenges \ZZZ in higher dimensions. \martin Resorting to energy densities with  higher-order spatial gradients, i.e., to    so-called nonsimple materials \cite{Toupin62Elastic, Toupin64Theory}\EEE\ZZZ,   existence of weak solutions \AAA has been shown in arbitrary space dimensions in  \III \cite{FiredrichKruzik18Onthepassage, MielkeRoubicek2020}. \EEE \EEE   Over the last years, \AAA these \EEE results were \ZZZ subsequently extended in various directions, including models allowing for self-contact \cite{gravina, kroemrou}, a nontrivial coupling with a diffusion equation \cite{liero}, homogenization \cite{gahn}, \MMM dimension reduction \cite{FK_dimred, MFLMDimension2D1D, MFLMDimension3D1D}, \ZZZ applications to fluid-structure interactions \cite{Schwarzacher}, and \lenni inertial  \ZZZ  effects \cite{Schwarzacher}.
\martin While the results mentioned above are formulated using the Lagrangian approach, several recent works employ the alternative Eulerian perspective instead,   see  \cite{Roubicek23Eulerian, Roubicek23Eulerian3,  Roubicek23Eulerian4}.

\MMM In the setting of thermoviscoelasticity, after the one-dimensional study in \cite{dafermos}, \ZZZ the first three-dimensional results appeared many years \MMMMM later \EEE \cite{Blanchard, Bonetti, Roubicek09}, exploiting the existence theory for parabolic equations with measure-valued data developed in \lll \cite{Boccardoetal, BoccardoGallouet89Nonlinear}. \MMM These results, however, \ZZZ are limited to linear viscous stresses. Nonlinear frame-indifferent \MMM models in \ZZZ thermoviscoelasticity were  analyzed only recently, first in \cite{MielkeRoubicek2020} and then subsequently  in \cite{BFK}, again exploiting stresses depending on higher-order gradients.
\AAA Both works derive existence of weak solutions \EEE which are \ZZZ consistent with the first two laws of thermodynamics: \MMM the first law, namely \MMMMM conservation \EEE of the total energy, up to the work induced by
the external loading or the \rb heat flux \ZZZ through the boundary, is addressed in \cite[Equation (2.21)]{MielkeRoubicek2020}. In contrast, the second law is expressed in the form of the Clausius-Duhem inequality, \MMMMM see \EEE \cite[Equation (2.22)]{MielkeRoubicek2020}. However, the question \martin of whether \EEE  weak solutions satisfy the third law of thermodynamics remained \rb open\ZZZ. According to \rb this \ZZZ law, also known as Nernst theorem, the temperature cannot \rb reach \III absolute zero. \EEE Similarly \III to \EEE the isothermal case, the Eulerian description has been recently used in thermoviscoelastic models, \lenni see~\cite{Roubicek23Eulerian2, Roubicek24Eulerian, Roubicek23Eulerian4}. \AAA Also there, \EEE the existence  results   only guarantee nonnegativity of the temperature.

\MMM In the first part of the article, we show that weak solutions of the model considered in \III\cite{BFK, MielkeRoubicek2020} \EEE indeed comply with the third law of thermodynamics. \MMM More precisely, this is achieved by proving an exponential-in-time lower bound on the temperature. 
\AAA To  our best knowledge, this is the first result proving \III positivity \EEE of the temperature in a fully nonlinear coupled system of thermoviscoelasticity.  \EEE 
  Our second result addresses \MMM the derivation of  \ZZZ linearized models for deformations near the identity and temperatures near a critical \emph{positive}  temperature $\theta_c >0$.  Here, we extend the \III work   in \cite{BFK}, \MMM where a \ZZZ linearization was performed around \MMM zero temperature \rb(\MMM $\theta_c =0$\rb)\MMM. (See also \cite{RBMFLM} for a \rb related \MMM problem \rb in \MMM dimension reduction.)  In \cite{BFK}, the argument was restricted to the case $\theta_c =0$ due to a missing a priori bound for the temperature below $\theta_c$.  We can now close this gap by suitably adapting the proof of the abovementioned exponential-in-time lower bound.  

\MMM While the \emph{nonnegativity} of the temperature for weak solutions has also been proved in nonlinear \MMMMM models \EEE \III\cite{BFK, MielkeRoubicek2020},  \EEE  it is considerably more challenging \ZZZ to show \emph{positivity} of the temperature. In fact, \ZZZ  such results in the literature are \MMM scarce, \ZZZ in \III particular, \EEE in \III highly \EEE nonlinear and coupled \III situations \EEE where the heat conductivity, the heat capacity, and the \lenni sink and source terms in \MMM the heat equation depend on deformation gradients and on \MMM the \ZZZ temperature itself. \rb Yet, another difficulty arises in the presence of a \III heat source with low integrability \EEE and an adiabatic \MMM heat-absorbing \rb term in the nonlinear heat equation. \MMM For instance, \rb the latter \MMM phenomenon is relevant in shape-memory alloys \cite{BJ92, Bha03} \III and  shape-memory polymers \cite{shapememorypolymers}, where different microstructures form \EEE upon cooling below a critical temperature. \ZZZ

To our best knowledge, the first result showing positivity of the temperature appeared in {\sc Colli  and Sprekels}  \cite{colli-sprekels} for a Fr\'{e}mond's model
of shape-memory alloys described in terms of linearized elasticity. 
\III  {\sc Paw\l ow and Zaj\c{a}czkowski} \cite{pawlow-zajackowski}  \AAA address \EEE positivity in a two-dimensional thermoelastic system with a \MMMMM mechanical equation governed by linear elasticity   and a \EEE nonlinear \lll heat-transfer equation \EEE with a constant heat-conductivity tensor, under the condition that solutions are sufficiently smooth. \EEE 
   \lenni This result was then \MMM  \ZZZ extended to \III a three dimensional model \EEE for shape-memory alloys described by a quasi-linear system in \cite{Yoshikawa-Pawlow-Zajackowski}, and to a linear Kelvin-Voigt type model in \cite{pawlow-zajackowski-1}, \AAA see also \cite{Roubicek-book}. \EEE

\MMM 
In \III\cite{BFK, MielkeRoubicek2020}, \EEE weak solutions have been identified \martin using \ee a time-discretized variational scheme, and the analysis of the corresponding minimization problem for the temperature directly showed that minimizers are nonnegative. Yet, this strategy cannot be transferred to the question of \lll preserving \martin the positivity \EEE of the temperature. For this, we follow a completely different approach. To explain the gist of the proof\ZZZ, we present the basic strategy in the simple case of a \rb classical \ZZZ heat equation
\rb\begin{equation}\label{EEE}
  \left\{
  \begin{aligned}
    c_V  \partial_t\theta - 
    \diver(\mathcal{K}  \nabla\theta)
    &= \AAA  h \EEE &&\text{in } [0,T] \times \Omega, \\
    \nabla \theta \cdot \nu &= 0 &&\text{on } [0,T] \times  \partial \Omega,
  \end{aligned}
  \right. 
    \end{equation}\ZZZ
    where $c_V$ and $\mathcal{K}$ are constants representing  the \emph{heat capacity} and the   \emph{heat conductivity}, respectively, $\nu$ denotes  the outward pointing unit normal on $\partial \Omega$, and $h\colon [0,T] \times \Omega \to [0,\infty)$ denotes \LMM \rb an \ZZZ external heat source. \MKK Consider a solution $\theta \colon [0,T] \times \Omega \to \R$ with $\theta(0) = \theta_0$ and $\inf_{x \in \Omega} \theta_0 \ge \lambda_0$ for some $\lambda_0 >0$.  Setting   $c_V=1$ and $\mathcal{K} = \mathbf{Id} \in \R^{d \times d }$ for simplicity, and letting $\lambda(t) :=  \lambda_0 \exp(-  t)$ for $t\in  [0,T] $ be the solution of    the \MMM differential equation \ZZZ $\frac{{\rm d}}{{\rm d}t} \lambda = - \lambda$, the goal is to show that $\theta(t) \ge \lambda(t)$ a.e.\ in $\Omega$ for all $t \in [0,T]$. \AAA This \EEE immediately provides positivity of the temperature \MMM and the \ZZZ exponential-in-time lower bound. \ZZZ Since $\theta_0 \ge \lambda_0$\lll, \EEE it suffices to show that  
\begin{align*}
  \frac{{\rm d}}{{\rm d}t} \int_\Omega  \frac{1}{2}  (\lambda(t)  - \theta(t))_+^2 \di x \le 0,
\end{align*}
where $(\cdot)_+$ denotes the positive part. This formally follows from the computation  
\begin{align}\label{FFF}
  \frac{{\rm d}}{{\rm d}t} \int_\Omega  \frac{1}{2}  (\lambda(t)  - \theta(t))_+^2 \di x  & = \int_\Omega (\lambda-\theta)_+ \big(\rb\tfrac{\di}{\di t} \ZZZ\lambda- \partial_t \theta \big) \, \di x = \int_\Omega (\lambda-\theta)_+  \big( -  \lambda - h - \Delta \theta  \big) \, \di x \notag \\ 
  &  = - \int_{\lbrace  \lambda \ge \theta \rbrace} |\nabla \theta|^2 \, \di x - \int_\Omega (\lambda-\theta)_+  \big(  \lambda  +  h)   \, \di x 
 \le 0,
\end{align} 
\martin where  we used the equation \eqref{EEE} in the second step, and in the third step, we performed an integration by parts. \EEE The actual realization of this computation in our framework is delicate, as $c_V$,  $\mathcal{K}$, and $h$ all depend on $\theta$ \emph{and} the deformation, and \eqref{EEE} is coupled additionally to a mechanical equation\lenni, see \eqref{strong_formulation} below. \EEE Moreover, a boundary term arises for nonzero Neumann boundary conditions, and in our setting $h$ can also be negative (but with $h \to 0$ as $\theta \to 0$), which complicates the last inequality in \eqref{FFF}. Further\rbb more\ee, the chain rule in the first step of \eqref{FFF} is intricate for \rb weak solutions and hence requires justification\ZZZ, see Section~\ref{sec:chainrule} for details. In fact, since the datum $h$ in \eqref{EEE} will only be in $L^1$, one \MMMMM expects \EEE $\partial_t \theta$ to \III have  \MMMMM low regularity. \EEE  Therefore, as an auxiliary step, we \ZZZ show a chain rule for a regularized problem. \MMM Then, \ZZZ once positivity of the regularized problem is established with bounds \MMM independent \ZZZ on the regularization itself, \MMM we \ZZZ send the regularization to zero \MMM and obtain the result for the original problem.

\AAA

In the second part of the paper, we focus on the case of small strains and temperatures close to a critical temperature $\theta_c \lll >0$, \EEE i.e., when  $ \nabla y - \Id $ is of order $\eps$ for some small $\eps >0$ and $\theta-\theta_c$ is of order $\eps^\alpha$ for some exponent $\alpha >0$. Then, in terms of rescaled displacements $u_\eps = \eps^{-1}(y - \id)$ and rescaled temperatures $\mu_\eps = \eps^{-\alpha} \III (\theta - \theta_c) \EEE$,  we rigorously pass to an effective linearized system as $\eps \to 0$, see \eqref{viscoel_small}--\eqref{initial_conds_lin}. \III With this, \AAA we contribute to the  understanding of the relations between nonlinear and linearized \III models, \EEE which \III has \AAA been an active field of research in the last years, see e.g.\  \cite{virginina, Braides-Solci-Vitali:07, DalMasoNegriPercivale02Linearized, DavoliFriedrich20Two-well, Friedrich:15-2, kostas, FiredrichKruzik18Onthepassage, JesenkoSchmidt21Geometric, MaininiPercivale21Linearization, MaorMora21Reference, Ulisse, Schmidt08Linear}.  In particular, from a modeling point of view, new interesting phenomena occur in the limiting system compared to \cite{BFK} where  linearization was performed in a rather   nonphysical case  $\theta_c=0$. Indeed, whereas   in \cite{BFK}   the mechanical and \III heat equation \AAA  decouple in a certain scaling regime \MMMMM for \EEE $\alpha$, in the present \III setting,  \EEE we  always   obtain a  coupled   system. Our argument relies on adapting the strategy in \eqref{FFF} for the choice $\lambda = \theta_c$.  This allows us to obtain  suitable    a priori bounds on $(\theta_c - \lll \theta)_+ \EEE $. For  all other a priori bounds we then rely on the strategy developed in   \cite{BFK}.


The plan of the paper is as follows. Section~\ref{sec:model} introduces the nonlinear and linearized models   and states our main results. Then, in Section~\ref{sec: reg sol}, we \rb show \martin the existence \EEE  of \ZZZ solutions to a \rb related \ZZZ regularized \rb model\ZZZ. 
Section~\ref{sec:chainrule} is devoted to the proof of a chain rule, which is \rb subsequently \ZZZ applied in Section~\ref{sec:strictposregu} to show the positivity of the temperature.
\martin In Section~\ref{sec:proofmaintheorem}, we perform the \rb rigorous \ZZZ linearization \rb at a positive critical temperature\ZZZ. \EEE 
While Section\rb s\ZZZ~\ref{sec:apriorireg}--\ref{sec: a priori} address \rb the derivation of   \ZZZ a priori bounds, the \AAA linear \EEE limiting equations are derived in Section~\ref{subsec: lin}.

\section{The model \MMM and main results\EEE}\label{sec:model}

\subsection*{Notation}
In what follows, we use standard notation for Lebesgue and Sobolev spaces. The lower index $_+$ means nonnegative elements, i.e., $L^2_+(\Omega)$ denotes the convex cone of nonnegative functions belonging to $L^2(\Omega)$\III, and  we \EEE  set $\R_+\defas [0,+\infty)$.  \lenni Given a measurable set $E$, $\indic_E$ denotes the \MMMMM characteristic \EEE function.  Let $a \wedge b \defas \min\setof{a, b}$ and $a \vee b \defas \max\setof{a, b}$ for $a, \, b \in \R$.
\martin Moreover, for any scalar function $f$, we write \ZZZ $f_+$ and $f_-$ \EEE for the positive \ZZZ and negative \AAA part, \EEE \lenni respectively. \EEE
Denoting by $d \ge 2$ the \rb space \ee dimension, we let $\Id \in \R^{d \times d}$  be the identity matrix, and  $\id(x) \defas x$ stands for the identity map on $\R^d$.
We define  the subsets  $SO(d) \defas \setof{A \in \R^{d \times d} \colon A^T A = \Id, \,  \det A = 1  }$, $GL^+(d) \defas \setof{F \in \R^{d \times d} \colon \det(F) > 0}$,  and  $\R^{d \times d}_\sym \defas \setof{A \in \R^{d \times d} \colon A^T = A}$.
Furthermore, for $F \in GL^+(d)$ we denote by $F^{-T} \defas (F^{-1})^T=(F^T)^{-1}$ the inverse of the transpose of $F$, and given a tensor $G$ (of arbitrary dimension), $\abs{G}$ \ZZZ indicates \EEE its Frobenius norm.
The scalar product between vectors, matrices, and \lenni third-order \EEE tensors will be written as $\cdot$, $:$, and $\cdddot$, \ZZZ respectively.
\AAA For \EEE $T \in \R^{d \times d \times d \times d}$ and $A \in \R^{d \times d}$,  $TA \in \R^{ d \times d}$ \rb is given by \ee \ZZZ $(TA)_{ij} = T_{ijkl} A_{kl}$  for $1 \leq i, \, j \leq d$, \AAA where we employ Einstein's summation convention. \EEE  \rb Any \lenni fourth-order tensor  $T \in \R^{d \times d \times d \AAA \times d}$ induces a bilinear form $T \colon \R^{d \times d} \times \R^{d \times d} \to \R$ given by $T[A,B] \defas TA : B = T_{ijkl} A_{kl} B_{ij}$ for any $A, \, B \in \R^{d \times d}$. \ee
As usual, generic constants   may vary from line to line.
If not stated otherwise, all constants \AAA  only \ee depend on \MMM the dimension \EEE $d$, on $p \geq \LLL 2 d \EEE$, \rb on \ee $\Omega$, \rb on a scalar \ee $\alpha \rb \in [1, 2]\ee$ \III introduced in Subsection~\ref{sec:linearization}, \EEE and the potentials \AAA and data \EEE defined \MMM in Subsection \ref{sec:setting}. \EEE

\subsection{\MMM Modeling assumptions\EEE}\label{sec:setting}
\MMM We start by introducing  the \EEE model of thermoviscoelasticity treated in \ZZZ \cite{BFK, RBMFLM, MielkeRoubicek2020}. \EEE Consider an open, bounded, \AAA and connected \MMM \emph{reference configuration} \EEE $\Omega \subset \R^d$ with Lipschitz boundary $\Gamma \defas \partial \Omega$.
Let $\Gamma_D, \, \Gamma_N$ be disjoint subsets of $\Gamma$ such that $\Gamma = \Gamma_D \cup \Gamma_N$ and $\haus^{d-1}(\Gamma_D) > 0$,   representing \emph{Dirichlet and Neumann parts} of the boundary, respectively.  
We further assume that $\Gamma_D$ itself has Lipschitz boundary in $\Gamma$. For $p \geq \LLL 2 d \EEE$, we introduce the set of \emph{admissible deformations} \MMM as \EEE
\begin{equation*}
  \Wid \defas \setof*{
    y \in W^{2, p}(\Omega; \R^d) \colon
    y = \id \text{ on } \Gamma_D, \,
    \det(\nabla y) > 0 \text{ in } \Omega
  },
\end{equation*}
and \rb further define the set \ee
\begin{equation}\label{def_Wzero}
   H^1_{\Gamma_D}(\Omega; \R^d)  \defas \setof{y \in \rb H^1\ee(\Omega;\R^d)  \colon y = 0 \text{ on } \Gamma_D}.
\end{equation}
Let $\ac, \, \aC$ with $0 < \ac < \aC < \infty $ be some fixed constants.
Our variational setting is as follows:

\subsection*{Mechanical energy and coupling energy:}
\ZZZ Adopting the concept of 2nd-grade nonsimple materials, see \cite{Toupin62Elastic, Toupin64Theory}, we assume that the \emph{mechanical energy} $\mechen \colon \Wid \to \R_+$ depends on both the gradient and the second gradient of a deformation $y \in \Wid$, and \EEE is   defined as the sum
\begin{equation}\label{mechanical}
  \mechen(y) \defas  \ZZZ  \int_\Omega \elpot(\nabla y) \di x + \int_\Omega \hypot(\nabla^2 y) \di x, \EEE
\end{equation}
\ZZZ where \rb the potentials \ZZZ $\elpot$ and $\hypot$ \rb have \AAA the \EEE following properties\ZZZ. The \emph{elastic energy density} $\elpot \colon GL^+(d) \to \R_+$ \AAA satisfies \EEE \MMM standard \EEE assumptions in nonlinear elasticity:
\begin{enumerate}[label=(W.\arabic*)]
  \item \label{W_regularity} $\elpot$ is $C^2$, and $C^3$ in a neighborhood of $SO(d)$;
  \item \label{W_frame_invariace} Frame indifference: $\elpot(QF) = \elpot(F)$ for all $F \in GL^+(d)$ and $Q \in SO(d)$;
  \item \label{W_lower_bound} Lower bound: $W^{\rm el}(F) \ge \ac \big(|F|^2 + \det(F)^{-q}\big) - \aC$ for all $F \in GL^+(d)$, where $q \ge \frac{pd}{p-d}$.
\end{enumerate}
The \ZZZ potential \EEE $\hypot \colon \R^{d \times d \times d} \to \R_+$ satisfies \III the following conditions: \EEE
\begin{enumerate}[label=(H.\arabic*)]
  \item \label{H_regularity} $\hypot$ is convex and $C^1$;
  \item \label{H_frame_indifference} Frame indifference: $\hypot(QG) = \hypot(G)$ for all $G \in \R^{d \times d \times d}$ and $Q \in SO(d)$;
  \item \label{H_bounds} $\ac \abs{G}^p \leq H(G) \leq \aC (1+ \abs{G}^p)$ and $\abs{\pl_G H(G)} \leq \aC \abs{G}^{p-1}$ for all $G \in \R^{d \times d \times d}$ \III and \EEE some $p\geq 2d$.
\end{enumerate}
Besides the mechanical energy, we introduce the \emph{coupling energy} \MMM which,  \ZZZ in addition to \EEE the deformation, also depends on \emph{temperature}\rb. More precisely, \EEE $\cplen \colon \Wid \times  L^1_+(\Omega)   \to \R$ \rb is \ee given by
\begin{equation}\label{couplenergy}
  \cplen(y, \theta) \defas \int_\Omega \cplpot(\nabla y, \theta) \di x,
\end{equation}
where its potential $\cplpot \colon GL^+(d) \times \R_+ \to \R$ satisfies \III the following conditions: \EEE
\begin{enumerate}[label=(C.\arabic*)]
  \item \label{C_regularity} $\cplpot$ is continuous, and $C^{\NNN 3}$ in $GL^+(d) \times (0, \infty)$;
  \item \label{C_frame_indifference} $\cplpot(QF, \theta) = \cplpot(F, \theta)$ for all $F \in GL^+(d)$, $\theta \geq 0$, and $Q \in SO(d)$;
  \item \label{C_zero_temperature} $\cplpot(F, 0) = 0$ for all $F \in GL^+(d)$;
  \item \label{C_lipschitz} $|\cplpot(F,\theta) - \cplpot(\tilde{F}, \theta)| \le \aC(1 + |F| + |\tilde{F}|)|F - \tilde{F}|$ for all $F, \, \tilde F \in GL^+(d)$, and $\theta \geq 0$;
  \item \label{C_bounds} For all $F \in  GL^+(d)$ and $\theta > 0$ it holds that
  \begin{align*}
    \abs{\partial_{F}^2 W^{\rm cpl}(F,\theta)} &\le \aC, &
    \abs{\pl_{F \theta} \cplpot(F, \theta)} & \leq \frac{\aC(1+|F|)}{\theta \vee 1}, &
    \ac & \leq -\theta \pl_\theta^2 \cplpot(F, \theta) \leq \aC.
  \end{align*}
\end{enumerate}
\MMM We \EEE remark that by \ref{C_zero_temperature} and the second bound in \ref{C_bounds}, $\pl_F \cplpot$ can be continuously extended to zero temperatures with $\pl_F W^{\rm cpl}(F, 0) = 0$ \rb for all $F \in GL^+(d)$\ee.
For $F \in GL^+(d)$ and $\theta \geq 0$, we define the \emph{total free energy potential} \rb as \ee
\begin{equation}\label{eq: free energy}
  \felpot(F, \theta) \defas \elpot(F) + \cplpot(F, \theta).
\end{equation}

\subsection*{Dissipation potential:} 
The \emph{dissipation functional} $\diss \colon \Wid \times \ZZZ H^1(\Omega;\R^d) \EEE \EEE \times  L^1_+(\Omega)  \to \R_+$  is defined as 
\begin{equation*}
  \diss( y, \partial_t y   , \theta)
  \defas \int_\Omega \disspot(\nabla   y,  \partial_t\nabla  y   , \theta) \di x,
\end{equation*}
where $\disspot \colon \R^{d \times d} \times \R^{d \times d} \times \R_+ \to \R_+$ is \MMM a \EEE \emph{potential of dissipative forces} satisfying
\begin{enumerate}[label=(D.\arabic*)]
  \item \label{D_quadratic} $\disspot(F, \dot F, \theta) \defas \frac{1}{2} D(C, \theta)[\dot C, \dot C] \defas \frac{1}{2} \dot C : D(C, \theta) \dot C$, where $C \defas F^T F$, $\dot C \defas \dot F^T F + F^T \dot F$, and $D \in C(\R^{d \times d}_\sym \times \R_+; \R^{d \times d \times d \times d})$ with $D_{ijkl} = D_{jikl}= D_{klij}$ for $1 \le i,\,j,\,k,\,l \le d$;
  \item \label{D_bounds} $\ac \abs{\dot C}^2 \leq \dot C : D(C, \theta) \dot C \leq \aC \abs{\dot C}^2$ for all $C, \, \dot C \in \R^{d \times d}_\sym$, and $\theta \geq 0$.
\end{enumerate}
The fact that $\disspot$ can be written as a function depending on the right Cauchy-Green tensor $C = F^T F$ and its time derivative $\dot C$ is equivalent to \emph{dynamic frame indifference}, \MMM see e.g.~\ZZZ\cite{Antmann98Physically}. \EEE
The symmetries of $D$ stated in \ref{D_quadratic} yield (see e.g.~\cite[\MMM Equation \EEE (2.8)]{BFK})
\begin{equation}\label{chain_rule_Fderiv}
  \partial_{\dot F} R(F, \dot F, \theta) = 2 F (D(C, \theta) \dot C).
\end{equation}
Moreover, we define the associated \emph{dissipation rate} $\drate \colon \R^{d \times d} \times \R^{d \times d} \times \R_+ \to \R_+$ as
\begin{align}\label{diss_rate}
  \drate(F, \dot F, \theta)
  \defas \pl_{\dot F} \disspot(F, \dot F, \theta) : \dot F
   = 2 F (D(C, \theta) \dot C) : \dot F
  = D(C, \theta) \dot C : (\dot F^T F + F^T \dot F)
   = 2 R(F, \dot F,\theta),
\end{align}
where the second identity follows from \eqref{chain_rule_Fderiv}, and the third from the symmetries stated in \ref{D_quadratic}.

\subsection*{Heat \III conductivity:} \EEE
The map $\hc \colon  \R_+  \to \R^{d \times d}_\sym$ \MMM denotes \EEE the \rb temperature-dependent \ee \emph{heat conductivity tensor} of the material in the deformed configuration.
We require that $\hc$ is continuous, symmetric, uniformly positive definite, and bounded.
More precisely, for all  $\theta \geq 0$  it holds that
\begin{equation}\label{spectrum_bound_K}
  \ac \leq  \hc(\theta)  \leq \aC,
\end{equation}
where the inequalities are meant in the eigenvalue sense.
We \rb further \ee define the pull-back $\hcm \colon  GL^+(d)  \times \R_+ \to \R^{d \times d}_\sym$ of $\hc$ into the reference configuration by (see also \cite[\ZZZ Equation~(2.24)\EEE]{MielkeRoubicek2020})
\begin{equation}\label{hcm}
   \hcm(F, \theta)  \defas \det(F) F^{-1}  \hc(\theta)  F^{-T}.
\end{equation}

\subsection*{Internal energy:} 
The  density    of the \EEE \textit{(thermal part of the)  internal energy}   $\inten \colon GL^+(d) \times (0, \infty) \to \R$ \ZZZ is given by \EEE 
\begin{equation}\label{Wint}
  \inten(F, \theta) \defas \cplpot(F, \theta) - \theta \pl_\theta \cplpot(F, \theta).
\end{equation}
 \MMM Then, we define the   \emph{heat capacity} by \EEE 
\begin{equation}\label{inten_mon}
 c_V(F,\theta) \defas \partial_{\theta} \inten (F, \theta)
  = -\theta \pl_\theta^2 \cplpot(F, \theta) \in [\ac, \aC]
  \qquad \text{for all $F \in GL^+(d)$ and $\theta  >  0$},
\end{equation}
\MMM where \ZZZ the bounds follow from  \EEE the third bound in \ref{C_bounds}. Hence, \III using \EEE \ref{C_zero_temperature}, the following \III relation between the internal energy and the temperature  holds true: \EEE
\begin{equation}\label{inten_lipschitz_bounds}
  \ac \theta \leq \inten(F, \theta) \leq \aC \theta.
\end{equation}
\MMM This also shows \EEE that $\inten$ can be continuously extended to zero temperatures by setting $\inten(F, 0) = 0$ for all $F \in GL^+(d)$.


\rb 
We remark that the above assumptions on the potentials $\elpot$, $\hypot$, $\cplpot$, and $\disspot$ coincide with the ones in \cite[Section 2.1]{BFK} up to a higher power $p \geq 2d$ instead of $p > d$ in the definition of the hyperelastic potential $\hypot$ \AAA (needed in \eqref{A_3estimate2.5} and \eqref{lennitrick} below). \EEE
For a comparison of the above conditions with the ones stated in \cite[Section 2]{MielkeRoubicek2020}, we refer to \MMMMM \cite[Remark~2.1]{BFK}. \EEE

\MMM Without further notice, all properties on the potentials  introduced above are assumed throughout the paper. Later, for specific results we require refined \MMMMM bounds \EEE which \rb will be \MMM always indicated explicitly. \EEE

\subsection*{Equations of nonlinear thermoviscoelasticity:}

Fixing a finite time horizon $T > 0$,   let us from now on shortly write $I \defas [0, T]$.
\MMM We \EEE consider a \emph{dead force} $\MMM  f \EEE \in W^{1, 1}(I; L^2(\Omega; \R^d))$, a \emph{boundary traction} $\MMM g \EEE \in W^{1, 1}(I; L^2(\Gamma_N; \R^d))$, and an \emph{external temperature} $ \MMM \theta_{\flat} \EEE  \in \NNN L^2(I; L^{ 2}_+(\Gamma))$.  We study thermoviscoelastic materials, \EEE governed by the following system of equations
\begin{subequations}\label{strong_formulation}
\begin{align}
  f &=
    -\diver\big(
      \pl_F \felpot(\nabla y, \theta)
      + \pl_{\dot F} \disspot(\nabla y, \partial_t \nabla y, \theta)
      - \diver(\pl_G \hypot(\nabla^2 y))
    \big), \label{strong_formulation_mechanical} \\
  \MMM c_V \EEE  (\nabla y, \theta) \, \partial_t{\theta} &=
    \diver\big(\hcm(\nabla y, \theta) \nabla \theta\big)
    + \xi (\nabla y, \partial_t \nabla y, \theta)
    + \theta \pl_{F \theta} \cplpot(\nabla y, \theta) : \partial_t \nabla y, \label{strong_formulation_thermal}
\end{align}
\end{subequations}
which is complemented by  \emph{initial conditions}
\begin{equation}\label{strong_formulation_initial_conditions}
  y\NNN(0)\EEE = y_{0} \in \Wid \qquad \text{and} \qquad \theta(0) = \theta_{0} \in L^2_+(\Omega),
\end{equation}
and the \emph{boundary conditions}
\begin{subequations}\label{strong_formulation_boundary_conditions}
\begin{align}
  \big(
    \pl_F \felpot(\nabla y, \theta)
    + \pl_{\dot F} \disspot(\nabla y, \partial_t \nabla y, \theta)
  \big) \nu
  - \diver_S \big( \pl_G \hypot(\nabla^2 y) \nu \big)
    &= g & &\text{ on }  I \times \Gamma_N, \label{strong_formulation_boundary_g} \\
  y &= \id & &\text{ on } I \times \Gamma_D, \label{strong_formulation_boundary_id}\\
  \pl_G \hypot(\nabla^2 y) : (\nu \otimes \nu)
    &= 0 & &\text{ on }I \times  \Gamma, \label{strong_formulation_boundary_H} \\
  \hcm(\nabla y, \theta) \nabla \theta \cdot  \nu  + \kappa \theta
    &= \kappa \theta_{\flat} & &\text{ on } I \times \Gamma. \label{strong_formulation_heat_transfer}
\end{align}
\end{subequations}
Above, $\nu$ denotes the outward pointing unit normal on $\Gamma$ and $\kappa \ge 0$ is a \emph{phenomenological heat-transfer coefficient} on $\Gamma$. Moreover, $\diver_S$ represents the \emph{surface divergence}, defined by $\diver_S(\cdot) = \trace(\nabla_S(\cdot))$, where $\trace$ denotes the trace and $\nabla_S \defas (\Id - \nu \otimes \nu) \nabla$ denotes the surface gradient (see e.g.~\cite[\ZZZ Equations~(2.28)--(2.29)\EEE]{MielkeRoubicek2020} for further details).  
\MMM We refer to \cite[Section 2]{MielkeRoubicek2020} for the derivation of the equations and details on the physical meaning of each term. 

 The first part of this paper \MMMMM addresses \III the existence of solutions where the temperature is not \EEE   only nonnegative but actually \emph{positive}. In the second part, we will perform a linearization of the system at a critical temperature $\theta_c > 0$ and small strains. \EEE

\subsection{Positivity of temperature in large-strain thermoviscoelasticity}
We consider the following  notion of weak solutions.
\begin{definition}[Weak solution of the nonlinear system]\label{def:weak_formulation-classic}
A couple $(y, \theta) \colon I \times \Omega \to \R^d \times \R$ is called a \emph{weak solution} of the initial-boundary-value problem \ZZZ \eqref{strong_formulation}--\eqref{strong_formulation_boundary_conditions} \EEE if and only if $y \in L^\infty(I; \Wid) \cap H^1(I; H^1(\Omega; \R^d))$ with $y(0, \cdot)= y_{0}$ \lll a.e.~in $\Omega$, \EEE $\theta \in L^1(I; W^{1,1}(\Omega))$ with $\theta \NNN \geq \EEE 0$ a.e.~in $I \times \Omega$, and if it satisfies the identities
\begin{equation}\label{weak_formulation_mechanical}
\begin{aligned}
  &\intQ \pl_G
    \hypot(\nabla^2 y) \cdddot \nabla^2 z
    + \Big(
      \pl_F \felpot(\nabla y, \theta)
      + \pl_{\dot F} \disspot(\nabla y, \partial_t \nabla y, \theta)
    \Big) : \nabla z \di x \di t \\
  &\quad=  \intQ f \cdot z \di x \di t
    +  \int_I \int_{\Gamma_N} g \cdot z \di \haus^{d-1} \di t
\end{aligned}
\end{equation}
for any test function $z \in C^\infty(I \times \overline{\Omega}; \R^d)$ with $z = 0$ on $I \times \Gamma_D$, as well as  
\begin{equation}\label{weak_limit_heat_equation}
\begin{aligned}
  &\intQ \hcm(\nabla y, \theta) \nabla \theta \cdot \nabla \vphi
    -\big(
      \drate(\nabla y, \partial_t \nabla y, \theta)
      + \pl_F \cplpot(\nabla y, \theta) : \partial_t \nabla y
    \big) \vphi
     - \inten(\nabla y, \theta) \partial_t \vphi   \di x \di t \\
  &\quad = \kappa \int_I \int_{\Gamma} (\theta_{\flat} - \theta) \vphi \di \haus^{d-1} \di t
     +\int_\Omega \inten(\nabla y_{0}, \theta_{0}) \, \vphi(0) \di x
\end{aligned}
\end{equation}
for any test function $\vphi \in C^\infty(I \times \overline \Omega)$ with $  \varphi(T) = 0$.
\end{definition}
In \ZZZ \cite[Theorem~2.3(ii)]{BFK} and \cite[Theorem 2.2]{MielkeRoubicek2020}, existence of weak solutions \rb for \ee the initial-boundary-value problem \ZZZ \eqref{strong_formulation}--\eqref{strong_formulation_boundary_conditions} \EEE  \MMM  \ZZZ in the sense of Definition \ref{def:weak_formulation-classic} \EEE is shown. \AAA One can \MMMMM check \EEE that sufficiently smooth weak solutions lead to the classical formulation \eqref{strong_formulation} along with the boundary conditions \eqref{strong_formulation_boundary_conditions}, see e.g.~the reasoning after \ZZZ \cite[Equation~(2.28)]{MielkeRoubicek2020}. \EEE

\MMM We \EEE  stress the important requirement of a.e.~\emph{nonnegativity} of solutions in the above definition.
This can be seen as a physical justification of the system \ZZZ \eqref{strong_formulation}--\eqref{strong_formulation_boundary_conditions} \EEE as it assures that along the evolution the temperature inside the material \III never drops \EEE below absolute zero. Nevertheless, the current existence theory potentially allows for temperatures reaching absolute zero in a set of non-negligible Lebesgue measure contradicting the \emph{third law of thermodynamics}.
The first main result of this paper shows that, under \MMM mild additional \EEE assumptions  \ZZZ compared to \cite{BFK}  \EEE  (on the potentials as well as \MMM on the \EEE boundary and initial conditions), \MMM there exist \EEE   weak solutions in the sense of Definition \ref{def:weak_formulation-classic} that are a.e.~\emph{strictly positive}.

The additional requirements on the coupling potential $\cplpot$ and \AAA the \EEE internal energy $W^{\rm in}$ are as follows: 
\begin{enumerate}[label=(C.\arabic*)]
  \setcounter{enumi}{5}
  \item \label{C_third_order_bounds}
  \ZZZ The function \EEE $\partial_{F\theta\theta} W^{\rm cpl}$  can be continuously extended to $GL^+(d) \times \R_+$ and satisfies \newline $\abs{ \III \partial_{F\theta\theta} W^{\rm cpl}(F,\theta)} \EEE
  \leq C_0 (1+ \vert F \vert)$  for all $F \in GL^+(d)$ and $\theta \geq 0$\III; \EEE
  \item \label{C_Wint_regularity}  
  $W^{\rm in}$ \MMM can be continuously extended to \rb a map in $C^{3}(GL^+(d) \times \R_+; \R_+)$ \ee and satisfies $|\partial_\theta^2 W^{\rm in}(F, \theta)| \leq C_0 (1+ \vert F \vert)$ for all $F \in GL^+(d)$ and $\theta \geq 0$.
\end{enumerate}
\rb In \AAA the example in Appendix \ref{expl:shapememory}, \EEE we will show that the classes of free energy potentials introduced in \ZZZ \cite[Example 2.4 and Example 2.5]{MielkeRoubicek2020} \rb contain examples which comply with all the abovementioned \MMMMM conditions. \EEE


\begin{theorem}[Positivity of the temperature]\label{thm:positivity_of_temperature}
   Assume that \ref{C_third_order_bounds}--\ref{C_Wint_regularity} hold. \EEE  Suppose \EEE that $\theta_{0, \rm min} \defas \essinf_{x \in \Omega} \theta_0 > 0$ and that there exists a constant $\tilde C > 0$ such that $\bt(t) \geq \theta_{0, \rm min} \exp(-\tilde C t)$ for all $t \in I$.  
 \ZZZ Then, there exists a weak solution $(y, \theta)$ of the boundary value problem \ZZZ \eqref{strong_formulation}--\eqref{strong_formulation_boundary_conditions} \EEE  in the sense of Definition~\ref{def:weak_formulation-classic} \EEE and a constant $C>0$ such that $\theta(t, x) \geq C^{-1} \exp(-C t)$ for a.e.~$(t,x)\in I \times \Omega$.
\end{theorem}

\MMM We emphasize that our proof relies on an approximation scheme for weak solutions and  \MMMMM that \EEE the statement of Theorem \ref{thm:positivity_of_temperature} holds for \rb weak solutions which arise as \MMM limits of such approximate solutions. Therefore, it is not guaranteed that \emph{every} weak solution in the sense of Definition \ref{def:weak_formulation-classic} \rb is strictly positive\ee, but we can prove only the existence of such a solution. \rb In this regard, the \martin  above-mentioned  \EEE approximation procedure \MMM may serve as a selection principle for physically relevant weak solutions. \EEE

\subsection{Linearization at a positive temperature}\label{sec:linearization}

\MMM In the second main result of this article, we perform a linearization of the system \eqref{strong_formulation}--\eqref{strong_formulation_boundary_conditions} for deformations close to the identity $\id$ and for temperatures close to a critical  \emph{positive} temperature   $\theta_c$. \EEE In particular, this overcomes a modeling issue of the linearization result in \III \cite{BFK}, \EEE where linearization was performed in a rather \ZZZ nonphysical \EEE case of temperatures close to absolute zero. \EEE \MMM  We fix a parameter  $\eps \in (0, 1]$ representing  the magnitude of the elastic strain.  For the temperature $\theta$ instead, we assume that $\theta - \theta_c$ is of order $\eps^\alpha$ for some $\alpha >0$.  In order to guarantee \ZZZ that \rb solutions \ZZZ  compl\rb y \EEE with \martin the smallness \EEE of strains and \AAA with small deviations \EEE from the critical temperature, \rb we   require appropriate $\eps$-scalings \ee for initial configurations, external loadings, boundary tractions, and external temperatures. More precisely, we assume that the data $f$, $g$, $\theta_{\flat}$ in \eqref{strong_formulation} and \eqref{strong_formulation_boundary_conditions} are replaced by 
\begin{align}\label{def:externalforces}
f_\eps \defas \eps f, \qquad g_\eps \defas \eps g, \qquad \theta_{\flat, \eps} = \theta_c + \eps^\alpha \mu_\flat
\end{align}
\ZZZ for \EEE $  f  \in W^{1, 1}(I; L^2(\Omega; \R^d))$,   $ g  \in W^{1, 1}(I; L^2(\Gamma_N; \R^d))$, and   $ \mu_\flat  \in  L^2(I; L^{ 2}(\Gamma))$,  and that the initial conditions in \eqref{strong_formulation_initial_conditions} take the form 
\begin{equation}\label{linearization_initial_conditions}
  y_{0, \eps} \defas \id + \eps u_0 \quad \text{and} \quad
  \theta_{0,\eps} \defas \theta_c + \eps^\alpha \mu_0,
\end{equation}
 where \ZZZ $u_0 \in W^{2,p}(\Omega;\R^d)\cap H^1_{\Gamma_D}(\Omega; \R^d)$  \EEE and $\mu_0 \in L^{2}(\Omega)$. \MMM To ensure that small strains imply small stresses, we need to assume that  \EEE
\begin{enumerate}[label=(W.\arabic*)]
  \setcounter{enumi}{3}
  \item \label{W_prefers_id} $\AAA  W \EEE (F, \theta_c) \geq c_0 \dist(F,SO(d))^2$ for all $F \in GL^+(d)$ and $\AAA  W \EEE (F, \theta_c) = 0$ if $F \in SO(d)$;
\end{enumerate}
\begin{enumerate}[label=(H.\arabic*)]
  \setcounter{enumi}{3}
  \item \label{H_prefers_id} $H(0) = 0$.
\end{enumerate} 
\MMM These are natural requirements to perform linearization for viscoelastic materials, see e.g.\ \cite{FiredrichKruzik18Onthepassage}. \EEE

\subsection*{Formal derivation of the linearized system}

\III Rewriting \eqref{strong_formulation}--\eqref{strong_formulation_boundary_conditions}
   in terms of  the \EEE
   \emph{rescaled displacement}  $u = \lp^{-1} (y - \id)$ and the \emph{rescaled temperature} $\mu = \lp^{-\alpha}(\theta - \theta_c)$,  dividing \eqref{strong_formulation_mechanical} by $\eps$ and \eqref{strong_formulation_thermal} by $\eps^\alpha$, and letting $\lp \to 0$ we obtain, at least formally, the system
\begin{equation}\label{viscoel_small}
\left\{
\begin{aligned}
  - \diver \big( \CW e(u) + \C_D e(\partial_t u)  + \mathbb{B}^{(\alpha)} \mu \big) &= f, \\
	\bar c_V \partial_t\mu - \diver(\mathbb{K}(\theta_c) \nabla \mu) 
  &= \CD^{(\alpha)} e(\partial_t u): e(\partial_t u)
    + \theta_c \hat{\mathbb{B}}   : e(\partial_t u),
\end{aligned}
\right.
\end{equation}
along with the boundary conditions
\begin{align}\label{viscoel_small_bdy}
	u &= 0 \text{ on } I \times \Gamma_D, \qquad
  \big( \CW e(u) + \CD e(\partial_t u)    + \mathbb{B}^{(\alpha)} \mu\big) \nu = g \text{ on } I \times \Gamma_N, \notag \\
 &\qquad \qquad \qquad \mathbb{K}(\theta_c) \nabla \mu \cdot  \nu  + \kappa \mu
   = \kappa \mu_\flat \text{ on } I \times \Gamma \ZZZ , \EEE
\end{align}
and \MMM the \EEE  initial conditions
\begin{equation}\label{initial_conds_lin}
  u(0) = u_0, \quad \mu(0) = \mu_0.
\end{equation}
Here, $e(u) \defas \frac{1}{2} (\nabla u + (\nabla u)^T)$ denotes the linearized strain tensor, and the tensors of elasticity and viscosity coefficients are given by
\begin{align}\label{def_WD_tensors}
  \C_W &\defas \partial^2_{F}\elpot(\Id) +\partial^2_{F} W^{\rm cpl}(\Id,\theta_c), &
  \C_D &\defas \partial^2_{\dot F} R(\Id,  0 ,   \theta_c) = 4D(\Id,\theta_c).
\end{align}
Moreover, $\overline c_V$   corresponds to a constant heat capacity of the linearized model \MMM which \EEE is related to  $c_V$  \MMM in  \EEE the nonlinear model  (see  \eqref{inten_mon}) \MMM by \EEE  
\begin{equation}\label{linearized_heat_capacity}
  \bar c_V \defas c_V(\Id, \theta_c).
\end{equation}
\rb Furthermore\MMM, $\hat{\mathbb{B}}$ is defined by 
\begin{align}\label{Bhatt}
\hat{\mathbb{B}} = \lim_{\eps \to 0}  \eps^{1-\alpha}   \partial_{F\theta} W^{\rm cpl} (\Id,\theta_c),
\end{align} \ZZZ and,   eventually, \EEE  the $\alpha$-dependent tensors are given by
\begin{align}\label{alpha_dep}
  \mathbb{B}^{(\alpha)} &= \begin{cases}
 +\infty &  \text{if } 0 < \alpha < 1,\\  
    \hat{\mathbb{B}}   & \text{if } \alpha =1,  \\ 0 & \text{if } \alpha >1 \ZZZ, \EEE \end{cases}\qquad  \qquad 
  &\CD^{(\alpha)} = \begin{cases}
    0 \ZZZ  \EEE & \text{if } 0 <  \alpha < 2 \III , \EEE \\
    \CD  & \text{if } \alpha =2, \\
    + \infty & \ZZZ \text{if } \EEE \alpha >2.
  \end{cases}
\end{align} 
\MMM As in the \MMM  linearization at zero temperature \ZZZ \cite{BFK}, \EEE the limiting model is only relevant in the range $\alpha \in [1,2]$ due to \eqref{alpha_dep}. In contrast to  \ZZZ \cite[Equation~(2.29)]{BFK}\MMM, the limiting heat equation in \eqref{viscoel_small} features the additional term \ZZZ $\theta_c \hat{\mathbb{B}} : e(\partial_t u) =  \theta_c \partial_{F\theta} W^{\rm cpl} (\Id,\theta_c): e(\partial_t u) $ if $\alpha = 1$. \MMM In order to allow for the entire range $\alpha \in [1,2]$, we suppose that the limit in \eqref{Bhatt} exists \emph{for all} $\alpha \in [1,2]$ which corresponds to an $\eps$-dependent coupling potential with $\partial_{F\theta} W^{\rm cpl} (\Id,\theta_c) \sim  \eps^{\alpha-1}$, i.e., to an asymptotically vanishing material parameter \MMMMM (for $\alpha>1$). \EEE This choice is also reflected in assumption \ref{C_adiabatic_term_vanishes} below.  (For \MMMMM notational \EEE convenience, we write $W^{\rm cpl}$ instead of $W^{\rm cpl}_\eps$. \MMMMM All conditions \ref{C_regularity}--\ref{C_Wint_regularity} and the ones mentioned below hold uniformly in $\eps$.)  

From a modeling point of view, \ZZZ $\mathbb{C}_W^{-1}\mathbb{B}^{(\alpha)}$  \EEE can be interpreted \EEE  as a  \emph{thermal expansion matrix} of the linearized evolution whereas $\hat{\mathbb{B}}$ \EEE plays the role of a heat source \MMMMM and \EEE sink, \ZZZ see \cite[Section~8.3]{KruzikRoubicek19mathmodels}. \EEE It is worth noting that, \AAA due to the presence of $\hat{\mathbb{B}}$, \EEE we cannot expect $\mu$ in \eqref{viscoel_small} to be nonnegative. \MMMMM In fact, \EEE it represents the (rescaled) deviation from the critical \EEE temperature $\theta_c$.  
Interestingly, the equations decouple if  $\alpha \in (1,2)$ and $\hat{\mathbb{B}} = 0$.
In the case $\alpha = 2$, the linearized heat equation additionally depends on the linearized mechanical equation via the linearized dissipation rate term $\C_D^{(\alpha)} e(\partial_t u) : e(\partial_t u)$ which can be interpreted as friction. \AAA The \EEE temperature contributes to the linearized mechanical equation only in the case $\alpha = 1$.

 Eventually, although the nonlinear system is given for a nonsimple material, as  a consequence of the growth conditions in \ref{H_bounds}, in the limit we obtain equations without spatial gradients of $e(u)$. \EEE 
\EEE

\ZZZ Finally, we address \martin the \ee properties of the tensors defined above. By Taylor expansion, polar decomposition, and frame indifference (see \ref{W_regularity}, \ref{W_frame_invariace}, \ref{C_regularity}, \ref{C_frame_indifference}, and \ref{D_quadratic}) one
can observe that the tensors $\mathbb{C}_W$ and $\mathbb{C}_D$ only depend on the symmetric part of the strain
and strain rate, respectively. Similarly, \ref{C_frame_indifference} implies that
$\hat{\mathbb{B}}$ is symmetric. Moreover, \AAA using additionally \EEE \ref{W_prefers_id} and \ref{D_bounds}, we see that the tensors induce positive definite quadratic forms on $\R^{d \times d}_{\rm sym}$, i.e., there exists a constant $c>0$ such that \begin{align}\label{positivedefiniteness}
\mathbb{C}_W[A,A] \geq c \,\vert \sym(A) \vert^2 \qquad \text{ and } \qquad \mathbb{C}_D[A,A] \geq c \,\vert \sym(A) \vert^2\quad \quad \text{for all $A \in \R^{d \times d}$}.
\end{align}
\textbf{Additional assumptions.}  \EEE
For the rigorous linearization procedure, we \AAA need \EEE to truncate the dissipation rate if $\alpha < 2$, similarly to \ZZZ \cite{BFK}. \EEE
More precisely, given $\Lambda \geq 1$, we define a truncated version $\xi^{(\alpha)} \colon GL^+(\R^d) \times \R^{d \times d} \times \R_+ \to \R_+$ of the dissipation rate as \MMM
\begin{equation}\label{def_xi_alpha}
  \xi^{(\alpha)}(F, \dot F, \theta) \defas \begin{cases}
    \xi(F, \dot F, \theta)
      &\text{if } \alpha \in [1,2]
      \text{ and } \xi(F, \dot F, \theta) \leq \Lambda, \\
   \Lambda^{1-\alpha/2}  \xi(F, \dot F, \theta)^{\alpha/2}
      &\text{if } \alpha \in [1,2]
      \text{ and } \xi(F, \dot F, \theta) > \Lambda.
  \end{cases}
\end{equation} \EEE
Notice that in the case $\alpha = 2$ no truncation is applied as we have $\xi^{(\alpha)} = \xi$.  \MMM For  $\alpha \in [1,2)$, the dissipation is changed for large strain rates. \martin Since we deal with small strains and strain rates, we heuristically have $\xi \le 1$, and the system is essentially not affected. \EEE  Indeed, this regularization has no influence on the effective model in  \eqref{viscoel_small}--\eqref{initial_conds_lin}.

With this regularization at hand, weak solutions in the nonlinear setting are defined as follows.  \EEE


\begin{definition}[Weak solution of the \AAA regularized \EEE nonlinear system]\label{def:weak_formulation}
A couple $(y, \theta) \colon I \times \Omega \to \R^d \times \R$ is called a \emph{weak solution} of the initial-boundary-value problem \ZZZ \eqref{strong_formulation}--\eqref{strong_formulation_boundary_conditions} \MMMMM (with initial conditions  $y_{0,\eps} \in \Wid $  and $\theta_{0,\eps}\in L^2_+(\Omega)$) \EEE \EEE if and only if $y \in L^\infty(I; \Wid) \cap H^1(I; H^1(\Omega; \R^d))$ with $y(0, \cdot)= y_{0,\eps}$  \lll a.e.~in $\Omega$, \EEE $\theta \in L^1(I; W^{1,1}(\Omega))$ with $\theta \NNN \geq \EEE 0$ a.e.~in $I \times \Omega$, and if it satisfies the identities  
\begin{equation}\label{weak_formulation_mechanical_eps}
\begin{aligned}
  &\intQ \pl_G
    \hypot(\nabla^2 y) \cdddot \nabla^2 z
    + \Big(
      \pl_F \felpot(\nabla y, \theta)
      + \pl_{\dot F} \disspot(\nabla y, \partial_t \nabla y, \theta)
    \Big) : \nabla z \di x \di t \\
  &\quad=  \intQ f_\eps \cdot z \di x \di t
    +  \int_I \int_{\Gamma_N} g_\eps \cdot z \di \haus^{d-1} \di t
\end{aligned}
\end{equation}
for any test function $z \in C^\infty(I \times \overline{\Omega}; \R^d)$ with $z = 0$ on $I \times \Gamma_D$, as well as  
\begin{equation}\label{weak_limit_heat_equation_eps}
\begin{aligned}
  &\intQ \hcm(\nabla y, \theta) \nabla \theta \cdot \nabla \vphi
    -\big(
      \drate^{(\alpha)}(\nabla y, \partial_t \nabla y, \theta)
      + \pl_F \cplpot(\nabla y, \theta) : \partial_t \nabla y
    \big) \vphi
     - \inten(\nabla y, \theta) \partial_t \vphi   \di x \di t \\
  &\quad = \kappa \int_I \int_{\Gamma} (\theta_{\flat,\eps} - \theta) \vphi \di \haus^{d-1} \di t
     +\int_\Omega \inten(\nabla y_{0,\eps}, \theta_{0,\eps}) \, \vphi(0) \di x
\end{aligned}
\end{equation}
for any test function $\vphi \in C^\infty(I \times \overline \Omega)$ with $  \varphi(T) = 0$.
\end{definition}

\ZZZ Existence of weak solutions in the sense of Definition~\ref{def:weak_formulation} for truncations of the form \eqref{def_xi_alpha} was shown in \cite[Proposition 2.5(ii)]{BFK}   for   the choice $\Lambda = 1$. The existence \AAA result \EEE  extends to general truncations as given in \eqref{def_xi_alpha} \MMM in a  straightforward way. \EEE

Due to technical reasons, \MMM for the rigorous linearization, we need additional assumptions: we \EEE require  \III that \EEE
 \begin{enumerate}[label=(C.\arabic*)]
  \setcounter{enumi}{7}
    \item   \label{C_adiabatic_term_vanishes}  \III there \EEE exists $\Lambda>0$ such that \EEE for all $F \in GL^+(d)$ it holds that  
  \begin{align*}
\vert \partial_{F\theta} W^{\rm cpl} (F,\theta_c) \vert \leq  
\MMM C_0 \lll (1+ \vert F \vert)  \Big( \EEE \eps^{\alpha-1}  \wedge \frac{1}{\Lambda} \lll \Big) \EEE \quad \  \text{ and } \ \quad  \vert \partial_{FF\theta} W^{\rm cpl} (F,\theta_c) \vert \leq  
C_0 \eps^{\alpha-1} \III ; \EEE
\end{align*} 
  \item \label{C_more_third_order_bounds} \III for \EEE all $F \in GL^+(d)$ and $\theta > 0$ it holds that
  \begin{equation*}
\MMMMM
    \vert  \partial_{F \theta \theta}   W^{\rm cpl}(F, \theta) \vert   \leq \frac{C_0(1+ \vert F \vert)}{ \III (\theta \vee 1)^2 \EEE }; \EEE
  \end{equation*}
  \item \label{C_entropy_vanishes} $\vert \MMM \partial_{\theta}^2 W^{\rm in}(F,\theta) \EEE \vert \leq c_V(F,\theta)  \III \frac{1}{2\theta_c}  \EEE$ for all $F \in GL^+(d)$ \asdf and $\theta > 0$. \EEE 
\end{enumerate}
\MMM The scaling in \ref{C_adiabatic_term_vanishes} has been motivated \III in the discussion \EEE below \eqref{alpha_dep}. The condition \EEE 
 \III in \ref{C_more_third_order_bounds} \MMMMM is a technical requirement \asdf and allows \EEE to \EEE control  the remainder resulting from Taylor expansions.
 \III  Due to \ref{C_entropy_vanishes}, the strategy \MMMMM in the proof of \III Theorem~\ref{thm:positivity_of_temperature} can be adjusted to derive a suitable $\eps$-dependent bound \MMMMM on  $( \theta_c- \theta)_+$, \III see Proposition~\ref{prop:lowerbound}. \EEE 
 \ZZZ Notice that the conditions in \ref{C_more_third_order_bounds}--\ref{C_entropy_vanishes} particularly refine the bounds in \ref{C_third_order_bounds}--\ref{C_Wint_regularity}.

\MMM 
\subsection*{Passage to the linearized model.} \EEE
Recalling \eqref{def_Wzero},  \MMM we start by defining \EEE  weak solutions of the linearized system \eqref{viscoel_small}--\eqref{initial_conds_lin}.

\begin{definition}[Weak solution of the linearized system]\label{def:weak_form_linear_evol}
A pair $(u, \mu) \colon I \times \Omega \to \R^d \times \R$ is a \emph{weak solution} to the initial-boundary-value problem \eqref{viscoel_small}--\eqref{initial_conds_lin} if $u \in H^1(I; H^1_{\Gamma_D}(\Omega; \R^d))$ with $u(0) = u_0$ \lll a.e.~in $\Omega$\EEE, $\mu \in L^1(I; W^{1,1}(\Omega))$, and if the following identities hold true:
\begin{equation}\label{linear_evol_mech}
  \intQ \big( \CW e(u) + \CD e(\partial_t u)   + \mathbb{B}^{(\alpha)} \mu  \big) : \nabla z \di x \di t
    = \intQ f \cdot z \di x \di t + \intSN g \cdot z \di \haus^{d-1} \di t
\end{equation}
for any $z \in C^\infty(I \times \overline\Omega; \R^d)$ with $z = 0$ on $I \times \Gamma_D$, as well as
\begin{align}\label{linear_evol_temp}
  &\intQ
      \mathbb{K}(\theta_c) \nabla \mu \cdot \nabla \vphi
      - \CD^{(\alpha)} e(\partial_t u) : e(\partial_t u) \vphi
      -  \theta_c \hat{\mathbb{B}} : e(\partial_t u) \vphi
      - \bar c_V \mu \partial_t \vphi \di x \di t \notag \\
  &= \kappa \int_I \int_\Gamma (\mu_\flat - \mu) \vphi \di \haus^{d-1} \di t
    + \bar c_V \int_\Omega \mu_0 \vphi(0) \di x
\end{align}
for any $\vphi \in C^\infty(I \times \overline\Omega)$   with $\vphi(T) = 0$. 
\end{definition}
Indeed, it is a standard matter to check that sufficiently smooth weak solutions lead to the classical formulation \AAA \eqref{viscoel_small}--\eqref{initial_conds_lin}. \EEE We are ready to state  our second main result.
\begin{theorem}[Passage to linearized thermoviscoelasticity at positive temperatures]\label{thm:linearization_positive_temp}
\MMM Suppose that \ref{W_prefers_id}, \ref{H_prefers_id},   and \ref{C_third_order_bounds}--\ref{C_entropy_vanishes} hold. \EEE Given $\alpha \in \MMM [1, \EEE 2]$ and $\eps \in (0, 1]$, we assume that the \MMM data and the initial conditions are as in \eqref{def:externalforces}--\eqref{linearization_initial_conditions}. For $\alpha=1$, we further assume that $\Lambda$ in \eqref{def_xi_alpha} and \ref{C_adiabatic_term_vanishes}  is \rb chosen \ee large  enough. \EEE
Then, the following holds true:
\begin{enumerate}[label=(\alph*)]
\item\label{linearization_comp} There exists a sequence of weak solutions $((y_{\eps}, \theta_{\eps}))_\eps$ in the sense of Definition~\ref{def:weak_formulation}  such that the \rb rescaled \ee functions $u_{\eps} \defas \eps^{-1} (y_{\eps} - \id)$ and $\mu_\eps \ZZZ \defas \EEE \eps^{-\alpha}(\theta_{\eps} - \theta_c)$ satisfy  
\begin{align}
  u_{\eps} &\to u \text{ in } L^\infty(I;H^1(\Omega; \R^d)), &
  \partial_t u_{\eps} &\to \partial_t u \text{ in } L^2(I; H^1(\Omega; \R^d)), \label{convergence:u} \\ 
  \mu_{\eps} &\to \mu \text{ in } L^s(I \times \Omega), & 
  \mu_{\eps} &\weakly \mu \text{ weakly in } L^r(I; W^{1, r}(\Omega)) \label{convergence:mu}
\end{align}
for any $s \in [1,\frac{2}{\alpha} + \frac{4}{\alpha d})$ and $r \in [1,  \frac{2d+4}{\alpha d +2})$. 
\item \label{linearization_limit} The limit $(u, \mu)$ from \ref{linearization_comp} is the unique weak solution of \eqref{viscoel_small}--\eqref{initial_conds_lin} in the sense of Definition~\ref{def:weak_form_linear_evol}.
\end{enumerate} 
\end{theorem}
\ZZZ Observe that \EEE it is not necessary to select a subsequence in the previous theorem due to the uniqueness of \AAA the solution to \EEE the limit problem.


\subsection*{Example}

\ZZZ In \cite{MielkeRoubicek2020}, the \rb authors provide \ZZZ a family of free energy potentials 
modeling austenite-martensite transformations in so-called shape-memory alloys, where the free energy potential in \eqref{eq: free energy}  takes the form
\begin{equation}\label{freeenergyexample}
  W(F, \theta) = (1 - a(\theta)) W_M(F) + a(\theta) W_A(F) + C_1 \theta (1 -  \log\theta).
\end{equation}
\martin Here, $C_1 > 0$ denotes a constant, and $a\colon \R_+ \to [0,1]$ \AAA represents \EEE the volume fraction \AAA  between   austenite and \EEE  martensite, which we assume to depend only on temperature. \EEE
Moreover, $W_M$ and $W_A$ denote the potentials governing the  martensite \EEE and \AAA the \EEE austenite states, respectively.
For  convenience of the reader,   \AAA  in \lll Lemma~\ref{lem:compatibilityshapememory} and \EEE Lemma~\ref{lem:specificchoice} of  Appendix \ref{expl:shapememory} we address \ZZZ suitable choices of $W_M$, $W_A$, and $a$ \rb complying \ZZZ with all the aforementioned conditions \ref{W_regularity}--\ref{W_prefers_id} \rb as well as \ZZZ \ref{C_regularity}--\ref{C_entropy_vanishes}.

\section{Regularized solutions} \label{sec: reg sol}

\MMM
The main results of this paper (Theorem \ref{thm:positivity_of_temperature} and Theorem \ref{thm:linearization_positive_temp}) crucially rely on the chain rule established in Section \ref{sec:chainrule}.  For technical reasons, \III   this chain rule \MMMMM is proved \III for solutions $(y, \theta)$ that particularly satisfy \EEE \ZZZ $W^{\rm in}(\nabla y, \theta) \in H^1(I;(H^1(\Omega))^*)$. \MMM To achieve this regularity, we first solve an  approximate system of \ZZZ equations \ZZZ where   $\xi$ in \eqref{weak_limit_heat_equation} or \EEE $\xi^{(\alpha)}$ in \eqref{weak_limit_heat_equation_eps} \rb are \ZZZ replaced by a regularization $\xi_{\nu, \alpha}^{\rm reg}$ defined in \eqref{def_xi_alpha_reg} below. Although the applicability of the chain rule relies on the regularity  \AAA of \EEE approximate solutions, the \MMMMM corresponding \EEE a priori bounds \AAA do \EEE not depend on the regularization itself. \AAA In particular, as $\nu \to 0$,  \EEE approximate solutions  \rb converge \MMM to solutions as given in  Definition \ref{def:weak_formulation-classic} and Definition \ref{def:weak_formulation}, respectively.
\rb This will allow us to \MMM recover properties of the original \ZZZ solutions, namely, \MMM positivity of the temperature and a priori bounds.

To keep the argument concise and  to treat both settings at the same time, we consider in the sequel general parameters   $\eps \in (0,1]$ and $\alpha \in [1, 2]$. We mention, however,  that \III in the context of \EEE Theorem \ref{thm:positivity_of_temperature} it suffices to set $\eps=1$ and $\alpha = 2$. \ZZZ Indeed, the notion of weak \rb solutions \ZZZ in Definition~\ref{def:weak_formulation-classic} and Definition~\ref{def:weak_formulation} coincide in this case. \EEE

Given $\nu \in \ZZZ (0,1] \EEE$, consider $\xi^{\rm reg}_{\alpha, \nu} \colon GL^+(d) \times \R^{d \times d} \times \R_+ \to \R_+$ given by
\begin{equation}\label{def_xi_alpha_reg}
  \xi_{\nu, \alpha}^{\rm reg}(F, \dot F, \theta)
  \defas \begin{cases}
    \xi^{(\alpha)}(F, \dot F, \theta)
    &\text{if } \xi^{(\alpha)} \leq \nu^{-1}, \\
    \nu^{1/\alpha-1} \xi^{(\alpha)}(F, \dot F, \theta)^{1/\alpha}
    &\text{else.}
  \end{cases}
\end{equation}
\AAA The key idea of the regularization is that the available  a priori  bound  $\xi(\nabla y, \partial_t \nabla  y, \theta) \in L^1(I \times \Omega)$ for weak solutions  ensures a required $L^2(I \times \Omega)$ bound on $\xiregnu (\nabla y, \partial_t \nabla  y, \theta)$. \EEE \MMM Moreover, notice that $\xiregnu \nearrow \xi^{(\alpha)}$ as $\nu \searrow 0$. We start with the definition of weak solutions. \EEE

\begin{definition}[Weak solution of \AAA $\nu$-regularized \MMMMM problem\EEE]\label{def:weak_solutions_regularized} 
 A couple $(y, \theta)$ is called a \emph{weak solution} to a regularized version of the initial-boundary-value problem  \ZZZ \eqref{strong_formulation}--\eqref{strong_formulation_boundary_conditions} \MMMMM (with initial conditions  $y_{0,\eps} \in \Wid $  and $\theta_{0,\eps}\in L^2_+(\Omega)$) \EEE if and only if $y \in L^\infty(I; \Wid) \cap H^1(I; H^{1}(\Omega; \R^d))$ with $y(0)= y_{0,\eps}$  \lll a.e.~in $\Omega$, \EEE $\theta \in L^2(I; H^1(\Omega)) \cap \ZZZ C(I; L^2(\Omega)) \EEE $ with $\theta \geq 0$ a.e.~in $I \times \Omega$, $\theta(0) = \theta_{0,\eps}$  \lll a.e.~in $\Omega$, \EEE and $w \defas W^{\rm in}(\nabla y, \theta) \in H^1(I; (H^{1}(\Omega))^*)$, and if it satisfies the identities
\begin{equation}\label{weak_limit_mechanical_equation_nu}
\begin{aligned}
  &\intQ \pl_G
    \hypot(\nabla^2 y) \cdddot \nabla^2 z
      + \big(
      \pl_F \felpot(\nabla y, \theta)
      + \pl_{\dot F} \disspot(\nabla y, \partial_t \nabla y, \theta)
    \big) : \nabla z    \di x \di t \\ \rb &\quad
  = \ee \intQ f_\eps \cdot z \di x \di t
    +   \int_I \int_{\Gamma_N} g_\eps \cdot z \di \haus^{d-1} \di t
\end{aligned}
\end{equation}
for any test function $z \in C^\infty(I \times \overline{\Omega}; \R^d)$ with $z = 0$ on $I \times \Gamma_D$, as well as  
\begin{equation}\label{weak_limit_heat_equation_nu}
\begin{aligned}
  &\intQ \hcm(\nabla y, \theta) \nabla \theta \cdot \nabla \vphi
    -\big(
      \NNN\xi_{\nu, \alpha}^{\rm reg}\EEE(\nabla y, \partial_t \nabla y, \theta)
      + \pl_F \cplpot(\nabla y, \theta) : \partial_t \nabla y
    \big) \vphi
     \di x \di t  + \int_I \langle \partial_t w , \vphi \rangle \di t  \\
  &\quad= \kappa \int_I \int_{\Gamma} (\theta_{\flat,\eps} - \theta) \vphi \di \haus^{d-1} \di t
\end{aligned}
\end{equation}
for any test function $\vphi \in \NNN L^2(I ;  H^1( \Omega))\EEE$, where  \rb $\langle \cdot, \cdot \rangle$ \ee in \eqref{weak_limit_heat_equation_nu} \III denotes \EEE the dual pairing of    $H^1(\Omega)$ and $(H^{1}(\Omega))^*$.
\end{definition}

\AAA
The weak formulations in Definition~\ref{def:weak_formulation-classic} and Definition~\ref{def:weak_formulation}  differ from \eqref{weak_limit_heat_equation_nu} by an integration by parts: here, the time derivative is applied to the solution $w$ instead of $\varphi$.  This stronger formulation has the advantage that the class of test functions in \eqref{weak_limit_heat_equation_nu} is larger and does not require regularity in time. In particular, this will allow us to test \eqref{weak_limit_heat_equation_nu} with functions of the form  $\vphi\indic_{[0,t]}$  for any $t \in I$\ee. \EEE   

\III
Before addressing the existence of weak solutions, we  \EEE recall  \eqref{mechanical} and introduce the functionals $\ell_\eps(t)$ on $H^1(\Omega;\R^d)$ defined by
\begin{align}\label{def:forcefunctional}
\langle \ell_\eps(t), v \rangle \defas \int_\Omega f_\eps(t) \cdot v \di x + \int_{\Gamma_N} g_\eps(t) \cdot v \di \mathcal{H}^{d-1}\rb, \ZZZ
\end{align}  \EEE \MMM 
\rb where $\langle \cdot, \cdot \rangle$ denotes the dual pairing between $H^1(\Omega; \R^d)$ and  $(H^1(\Omega; \R^d))^*$. \III
Then,  the following existence and convergence result holds. \EEE
\begin{proposition}[\lll Solutions to the regularized system]\label{thm:existence_positivity_regularized} 
\MMMMM For each \EEE $\nu>0$, $\alpha \in [1,2]$, and $\eps \in (0, 1]$,  
the following holds:   

\noindent  \lll \emph{(i)\hspace*{5.6pt}Existence:}  There exists  \AAA a  \MMMMM  $\nu$-regularized \EEE weak solution $(\yepsnu, \thetaepsnu)$  \EEE  in the sense of Definition \ref{def:weak_solutions_regularized}.

\noindent  \lll \emph{(ii) Energy balance:} \III For all~$t \in I$,  the solution $(\yepsnu, \thetaepsnu)$ satisfies \EEE
\begin{align}\label{energybalanceregularized}
&\mathcal{M}(\yepsnu(t)) +   \int_0^t \int_\Omega \xi(\nabla \yepsnu, \dotnablayepsnu, \thetaepsnu) \di x \di s  \nonumber \\
& \qquad = \mathcal{M}(\yepsnu(0))  + \int_0^t \langle \ell_\eps(s), \dotyepsnu \AAA (s) \EEE \rangle \di s - \int_0^t \int_\Omega \partial_F W^{\rm cpl} (\nabla \yepsnu, \thetaepsnu) : \dotnablayepsnu\di x \di s.
\end{align}

\noindent  \lll \emph{(iii) Uniform bounds:} \III There exists some $M>0$  such that for all $\eps  \in (0,1]$ and $\nu \in (0,1]$ \III the solution $(\yepsnu, \thetaepsnu)$ satisfies \EEE \begin{align}
 &    \esssup_{t \in I}   \III \mechen(   \yepsnu(t) ) \EEE  \leq M, \label{toten_bound_scheme}\\
   &           \norm{\yepsnu}_{L^\infty(I;W^{2, p}(\Omega)\ZZZ ) \EEE} \leq M  , \label{higherorerbounds} \\
&\norm{\yepsnu}_{L^\infty(\ZZZ I; \EEE C^{1, 1-d/p}(\Omega))} \leq M, \ \
    \norm{(\nabla \yepsnu)^{-1}}_{L^\infty(\ZZZ I; \EEE C^{1 - d/p}(\Omega))} \leq M, \ \
    \det(\nabla \yepsnu) \geq \frac{1}{ M}  \text{ \rb a.e. \MMM in } I \times \Omega, \label{pos_det} \\
&  \textstyle \int_I \int_\Omega \ZZZ \xi (\nabla \yepsnu, \dotnablayepsnu, \thetaepsnu) \di x \di t \leq M, \label{bound:dissipationepsnu}
\end{align} \EEE
where $M$ is independent of $\nu$, $\eps$, \MMM and $\Lambda$ in \III \eqref{def_xi_alpha}. \EEE

\noindent  \lll \emph{(iv) Convergence:} \LLL Given a sequence of solutions $(\yepsnu, \thetaepsnu)_\nu$ \ZZZ  to the $\nu$-regularized  \EEE system in Definition~\ref{def:weak_solutions_regularized}, \EEE there exists a subsequence (not relabeled) such that \III   $\yepsnu \to y_\eps$ in $L^\infty(I; W^{1,\infty}(\Omega;\R^d))$, $\thetaepsnu \to \theta_\eps$ in  $L^1 (I \times \Omega)$, and $(y_\eps,\theta_\eps)$ \EEE is a weak solution to the boundary value problem \eqref{strong_formulation}--\eqref{strong_formulation_boundary_conditions} in the sense of \ZZZ   Definition~\ref{def:weak_formulation}.
 \end{proposition}
\III The existence of solutions in similar scenarios, even without regularization, can be shown by a variational time-discretization scheme \cite{BFK, MielkeRoubicek2020}.
We highlight that versions of the statements \eqref{toten_bound_scheme}--\eqref{bound:dissipationepsnu} with fine dependence on the scaling of external loadings and initial data will be proved in Proposition~\ref{lem:fineapriori} below. 
Moreover, part (iv) of the statement also encompasses convergence to weak solutions in the sense of Definition \ref{def:weak_formulation-classic} as this corresponds to the case $\eps = 1$ and $\alpha = 2$.

 
\begin{proof} 
\MMMMM We start the proof by noting that, in view of \eqref{def_xi_alpha} and \eqref{def_xi_alpha_reg},  for $\xi \ge \max \lbrace \nu^{-1}, \Lambda \rbrace$ it holds that
\begin{align}\label{very new}
\xi_{\nu, \alpha}^{\rm reg} = \nu^{\frac{1}{\alpha}-1} \Lambda^{\frac{1}{\alpha} - \frac{1}{2}} \xi^{1/2}. 
\end{align}
This shows that for $\xi$ large enough the regularization coincides, up to a constant, with the one considered in \cite[Equation~(2.35) for $\alpha = 1$]{BFK}. Therefore, \cite[Proposition~2.5(ii)]{BFK} (again for the choice $\alpha = 1$ therein) and   \cite[Remark~4.3(iii)]{BFK} yield the existence of a weak solution $(y_{\eps,\nu}, \theta_{\eps,\nu})$ to the system in the sense of Definition~\ref{def:weak_formulation} with $\xi^{(\alpha)}$ replaced by $\xiregnu$, satisfying  particularly $\theta_{\eps,\nu} \in L^2(I; H^1(\Omega))$. (Indeed, minor adaptations of the proof show that the regularization in \cite[Equation~(2.35)]{BFK} can be replaced by $\xiregnu$ as they have the same qualitative behavior at zero and infinity.)  

\MMMMM To complete the proof of (i), we need \III to recover the stronger notion of weak solutions in Definition~\ref{def:weak_solutions_regularized}. Firstly,   \eqref{weak_limit_mechanical_equation_nu} and the regularity of $y_{\eps,\nu}$ coincide in both notions. \MMMMM Therefore, it remains to show $w_{\eps,\nu} \defas W^{\rm{in}}(y_{\eps,\nu}, \theta_{\eps,\nu}) \in H^1(I;(H^1(\Omega))^* )$, $\theta_{\eps,\nu} \in C(I; L^2(\Omega) )$, as well as \eqref{weak_limit_heat_equation_nu}. For the regularity properties, we will particularly  make use of the regularity results stated in  Lemma \ref{lem:deriphi}(ii),(iii) below. 

Firstly, we show $w_{\eps,\nu}   \in H^1(I;(H^1(\Omega))^* )$. By $\theta_{\eps,\nu} \in L^2(I; H^1(\Omega))$ and   Lemma \ref{lem:deriphi}(ii)
below we get $w_{\eps,\nu}   \in L^2(I; H^1(\Omega))$, so we can focus on $\partial_t w_{\eps,\nu}$.  To this end, for a.e.~$t \in I$ we can define \III the \asdf functional \EEE $\sigma(t)$ by    \MMMMM    
\begin{align*}
  \langle \sigma(t),\varphi  \rangle &\defas  \int_\Omega \big(      \xi_{\nu, \alpha}^{\rm reg}(\nabla y_{\eps,\nu}, \partial_t \nabla y_{\eps,\nu}, \theta_{\eps,\nu})   + \pl_F \cplpot(\nabla y_{\eps,\nu}, \theta_{\eps,\nu}) : \partial_t \nabla y_{\eps,\nu} \big)   \vphi
     \di x \di t  \\
   &  \qquad  - \int_\Omega \hcm(\nabla y_{\eps,\nu}, \theta_{\eps,\nu}) \nabla \theta_{\eps,\nu}\cdot \nabla \vphi
          \di x \di t      + \kappa  \int_{\Gamma} (\theta_{\flat,\eps} - \theta_{\eps,\nu}) \vphi \di \haus^{d-1}  
\end{align*}    \III
for every $\varphi \in  \asdf H^1(\Omega)$, \III where all functions appearing on the right-hand side are evaluated at $t$.
      Then, as $(y_{\eps,\nu}, \theta_{\eps,\nu})$ is a weak solution in the sense of Definition \ref{def:weak_formulation} \MMMMM (with $\xi_{\nu, \alpha}^{\rm reg}$ in place of $  \drate^{(\alpha)}$), \III we see that for every $\psi  \in C^\infty_{c} (I)$ and $\varphi \in \asdf C^\infty(  \overline{\Omega}) \EEE$ \EEE it holds that
      \begin{equation*}
        \int_I \langle \sigma(t), \varphi \rangle \psi(t)\di t
        = - \int_I \int_{\Omega} w_{\eps,\nu} \partial_t \psi(t) \varphi \di x  \di t.
      \end{equation*}
  The  arbitrariness of $\varphi$ implies that the weak time derivative of $ w_{\eps,\nu}$ coincides in the distributional sense with $\sigma$ for a.e.~$t \in I$.  Hence, it remains to \III show that $\sigma\in L^2(I;  \III (H^1 (\Omega))^* \EEE )$. To this end, we consider an element $ \III \tilde \vphi \EEE \in L^2(I;H^1(\Omega))$ of the dual satisfying $\Vert \III \tilde \vphi \EEE \Vert_{L^2(I;H^1(\Omega))} \leq 1$.
  By 
  \eqref{spectrum_bound_K}, \eqref{hcm}, \eqref{pos_det}, Hölder's inequality,  and trace estimates
 we find that 
  \begin{align*}
\int_I \langle \sigma(t),\III \tilde \vphi \MMMMM (t) \III  \EEE \rangle \III \di t \EEE  &\leq C \Vert  \nabla \theta_{\eps,\nu} \Vert_{L^2(I\times \Omega)}\Vert   \nabla \III \tilde \vphi \EEE \Vert_{L^2(I\times \Omega)} + C \left( \int_I \int_\Omega  \xi_{\nu, \alpha}^{\rm reg}(\nabla y_{\eps,\nu}, \partial_t \nabla y_{\eps,\nu}, \theta_{\eps,\nu})^2 \di x \di t \right)^{1/2} \hspace{-0.1cm} \Vert \III \tilde \vphi \EEE \Vert_{L^2(I \times \Omega)}
   \\& \quad+ C  \left( \Vert \pl_F \cplpot(\nabla y_{\eps,\nu}, \theta_{\eps,\nu}) : \partial_t \nabla y_{\eps,\nu} \Vert_{L^2(I \times \Omega)}  +   \Vert (\theta_{\flat,\eps} - \theta_{\eps,\nu}) \Vert_{L^2(I \times \Gamma)} \right) \Vert \III \tilde \vphi \EEE \Vert_{L^2(I;H^1(\Omega ))}.
  \end{align*}
By \eqref{pos_det},  \eqref{bound:dissipationepsnu},    \MMMMM   \eqref{very new}, \III   \eqref{avoidKorn},  the fact that $\theta_{\eps,\nu} \in L^2(I; H^1(\Omega) )$, and the regularity of $\theta_{\flat,\eps}$ (see \eqref{def:externalforces}) we thus get $\partial_t w_{\eps,\nu} \in L^2(I; (H^1(\Omega))^* )$.
%
%
 
 \MMMMM
Next, we show $\theta_{\eps,\nu} \in C(I; L^2(\Omega) )$. By \eqref{toten_bound_scheme} and
\ref{W_lower_bound} we get  $(\det \nabla y_{\eps,\nu})^{-1}  \in L^\infty(I; L^q(\Omega;(0,\infty))) $. Therefore, the solution $(y_{\eps,\nu}, \theta_{\eps,\nu})$ lies in the set $ \mathcal{S}_{\rm chain} $ defined at the beginning of Section \ref{sec:chainrule}. Then,  Lemma~\ref{lem:deriphi}(iii) below yields the desired regularity of $\theta$. We also get $w_{\eps,\nu}\in C(I; L^2(\Omega) )$.

Eventually, we derive the formulation  \eqref{weak_limit_heat_equation_nu}.  \III Using test functions $ \III \hat \varphi \EEE \in C^\infty(  I \times \overline{\Omega}   )$ with $ \III \hat \varphi \EEE(T) =0$, the integration by parts $\int_I \III \int_\Omega w_{\eps,\nu}  \partial_t  \III \hat \varphi \EEE \di x \EEE \di t = \III - \EEE \int_I \langle \partial_t w_{\eps,\nu} ,  \III \hat \varphi \EEE\rangle \di t -  \int_\Omega w_{\eps,\nu}(0)  \III \hat \varphi \EEE(0)\di x  $,  the weak formulation \eqref{weak_limit_heat_equation_eps}, and the fact that $W^{\rm in}(F,\cdot)$ is monotonously  increasing (see \eqref{inten_mon}) we also find $\theta(0) = \theta_{0,\eps}$. The latter integration by parts also implies that \eqref{weak_limit_heat_equation_nu} holds for $\varphi \in C^\infty(  I \times \overline{\Omega}   )$ with $\varphi(T) =0$. A standard density argument shows that \eqref{weak_limit_heat_equation_nu}  also holds for test functions $\varphi \in L^2(I; H^1(\Omega))$.
This \MMMMM completes (i).

\MMMMM Next, we address (ii).  \LLL Formally, one can derive \eqref{energybalanceregularized}  by testing the mechanical equation \eqref{weak_limit_mechanical_equation_nu} with $\partial_t \yepsnu$. However,   there is no control on $\partial_t \nabla^2 \yepsnu$ for weak solutions in the present setting.
For the rigorous argument, one needs to apply a chain rule, as discussed in \cite[Equation~(5.9)]{MielkeRoubicek2020}, yielding $t \mapsto \mathcal{M}(\yepsnu(t)) \in W^{1,1}(\lll I \EEE)$ and \eqref{energybalanceregularized}. \MMMMM Concerning (iii), \EEE
  \cite[Theorem~3.13 and Lemma~3.1]{BFK} provide the bounds \eqref{toten_bound_scheme}--\eqref{bound:dissipationepsnu} in a time-discrete setting. By lower semicontinuity of norms, the bounds are preserved in the limiting passage. Here,   the crucial observation \III is that the regularized dissipation rate satisfies $\xiregnu \leq \xi$, allowing the heat source to be bounded by an integrable function that does not depend on the regularization $\nu$ and the \MMMMM parameter $\Lambda$.   This ensures that the bounds can be chosen uniformly in $\nu$. \III Similarly, the bounds are independent of $\eps$ as the data in \eqref{def:externalforces} and \eqref{linearization_initial_conditions} can be bounded uniformly for $\eps \in (0,1]$. \MMMMM Eventually,   \EEE for the proof of the limiting passage $\nu \to 0$ in (iv), one argues along the lines of \cite[Section~6]{MielkeRoubicek2020}. 
\end{proof}

\section{Chain rule}\label{sec:chainrule}
\MMMMM In this section we state and prove \EEE a chain rule \MMMMM for \EEE weak solutions $(y, \theta)$ in the sense of Definition~\ref{def:weak_solutions_regularized}. 
 \ZZZ For convenience, we introduce the space  
\begin{align*}
 \mathcal{S}_{\rm chain}  &  \defas \Big\{ (y,\theta) \colon  y    \in   L^\infty(I;W^{2,p}(\Omega;\R^d)) \cap H^1(I; H^1(\Omega;\R^d) ),  \ \rb (\det \nabla y)^{-1} \ZZZ \AAA \in L^\infty\big(I; L^q(\Omega;(0,\infty))\big), \EEE \\   & \quad\qquad\qquad\qquad\qquad\qquad\qquad\qquad \qquad\theta \in L^2(I;H^1_{\rufb +}\rufb(\Omega)\EEE), \ w \defas W^{\rm in}(\nabla y,\theta)  \in    H^1(I;(H^{1}(\Omega))^*) \Big\}, 
\end{align*}
where  $q$ is \III defined \EEE in \ref{W_lower_bound}.
\MMM  Recall that $\langle \cdot , \cdot \rangle $ denotes the  dual pairing \rb between  $H^1(\Omega)$ and $(H^1(\Omega))^*$.

\begin{theorem}[Chain rule]\label{thm:chainrule}
 Let $(y, \theta) \MMM \in \ZZZ \mathcal{S}_{\rm chain} \EEE$, \III  $\lambda \in C^{1}(I)$, \EEE and assume that \AAA \ref{C_third_order_bounds}--\ref{C_Wint_regularity} hold. \III Then,  we have that \EEE $t \mapsto \int_\Omega ((\lambda - \theta)_+(t))^2 \di x$ lies in $W^{1,1}(I)$ and for a.e.\ $t \in I$ it holds \ZZZ that   \EEE  
 \begin{align}\label{final chain rule}
 & \frac{\rm d}{{\rm d}t} \frac{1}{2} \int_\Omega ((\lambda - \theta)_+)^2 \di x   \EEE  \\
 &\quad=  \int_\Omega  (\lambda  - \theta )_+ \big( \partial_t \lambda   + \rb c_V(\nabla y, \theta)^{-1} \ee \partial_F W^{\rm in}(\nabla y, \theta) : \partial_t \nabla y \big) \di x  -   \rb \langle \partial_t w, (\lambda  - \theta )_+ c_V(\nabla y, \theta)^{-1} \rangle. \notag
  \end{align} 
\end{theorem}

\MMM

\begin{remark} \label{rem:consequencechainrule} \AAA We proceed with some comments on the chain rule. \EEE
 
{\normalfont
\begin{itemize}

\ZZZ
\item[(i)] A weak solution $(y, \theta)$ in the sense of Definition \ref{def:weak_solutions_regularized} lies in the space  
$\mathcal{S}_{\rm chain}$. In particular,
\eqref{toten_bound_scheme} and
\ref{W_lower_bound} imply that $(\det \nabla y)^{-1} \AAA \in L^\infty(I; L^q(\Omega;(0,\infty)))\EEE$.  
\item[(ii)] In the proof, we particularly show that $(\lambda  - \theta )_+  c_V(\nabla y, \theta)^{-1} \in   L^2   (I;H^1(\Omega))$. This along with \lll the regularity of $y$, \rb $\theta$, and $w$, \lll \eqref{inten_mon}, and \rb $\partial_F W^{\rm in}(\nabla y, \theta) \in L^\infty(I \times \Omega; \R^{d \times d})$ (see \eqref{est:coupl}) \lll guarantees \EEE that the right-hand side of \eqref{final chain rule} lies in $L^1(I)$.  \EEE 

%

\end{itemize}
}
\end{remark}

\MMM 
The proof of Theorem \ref{thm:chainrule} relies on regularization of the positive part $(\cdot)_+$. To this end, \EEE given $\beta > 0$ we define the function $\phi_\beta \colon \R \to \R_+$ through
  \begin{align}
  \phi_\beta (s) \defas \begin{cases}
  (s^4 + \beta^4)^{1/4} - \beta & \text{\MMM  if } s > 0, \\
  0 & \text{else}. 
  \end{cases}\label{def:phi_beta}
  \end{align}
This function has the following properties.

\begin{lemma}[Properties of $\phi_\beta$]\label{lem:propertyphibeta}
We have $\phi_\beta \in C^3(\R)$ and $\phi_\beta > 0$, $\phi_\beta' > 0$, and $\phi_\beta'' > 0$ in $(0, \infty)$.
Moreover, as $\beta \searrow 0$, the sequences of functions $(\phi_\beta)_\beta$ and $(\phi_\beta')_\beta$ are \MMM increasing \EEE with pointwise limits $(\cdot)_+$ and $ {\lenni \indic \EEE }_{(0, \infty)}$, respectively. Finally, for all $s \in \R$ \MMM it holds that \EEE 
\begin{align}\label{eq: the good inequality}
\MMM \phi_\beta(s) \leq s^+, \EEE \qquad 
\phi_\beta(s) \leq \phi_\beta'(s) s \leq 4 \phi_\beta(s), \qquad \text{ and } \qquad \phi_\beta''(s) s \leq 3 \phi_\beta'(s).
\end{align}
\end{lemma}

  \MMM
In order to formulate an auxiliary chain rule for $\phi_\beta$, we need to control the analog of $(\lambda  - \theta )_+  c_V(\nabla y, \theta)^{-1}$ in the regularized framework. \EEE   
 
\begin{lemma}\label{lemma: phii}
\MMM  Let $\beta>0$, $\lambda \in \AAA C^1\EEE(I)$, and $(y,\theta) \in \ZZZ \mathcal{S}_{\rm chain} \EEE$. \EEE   Then, the function
\begin{equation*}
  \varphi_\beta\defas \frac{\phi_\beta(\lambda - \theta) \phi_\beta'(\lambda - \theta)}{c_V(\nabla y, \theta)}
\end{equation*}
lies in  $L^2(I;H^1(\Omega))$ \MMM and satisfies the bound  \III
\begin{align}\label{eq: H1bound0}
\Vert   \varphi_\beta \Vert_{L^2(I;H^1(\Omega))}         
\III & \leq C \Vert \lambda \Vert_{L^\infty(I)} \left( \III 1+  \Vert \nabla y \Vert_{L^\infty(I)} \EEE \right) \left( \Vert \nabla \theta \Vert_{L^2(I\times \Omega)} + \Vert \lambda \Vert_{L^\infty(I)}   \Vert \nabla^2 y \Vert_{L^2(I\times \Omega)}    \right) 
\notag \\& \qquad + C \left(  \Vert \nabla \theta \Vert_{L^2(I\times \Omega)}+ \Vert \lambda \Vert_{L^\infty(I)} \right) .
\end{align} 
 for some universal $C>0$. 
\end{lemma}

For technical reasons, \MMM the chain rule contains both the temperature  $\theta$  and the internal energy $w$, although $w$ can be expressed by $\theta$ and $y$ in terms of $w =  W^{\rm in }(\nabla y, \theta)$. \EEE In this regard, it will turn \MMM out \EEE to be useful to introduce the inverse function of $\inten$ with respect to the $w$-variable, \MMM namely \EEE
\begin{equation}\label{def_Phi}
\Psi(F, w) \defas W^{\rm in}(F, \cdot)^{-1}(w) \quad \text{\ZZZ for \EEE any } F \in GL^+(d) \text{ and } w \geq 0,
\end{equation}
 where the inverse of $\inten(F, \cdot)$ exists for all $F \in GL^+(d)$ due to \eqref{inten_mon}. \MMM In particular,   for all $F \in GL^+(d)$ and $w \geq 0$ we have \EEE
 \begin{equation}\label{Phi_identity}
  \inten(F, \Psi(F, w)) = w.
 \end{equation}

\begin{lemma}[Properties of $\Psi$; regularity of $\theta$ and $w$]\label{lem:deriphi} \hfill\newline
\rb \lll \emph{(i)}  \MMM The function $\Psi$ defined in \eqref{def_Phi}   is $C^3$ on $GL^+(d) \times \R_+$. In particular,  for all $F \in GL^+(d)$ and $w \geq 0$, it holds that  \EEE
\begin{align}
\partial_w \Psi(F, w) = \frac{1}{c_V(F, \theta)}, \quad \quad 
 \partial_F \Psi(F,w) = - \frac{\partial_F W^{\rm in} (F, \theta )}{c_V(F, \theta)} \asdf, \EEE   \label{deri1}
\end{align}
where we shortly \AAA write \EEE $\theta$ for $\Psi(F, w)$.

\rb \noindent \lll \emph{(ii)}  \MMM Let \ZZZ $y \in L^\infty(I;W^{2,p}(\Omega;\R^d))$   and set $w  = W^{\rm in}(\nabla y,\theta)$. \EEE Then,  we have $w \in L^2(I;H^1(\Omega))$ if and only if  $\theta \in L^2(I;H^1(\Omega))$, and there exists $C>0$ such that    
\begin{align}\label{what an estimate}
  C^{-1} \Vert \nabla  w \Vert_{L^2(I \times \Omega)} - C\big(1+ \ZZZ \Vert   y\Vert^2_{L^{\infty}(I; W^{2,p}(\Omega))} \EEE \big)   \le  \Vert \nabla  \theta \Vert_{L^2(I \times \Omega)}   \le  C \Vert \nabla  w \Vert_{L^2(I \times \Omega)} + C\big(1+ \ZZZ \Vert   y\Vert^2_{L^{\infty}(I; W^{2,p}(\Omega))} \EEE  \big).  
  \end{align}
  

\rb \noindent \lll \emph{(iii)}  \ZZZ Let  $(y,\theta)\in \mathcal{S}_{\rm chain} $.
Then, $w \in  C(I; L^2(\Omega))$  and $\theta \in C(I; L^2(\Omega))$.   \EEE 
\end{lemma}

\MMM  We defer the \MMMMM proofs \EEE of the three lemmas to Subsection \ref{aux: lemma} below. We now formulate a regularized \ZZZ chain \EEE rule\rufb, \EEE where compared to  Theorem~\ref{thm:chainrule} \EEE the positive part $(\cdot)_+$ is replaced by $\phi_\beta$ \MMMMM defined in \EEE \eqref{def:phi_beta}.

\begin{proposition}[Chain rule \MMM for \EEE  regularized positive part]\label{lem:chainrulepos}  
\rb Given $\beta > 0$, let $\phi_\beta$ be as in \eqref{def:phi_beta}. \rb Moreover, \MMM let $(y, \theta) \in \ZZZ \mathcal{S}_{\rm chain} \EEE$,  \III $\lambda \in   C^1  (I)$, \EEE  and \EEE assume that   \AAA \ref{C_third_order_bounds}--\ref{C_Wint_regularity} hold. \EEE
Then, \III it holds that \EEE
\begin{align}\label{formula:chainrulereg}
 &\frac{\di}{\di t} \frac{1}{2} \int_\Omega (\phi_\beta(\lambda  - \theta))^2 \di x  \\
 &\quad=\int_\Omega \phi_\beta(\lambda  - \theta ) \phi_\beta'(\lambda  - \theta) \big( \partial_t \lambda    + \rb c_V(\nabla y, \theta)^{-1} \ee \partial_F W^{\rm in}(\nabla y, \theta) : \partial_t \nabla y \big) \di x -   \rb\Big\langle \partial_t w, \frac{\phi_\beta(\lambda  - \theta ) \phi_\beta'(\lambda  - \theta)}{c_V(\nabla y, \theta)} \Big\rangle \ee \notag
\end{align}
  for a.e.~$t \in I$.
\end{proposition}

\MMM Observe that Lemma \ref{lemma: phii}  \AAA  guarantees \EEE that the last term \EEE of \eqref{formula:chainrulereg} lies in $L^1(I)$. The proof of this auxiliary chain rule will be given below in Subsection \ref{sec: aux chr}. We first show that  Proposition \ref{lem:chainrulepos}   implies Theorem \ref{thm:chainrule}.

\begin{proof}[Proof of Theorem~\ref{thm:chainrule}]
We can employ Proposition~\ref{lem:chainrulepos}   and integrate the resulting equation over $[t_1, t_2]$ for general $0 \leq t_1 \leq t_2 \leq T$ yielding 
\begin{align}\label{formula:chainrulemaxproof}  
&\int_{t_1}^{t_2} \int_\Omega \phi_\beta(\lambda  - \theta ) \phi_\beta' (\lambda  - \theta ) \rb \big( \ee \partial_t  \lambda + \rb c_V(\nabla y, \theta)^{-1} \ee \partial_F W^{\rm in}(\nabla y, \theta) : \partial_t \nabla y \rb \big) \ee \di x \,  - \rb\Big\langle \partial_t w, \frac{\phi_\beta (\lambda  - \theta ) \phi_\beta '(\lambda  - \theta)}{c_V(\nabla y, \theta)}\Big\rangle \ee \di t  \notag \\
  &\qquad = \frac{1}{2} \int_\Omega \phi_\beta(\lambda(t_2)  - \theta(t_2) )^2 \di x -   \frac{1}{2} \int_\Omega \phi_\beta(\lambda(t_1)  - \theta(t_1) )^2 \di x  .
\end{align}
\MMMMM Here, we also used that   $\theta \in C(I; L^2(\Omega) )$ by  Lemma~\ref{lem:deriphi}(iii). \AAA Our goal is to show that in the limit $\beta \to 0$ it holds that \EEE  
\begin{align}\label{chainrule}
 &  \int_{t_1}^{t_2} \int_\Omega  (\lambda  - \theta )_+\rb\big(\ee \partial_t \lambda   + \rb c_V(\nabla y, \theta)^{-1} \ee \partial_F W^{\rm in}(\nabla y, \theta) : \partial_t \nabla y \rb\big)\ee \di x \, \di t -  \int_{t_1}^{t_2} \rb\Big\langle \partial_t w, (\lambda  - \theta )_+  c_V(\nabla y, \theta)^{-1} \Big\rangle \ee \di t \notag \\
  &\qquad = \frac{1}{2} \int_\Omega ((\lambda - \theta)_+(t_2))^2 \di x - \frac{1}{2} \int_\Omega ((\lambda-\theta)_+(t_1))^2 \di x  .
  \end{align} 
\EEE Due to Lemma~\ref{lem:propertyphibeta}, \AAA as $\beta  \searrow 0$, \EEE $\ZZZ (\phi_\beta)_\beta \EEE $ is \MMM an increasing and nonnegative \EEE sequence \rb converging \rufb pointwise \lll to \EEE $(\cdot)_+$. \ee Moreover, \MMM as \ZZZ $\theta \in C(I; L^2(\Omega) )$ by \ZZZ Lemma~\ref{lem:deriphi}(iii), \EEE we get \EEE $\theta(t_i) \in \ZZZ L^2(\Omega) \EEE$ for $i =1,2$. Thus, we have $ \phi_\beta(\lambda(t_i)-\theta(t_i))\nearrow (\lambda(t_i)-\theta(t_i))_+$ pointwise for a.e.~$x\in \Omega$, and \MMM then \EEE the \ZZZ right-hand \EEE side of \eqref{formula:chainrulemaxproof} converges to the \ZZZ right-hand \EEE side of \eqref{chainrule} by the monotone convergence theorem.

  For the convergence of the \ZZZ left-hand side, \EEE \lll we first show \EEE that  $\varphi_\beta  = c_V(\nabla y, \theta)^{-1} \phi_\beta(\lambda - \theta) \phi_\beta'(\lambda-\theta) $ \AAA defined in Lemma \ref{lemma: phii} \EEE converges weakly in $L^2(I; H^1(\Omega))$ to  $c_V(\nabla y, \theta)^{-1}(\lambda - \theta)_+$ as $\beta\to 0$.  \MMM By Lemma~\ref{lem:propertyphibeta}\ZZZ, \EEE we find that   $(\varphi_\beta)_\beta$ converges pointwise a.e.\ in $ I \times \Omega$   to $c_V(\nabla y, \theta)^{-1}(\lambda - \theta)_+$. Moreover, \III using \EEE Lemma~\ref{lemma: phii} and recalling the regularity of $(\lambda, \MMMMM y,\theta)$\III, \EEE we get that  $( \varphi_\beta)_\beta $ is bounded in $L^2(I; H^1(\Omega))$. This  shows $\varphi_\beta \rightharpoonup c_V(\nabla y, \theta)^{-1}(\lambda - \theta)_+$ weakly in $L^2(I;H^1(\Omega))$, as desired.  \MMMMM In  particular, we have  $c_V(\nabla y, \theta)^{-1}(\lambda - \theta)_+ \in L^2(I;H^1(\Omega))$, i.e.,  Remark~\ref{rem:consequencechainrule}(ii) holds. \EEE   \lll The remaining terms on the left-hand side converge due to Lemma~\ref{lem:propertyphibeta}, \eqref{est:coupl}, \eqref{inten_mon}, the regularity of $y$, and the dominated convergence theorem. \EEE
  Summarizing, we have shown that  \EEE    the \ZZZ left-hand \EEE side of \eqref{formula:chainrulemaxproof} converges to the \ZZZ left-hand \EEE side of \eqref{chainrule}.
  
  \MMM Eventually, as \eqref{chainrule} is satisfied for arbitrary $t_1,t_2 \in I$, we  conclude    that  $t \mapsto \int_\Omega ((\lambda - \theta)_+(t))^2 \di x$ lies in $W^{1,1}(I)$ and for a.e.\ $t \in I$ the chain rule \eqref{final chain rule} holds. \EEE  
\end{proof}

\MMM 

\subsection{Proof of the auxiliary chain rule}\label{sec: aux chr}
 
This subsection is devoted to the proof of the auxiliary chain rule stated in Proposition \ref{lem:chainrulepos}. \lll We \EEE will first prove the result under a \III higher \EEE regularity assumption on $w$ \ZZZ and $y$, \EEE and pass to the general case by approximation at the end of this subsection.  
\ZZZ
More precisely, we will first consider functions $(y,\theta) \in \mathcal{S}_{\rm chain}$ such that $w \rb \defas \inten(\nabla y, \theta) \ZZZ \in \III H^1 \EEE(I; H^1(\Omega))$ and $y \in H^1(I;H^{k_0}(\Omega;\R^d))$ for  some  $k_0 \in \N$  with  $k_0 > 1 + \frac{d}{2}$.
Here, we note that the choice of $k_0$ and Morrey's inequality ensure that
\begin{align}\label{embed}
 \III C(\overline{\Omega})  \EEE \subset\subset   H^{k_0 - 1}  (\Omega).
\end{align}

\begin{proposition}[Auxiliary chain rule for more regular \MMMMM $y$ and \EEE $w$]\label{prop:chainrulepos-new}

Let $\phi_\beta$ be as in \eqref{def:phi_beta} for arbitrary $\beta > 0$.  Let $(y, \theta) \in \ZZZ \mathcal{S}_{\rm chain} \EEE$, \III $\lambda \in C^1(I)$, \EEE  and  assume that  \AAA \ref{C_third_order_bounds}--\ref{C_Wint_regularity} hold. \EEE  Assume in addition that  \III $w \in \III H^1 \EEE(I;H^1(\Omega))  $ \ZZZ and $y \in H^1(I;H^{k_0}(\Omega;\R^d))$. \EEE  Then, the statement of Proposition \ref{lem:chainrulepos} holds \III with the dual pairing $\langle \cdot, \cdot \rangle$ in \eqref{formula:chainrulereg} replaced by the scalar product in $L^2(\Omega)$.
\end{proposition}

A key ingredient for the proof is  a  chain rule for locally semiconvex functionals, see  \EEE  \cite[Proposition~3.6]{MielkeRoubicek2020} or also \cite[Proposition~2.4]{MielkeRossiSavare}. Before stating this result, we briefly recall the definition of local semiconvexity and the definition of the Fréchet subdifferential.
In this regard, let $X$ be a reflexive and separable Banach space.
A functional $\mathcal{J} \colon X \to \R \cup \{ + \infty \}$ is called locally semiconvex if for all $z \in X$ with $\mathcal{J}(z) < + \infty$ \MMM there \EEE exist $\Lambda(z) \geq 0$ \rufb and $r(z) > 0$ \ee such that the restriction of $\mathcal{J}$ \rufb to the ball $B_{r(z)}(z) \defas \{ \hat z \in X \colon \Vert \hat z - z \Vert_X \leq r(z) \}$ \ee is \MMM $\Lambda(z)$-semiconvex, \EEE i.e.,
\begin{align}\label{def:lambdasemiconv}
\mathcal{J}((1-s) z_0 + sz_1) \leq (1-s) \mathcal{J}(z_0) + s \mathcal{J}(z_1) + \Lambda(z) \frac{s-s^2}{2} \Vert z_1 - z_0 \III \Vert_X^2 \EEE
\end{align}
for all $ z_0, z_1 \in B_{\MMMMM r(z)}(z)$  and all $s \in [0,1]$. \rufb Furthermore, t\ee he Fréchet subdifferential is defined by
\begin{align*}
\overline{\partial} \mathcal{J}(z) = \left\{ \Theta \in X^* : \mathcal{J}(\tilde z) \geq \mathcal{J}(z) + \langle \Theta, \tilde z - z \ZZZ \rangle_X \EEE -  \frac{1}{2}\EEE   \Lambda (z) \Vert \tilde z - z \Vert_X^2 \text{ for } \tilde z \in B_{r( z) } (z) \right\},
\end{align*}
where $X^*$ denotes the dual space of $X$ \rb and $\langle \cdot, \cdot \rangle_X$ stands for the dual pairing between \MMMMM $X$ and $X^*$. \EEE
For \MMMMM convenience, \EEE  we \MMM formulate \EEE the result \cite[Proposition~3.6]{MielkeRoubicek2020} in the special case $q = 2$:

\begin{proposition}[Chain rule for locally semiconvex functionals]\label{lem:mielkechainrule}
Consider a separable\ZZZ, \EEE reflexive Banach space $X$ and let $\mathcal{J}\colon X \to \R \cup \{+\infty\}$ be a lower semicontinuous and locally semiconvex functional. If $z \in H^1(I;X)$ and $\Theta \in L^2(I; X^*)$ satisfy
\begin{align}
&\sup\EEE \{ \mathcal{J}(z(t)) \colon t \in I \} < + \infty  \text{ and } \label{bed1} \\
&\Theta(t) \in \overline{\partial} \mathcal{J}(z(t)) \text{ for a.e.~} t \in I,  \label{bed2}
\end{align}
then
\begin{align}\label{eq: the chain rule}
\MMM t \mapsto \mathcal{J}(z(t))  \text{ lies in } \EEE   W^{1,1}(I) \quad \text{ and } \quad \frac{\di}{\di t} \mathcal{J} (z(t) ) = \langle \Theta(t),  \AAA \partial_t \EEE z (t) \rb \rangle_{\AAA X} \EEE \text{ for a.e.~} t \in I.
\end{align}
\end{proposition}  
\lll In view of \eqref{def_Phi}, formula \EEE \eqref{formula:chainrulereg} contains the function $ \Phi_\beta \colon \lll \R \EEE \times GL^+(d) \times \R_+ \to \R$ defined  by
\begin{equation}\label{shorthand0}
\Phi_\beta(\lambda, F, w) \defas  \frac{1}{2}\phi_\beta\big(\lambda - \Psi(F, w)\big)^2,
\end{equation}
as well as its $\lambda$-derivative, which we denote by 
\begin{equation}\label{shorthand}
\tilde \Phi_\beta(\lambda, F, w) \defas \phi_\beta\big(\lambda - \Psi(F, w)\big) \phi'_\beta\big(\lambda - \Psi(F, w)\big).
\end{equation} 
Their properties are summarized in the following lemma.

\begin{lemma}[\rb Properties of \III $\Phi_\beta$\EEE]\label{Phi-lemma} \III The function \EEE  $\Phi_\beta$ is $C^3$ on $ \lll \R \EEE \times GL^+(d) \times \R_+$, and  has the first-order partial derivatives
\begin{align}\label{ttt0.}
\partial_\lambda   \Phi_\beta = \tilde \Phi_\beta,  \quad 
\partial_F   \Phi_\beta = - \tilde \Phi_\beta \partial_F \Psi, \quad 
\partial_w   \Phi_\beta= - \tilde \Phi_\beta \partial_w \Psi.
\end{align}
\end{lemma}

\begin{proof} 
By an elementary \ZZZ computation, \EEE we obtain \eqref{ttt0.}. In a similar fashion, we get that the second-order \MMMMM and third-order \EEE partial derivatives feature products of  $\phi_\beta$ and its derivatives up to third order, as well as $\Psi$ and its derivatives up to \MMMMM third \EEE order. Since $\phi_\beta \in C^3(\R)$ and $\Psi \in C^3(GL^+(d) \times \R_+)$ by Lemma \ref{lem:propertyphibeta} and Lemma~\ref{lem:deriphi}(i), respectively, the statement \III follows.  
\end{proof}
\EEE

We proceed with the proof of the auxiliary chain rule, under the  additional  assumptions   $w \in \III H^1 \EEE(I;H^1(\Omega))   $ \AAA and   $y \in H^1(I;H^{k_0}(\Omega;\R^d))$. \EEE

\begin{proof}[Proof of Proposition~\ref{prop:chainrulepos-new}]
 The proof is divided into \III four \EEE steps. We first show that the chain rule holds in a small time interval. To this end, we define suitable $X$ and $\mathcal J$ from Proposition~\ref{lem:mielkechainrule} in our present setting, and start by showing that condition \eqref{bed1} is satisfied (Step 1). \III In Step 2  we compute the subdifferential, address  the local semiconvexity, \III and show  \eqref{bed2}. In \EEE Step \III 3  \EEE we prove \EEE the lower semicontinuity of $\mathcal{J}$, \III and \EEE in Step \III 4 \EEE we pass to a global version of the chain rule on \MMM the time interval \EEE  $I$.  For notational convenience, we drop the subscript $\beta$, and write \III  $\Phi$ and $\tilde{\Phi}$ in place of    $\Phi_\beta$ and $\tilde{\Phi}_\beta$, respectively. \EEE

\EEE

\emph{Step 1 (Definition of $X$ and $\mathcal J$,  \ZZZ and property \EEE \eqref{bed1}):}  
\MMM Consider $\lambda$, $y$,  and $w$ as in the statement, i.e.,  $(y,\theta) \in \ZZZ \mathcal{S}_{\rm chain} \EEE$,  \III $w \in \III H^1 \EEE (I;H^1(\Omega;   \R_+))   $, \ZZZ 
$y \in H^1(I;H^{k_0}(\Omega;\R^d))$,  \III and $\lambda \in \AAA C^1 \EEE (I)$. \III 
 Moreover, let \EEE
\begin{equation}\label{Linftytemp}
C_\infty \defas \III \sup_{t \in I}  \big( \vert \lambda(t) \vert + \Vert w(t) \Vert_{ \III L^2(\Omega) } +  \Vert y(t) \Vert_{ \III H^{1}(\Omega)} \big) \EEE   < +\infty,
\end{equation} 
where \III in \eqref{Linftytemp} we can write sup in place of  $\esssup$ since $w \in C(I; H^1(\Omega) )$ and $y \in C(I;H^{k_0}(\Omega;\R^d) )$. \EEE
\ZZZ Then, the regularity of $y$, and $(\det \nabla y )^{-1} \in \AAA L^\infty(I; L^q(\Omega;(0,\infty))) \EEE$ for $q \geq \frac{pd}{p-d}$ along with \cite[Theorem~3.1]{MielkeRoubicek2020} (see also \cite[Theorem~3.1]{HealeyKroemer09Injective}) \EEE show  that  
\begin{equation}\label{Linftytemp2}
c_\infty \defas \inf_{I \times \Omega }  \det \nabla y  >0,
\end{equation}
where  we can write inf in place of \ZZZ $\essinf$ \EEE since $\nabla y \in \rufb C(I \times \overline{\Omega}; \R^{d\times d})\ee$, \III see \eqref{embed}. \EEE In particular, \III we find \ZZZ a constant \EEE $C >0$ depending only on $C_\infty$ and $\Omega$ such that
\begin{align}\label{pointwise}
\text{$\nabla y \in GL^+_{2c_\infty}(d) \cap B^{d \times d}_{C}$  on $I \times \Omega$},
\end{align}
 \ZZZ where for $c>0$ we let $GL^+_{c}(d) \defas \lbrace F \in GL^+(d) \colon \, \det F \ge c/2\rbrace$, and
$B^{d \times d}_{C} \subset \R^{d \times d}$ denotes the ball centered at zero with radius $C$. \EEE \ZZZ The continuity of the determinant implies that there \EEE exists $\eps >0$ such that  
\begin{align}\label{GL bound}
\text{$F \in GL^+_{c_\infty}(d)$ for all $F \in \R^{d \times d}$ with ${\rm dist}\big(F,  GL^+_{2c_\infty}(d) \cap B^{d \times d}_{C}\big) \le \eps$.}
\end{align}
 As $\nabla y \in \rufb C(I \times \overline{\Omega}; \R^{d\times d} )$, we find $\delta >0$ such that, for each fixed $\tau \in I$, we have  
\begin{align}\label{auch noch} 
\text{$\Vert \nabla y(t) - \nabla y(\tau) \Vert_{L^\infty(\Omega)} \leq \eps$ for all $t \in I_\tau \defas I \cap (\tau - \delta, \tau + \delta)$.}
 \end{align} \EEE 
Our first goal is to show that for fixed $\tau \in I$   \MMM the chain rule \eqref{formula:chainrulereg} holds for a.e.~$t \in I_{\tau}$, \III  where the dual pairing $\langle \cdot, \cdot \rangle$ in \eqref{formula:chainrulereg} is replaced by the scalar product \III in $L^2(\Omega)$\ee.  \EEE The global version \ZZZ is \EEE deferred to Step \lll 4. \EEE We want to employ Proposition \ref{lem:mielkechainrule}: we choose as $X$ the separable, reflexive Banach space \III $\R \times H^{1}(\Omega;\R^d) \times \III L^2(\Omega) \EEE$.   Note that \III the assumed regularity on $y$, $w$\rufb, \III and $\lambda$  particularly guarantees $(\lambda,y,w) \in H^1(I;X)$. \EEE We \MMM define \EEE  the functional $\mathcal J \colon X \to \R \NNN \cup \{+\infty\}$ \ZZZ as  
\begin{align}\label{def_mathcal_J}
 \mathcal{J}(\hat \lambda, \hat  y, \hat w)
 &\defas\begin{cases}  \int_\Omega \MMM \Phi( \hat \lambda, \nabla \hat  y, \hat  w)   \EEE \di x &   \text{if} \quad   \III \hat w \ge 0, \   \vert \hat \lambda \vert + \Vert \hat w \Vert_{ \III L^2(\Omega)} + \Vert  \hat y \Vert_{ \III H^{1}( \Omega)}    \leq C_\infty,\\ & \quad \ \,  \Vert \nabla \hat y - \nabla y(\tau) \Vert_{L^\infty(\Omega)} \leq \eps, \EEE\\
  + \infty &  \text{else}\rbb, \ee
  \end{cases}   
\end{align} 
for $(\hat \lambda, \hat y, \hat w) \in X$, \III where $\Phi$ is defined in \eqref{shorthand0}.  In the first case of \eqref{def_mathcal_J}, the integrand $\Phi$ \III  is   well-defined since the choice of $\eps>0$ and  \eqref{GL bound} ensure  that $\nabla \hat y \in\MMM  GL^+_{c_\infty}(d)$  in $\Omega$. \EEE  This specific definition is made for two reasons:

Firstly, \EEE 
for \MMM  $\lambda$, $y$, and $w$ as in the statement,   we see by the definition of $C_\infty$ in \eqref{Linftytemp},  $\phi(s) \leq s^+$ for all $s \in \R$ (see Lemma~\ref{lem:propertyphibeta}), \MMM \eqref{auch noch}, and $\theta = \Psi(\nabla \ZZZ y, \EEE w) \ge 0$    that   
\begin{align}\label{for-bed1}
\sup_{ t \in I_\tau \EEE} \mathcal{J} (\lambda(t), y(t), w(t)) \leq \frac{1}{2} \int_\Omega \III C_\infty \EEE \di x < +\infty,
\end{align}
\MMM where \III the supremum coincides with the essential supremum due to the continuity of $\lambda$.  This shows \eqref{bed1}.  

\MMM Secondly, \III provided $(\hat \lambda, \hat y, \hat w) \in X$ with $\mathcal{J} (\hat \lambda, \hat y, \hat w) <\infty$,  
\III for a.e.~$x\in \Omega$  \lll satisfying \EEE  $\Phi(\hat \lambda, \nabla \hat y(x),\hat w(x))>0$ (or equivalently $\Psi(\nabla \hat y(x), \hat w(x)) < \hat \lambda$)      the value $(\hat \lambda,\nabla \hat y(x),\hat w(x))$  lies in the compact set $K \defas K_\lambda \times K_y \times K_w \subset  \lll \R \EEE \times GL^+(d) \times \R_+ $ \rufb given by \ee
\begin{align}\label{eq: K}
 K_\lambda \defas [ - C_\infty , C_\infty ], \quad \quad  K_y \defas GL^+_{c_\infty}(d) \cap B^{d \times d}_{ \III 2C}, \quad \quad K_w \defas [0,C_0 C_\infty]. 
 \end{align} 
In fact, for $\hat{\lambda}$ this follows by definition, and for $\nabla \hat{y}$ we use   \eqref{pointwise}--\eqref{GL bound}.  Eventually,  \EEE  using \eqref{inten_lipschitz_bounds}, \eqref{Phi_identity}, and  $\Psi(\nabla \hat y(x), \hat w(x)) < \hat \lambda$,   we have \III
\begin{equation*}
 \EEE 0 \le \hat{w}(x) = \inten(\nabla \hat y(x), \Psi(\nabla \hat y(x), \hat w(x))) \leq C_0 \Psi(\nabla \hat y(x), \hat w(x)) \III \leq C_0 \vert \hat\lambda \vert \leq C_0 C_\infty.
\end{equation*}
\III \emph{Step 2 (Semiconvexity of $\mathcal{J}$, property \eqref{bed2}, and subdifferential  of $\mathcal{J}$):}  
 Consider $(\hat \lambda, \hat y, \hat w) \in X$ with $\mathcal{J} (\hat \lambda,  \hat y, \hat w)  < \infty$.  Let $\Xi \colon X \to \R$ be \MMM the linear functional \EEE given by  
\begin{equation}\label{def-Xi}
  \big\langle \Xi, (\tilde\lambda, \tilde y, \tilde w) \big\rangle_{\III X}
  \defas \int_\Omega \tilde \Phi(\hat \lambda, \nabla \hat y, \hat w) \big(
    \tilde \lambda
    - \partial_F \Psi(\nabla \hat y, \hat w) : \nabla \tilde y    - \partial_w \Psi(\nabla \hat y, \hat w) \, \tilde w \big) \di x 
\end{equation}
for any $(\tilde \lambda, \tilde y, \tilde w) \in \III X$ \MMM which corresponds to the pointwise derivative of $\Phi$ under the integral, \ZZZ see \III \eqref{shorthand} and  \eqref{ttt0.}. 
\MMMMM This  functional will be instrumental to  compute  the subdifferential of $\mathcal{J}$.
By  Lemma~\ref{Phi-lemma} we have \EEE $\tilde \Phi\in C^2 ( \lll \R \EEE \times GL^+(d) \times \R_+)$. Combining this fact with \lll the discussion  \MMMMM preceding \EEE \eqref{eq: K} and Lemma~\ref{lem:deriphi}(i) \MMMMM we find \EEE
\begin{equation}\label{Theta} 
  \big|\big\langle \Xi, (\tilde\lambda, \tilde y, \tilde w) \big\rangle_{\MMM X}\big|
  \leq C \left( \vert \tilde \lambda \vert +  \Vert \nabla \tilde y \Vert_{L^1(\Omega)}  + \Vert \tilde w \Vert_{ L^{1}(\Omega)} \right)  \leq C \left( \vert \tilde \lambda \vert +  \Vert \nabla \tilde y \Vert_{L^2(\Omega)}  + \Vert \tilde w \Vert_{ L^{2}(\Omega)} \right),
\end{equation} \III for a suitable constant $C>0$ \MMMMM only depending on $C_\infty$. \III Thus, $\Xi \in X^*$. \EEE The core of this step lies in showing that \MMM there exists $\bar{C}>0$ such that \III   for \EEE any   $(\tilde \lambda, \tilde y, \tilde w) \in \III X$ \III it \EEE holds  \ZZZ that \EEE
\begin{equation}\label{chainrulelambdaconv}
\begin{aligned}
  &\mathcal{J}(\tilde \lambda, \tilde y, \tilde w) -\mathcal{J}(\hat \lambda, \hat y, \hat w)- \big\langle \Xi, (\tilde \lambda - \hat \lambda,   \tilde y -   \hat y, \tilde w - \hat w) \big\rangle_{\III X}  \geq - \MMM \bar{C} \EEE \big(\ZZZ \vert \tilde \lambda - \hat \lambda \vert^2 + \EEE \Vert \tilde w - \hat w \Vert^2_{\III L^{2}(\Omega)} + \Vert \tilde y - \hat y \Vert^2_{\III H^{1}(\Omega)} \big).
\end{aligned}
\end{equation}
  First, \EEE  \eqref{chainrulelambdaconv} is trivially satisfied in the case $\mathcal J(\tilde \lambda, \tilde y, \tilde w) = \infty$. \martin  \MMMMM Therefore, we \EEE can assume that $\mathcal J(\tilde \lambda, \tilde y, \tilde w) < \infty$. \EEE
By \MMM  applying \III the fundamental theorem of calculus  twice, and by recalling \EEE the definition in \eqref{def-Xi}, we see that
\begin{align*}
&\mathcal{J}(\tilde \lambda, \tilde y, \tilde w) - \mathcal{J}(\hat \lambda,  \hat y, \hat w) - \big\langle \Xi, (\tilde \lambda - \hat \lambda,   \tilde y -   \hat y, \tilde w - \hat w) \big\rangle_{\III X} \\
&\quad=    \int_\Omega \int_0^1  \III (1-z) \EEE
   \MMM D^2\Phi \III (\lambda_z, \nabla y_z, w_z) \EEE [(\tilde \lambda - \hat \lambda,   \nabla \tilde y -  \nabla  \hat y, \tilde w - \hat w),(\tilde \lambda - \hat \lambda,   \nabla \tilde y -   \nabla\hat y, \tilde w - \hat w)] \di \III z \EEE \di x,
\end{align*}
where \III $(\lambda_z,\nabla y_z,w_z) \defas (1-z) (\hat \lambda, \nabla \hat y, \hat w) + z(\tilde \lambda, \nabla \tilde y, \tilde w)$   for $z \in [0,1]$. \III By the convexity of the norms in the constraints of \eqref{def_mathcal_J},  \lll we \EEE get that convex combinations satisfy  $\mathcal{J} \III (\lambda_{ z},  y_z, w_z) \EEE < +\infty $.  
\MMM By using  \III Lemma~\ref{Phi-lemma} and the arguments in \eqref{eq: K} above, we \III derive for $z \in [0,1]$ that \EEE $ \III
\Vert D^2\Phi  (\lambda_z, \nabla y_z, w_z) \Vert_{L^\infty(\Omega)}    \le \III C
$,
passing to a possibly larger constant $C$.  \EEE
By Young's inequality, we see that \eqref{chainrulelambdaconv} is satisfied.

\MMMMM We \EEE  now show that  $\mathcal{J}$ is semiconvex in the sense of \eqref{def:lambdasemiconv}. To this end, consider \MMM $(\lambda_0,y_0, w_0), \ (\lambda_1,   y_1, w_1) \in X$. Without restriction, we   assume that $\mathcal{J} ( \lambda_0, y_0, w_0) , \ \mathcal{J} ( \lambda_1, y_1, w_1)  < +\infty$ as otherwise the inequality in \eqref{def:lambdasemiconv} is trivial. For $s\in [0,1]$, we   define  $(\lambda_s,  y_s,w_s) \defas (1-s)( \lambda_0, y_0, w_0) +  s( \lambda_1,   y_1,  w_1)$. \III As above, \EEE we have that $\mathcal{J}(\lambda_s,  y_s,w_s) <+\infty $. Using \eqref{chainrulelambdaconv} for \MMM  $(\hat \lambda,\hat y,\hat w)= (\lambda_s,  y_s,w_s)$ and   $(\tilde \lambda,\tilde  y,\tilde  w) = ( \lambda_i, y_i, w_i) $ for $i=0,1$,   \EEE   an elementary computation leads to  
\begin{align*}
\mathcal{J}(\lambda_s,  y_s, w_s)  & \leq (1-s)\mathcal{J}( \lambda_0, y_0, w_0) + s \mathcal{J}(\lambda_1,  y_1,  w_1)  \\ & \ \ \ + \MMM \bar{C} \EEE  s(1-s) \big( \ZZZ \vert  \lambda_1 -  \lambda_0 \vert^2 + \EEE\Vert  w_1 -  w_0\Vert^2_{ \III L^2(\Omega)} + \Vert  y_1 -  y_0 \Vert_{H^{1}(\Omega)}^2 \big).
\end{align*}
This shows that  $\mathcal{J}$ is \MMM locally \EEE semiconvex.

\MMM 
\III Next, we deduce \eqref{bed2}. \MMM Consider $\lambda$, $y$,  and $w$ as in the statement, and recall that for each $t \in \III I_\tau \EEE$ we have  $(\lambda(t), y(t),w(t)) \in X$ and $\mathcal{J} (\lambda(t),   y(t),  w(t))  < \III + \EEE \infty$ by \eqref{for-bed1}. Then, \MMMMM  \eqref{Theta}  shows \EEE that $\Theta(t)$ defined by \EEE 
\begin{align}\label{Theta-def}
  \big\langle \Theta(t), (\tilde\lambda, \tilde y, \tilde w) \big\rangle_{\MMM X}
  \defas \int_\Omega & \tilde \Phi( \lambda(t), \nabla  y(t), w(t)) \big(
    \tilde \lambda
    - \partial_F \Psi(\nabla  y(t),   w(t)) : \nabla \tilde y  -   \III  \partial_w \Psi(\nabla y(t),  w(t) \ZZZ ) \EEE  \tilde w  \big) \di x  
\end{align}
for $t \in \III I_\tau \EEE$  is an element of $X^*$, and \III \eqref{def-Xi} and \eqref{chainrulelambdaconv} imply that $\Theta$ lies  in \EEE the subdifferential of $\mathcal{J}$ at $(\lambda(t), y(t),w(t))$. This  shows  \eqref{bed2}.  Moreover, $\Theta \in L^2( \III I_\tau;X^*)$ \MMMMM also follows from \eqref{Theta}. \EEE


\MMM

\emph{Step \III 3  \EEE (Lower semicontinuity of $\mathcal{J}$):} \EEE  \III Consider \EEE a sequence  $(\lambda_n, y_n, w_n)_n$ in $X$ converging strongly \MMM in $X$ \EEE \lll to \EEE some $(\hat \lambda,\hat y,\hat w) \in X$. Without loss of generality, we can assume that there exists a subsequence (not relabeled) such that $\mathcal{J}(\lambda_n, y_n, w_n)<+\infty$ for all $n \in \N$.  In particular, this implies \III $ \vert \lambda_n \vert + \Vert w_n \Vert_{L^2(\Omega)} + \Vert y_n \Vert_{H^{1}(\Omega)} \le C_\infty$ and $\Vert \nabla y_n - \nabla y(\tau) \Vert_{L^\infty(\Omega)} \leq \eps$. \EEE Thus, also \III  $ \vert \hat \lambda \vert + \Vert \hat w \Vert_{L^2(\Omega)} +\Vert \hat y \Vert_{H^{1}(\Omega)}\le C_\infty$ and $\Vert \nabla \hat y - \nabla y(\tau) \Vert_{L^\infty(\Omega)} \leq \eps$ by the lower semicontinuity of norms. 
  \III This \EEE shows $\mathcal{J}(\hat \lambda, \hat y, \hat w)<+\infty$ \EEE and we can apply \eqref{chainrulelambdaconv} for  $(\hat \lambda, \hat y, \hat w)$ and  $(\lambda_n, y_n, w_n)$. The fact $(\lambda_n, y_n, w_n) \to (\hat \lambda, \hat y, \hat w)$ in $X$ then shows that $\mathcal{J}$ is lower semicontinuous.\EEE

\emph{Step \III 4 \EEE (Chain rule on $I$):}  
To summarize \MMM the previous steps, \EEE  we have \III verified \EEE all   assumptions of Proposition~\ref{lem:mielkechainrule}. 
\III Thus, \EEE the chain rule in its localized version on $I_{\tau}$ follows from \eqref{eq: the chain rule}, the definition of $\Theta$ in \eqref{Theta-def}, and the formulas in \eqref{deri1}. \MMM Now, it suffices to cover $I$ with a finite number of open intervals of length $2\delta$, i.e., we choose $\tau_1, ..., \tau_m \in I$ for  some $m \in \N$ such that $I \subset \bigcup_{i = 1}^m  I_{\tau_i}$.  \EEE Since the chain rule holds locally on  \MMM  each interval,  it also holds on $I$. \EEE This concludes the proof. 
\end{proof} 
  
%

%

\MMM By an approximation argument we now reduce the proof of Proposition \ref{lem:chainrulepos}   to Proposition \ref{prop:chainrulepos-new}. \EEE

\begin{proof}[Proof of Proposition \ref{lem:chainrulepos}]
 \MMM Let $(y, \theta) \in \ZZZ \mathcal{S}_{\rm chain} \EEE$  and  let  $\lambda \in \AAA C^1 \EEE (I)$. Recall that by assumption and by Lemma~\ref{lem:deriphi}(ii) we have  $w \in L^2(I; H^1(\Omega)) \cap H^1(I;(H^{1}(\Omega))^*)$. \EEE   We   extend $w$ by $0$ on $(-\infty,0)$ and $(T,+\infty)$ and define $w_\eps \defas \eta_\eps \ast w$, where $\eta_\eps \in C^\infty_c(-\eps,\eps)$ denotes a standard mollifier. It is \AAA a \EEE standard matter to check that this mollification satisfies, \AAA as $\eps \to 0$, \EEE
\begin{align}\label{approxchainrule}
w_\eps \to w \quad {\rm in} \ L^2(I; H^1(\Omega)) \qquad {\rm and} \qquad \partial_t w_\eps \to \partial_t w \quad {\rm in} \ \MMM L^2_{\rm loc} \EEE ( \rb (0, \AAA T) \EEE  \EEE; (H^{1}(\Omega))^* ).
\end{align}
In particular, we have $w_\eps \in \III H^1 (I;H^1(\Omega)) \EEE$ for all $\eps >0$. \ZZZ
\rb Furthermore\ZZZ, \AAA an in-space mollification (for each $t \in I$) \EEE provides $y_\eps \in L^\infty(I; W^{2,p}(\Omega;\R^d) \III ) \EEE \cap H^1(I;H^{k_0}(\Omega;\R^d))$ such that
\begin{align}\label{approxchainruley}
y_\eps \to y \quad {\rm in } \ L^{\infty}(I; W^{2,p}(\Omega;\R^d) ) \cap H^1(I; H^1(\Omega;\R^d) ).
\end{align}
Defining \ZZZ $\theta_\eps \defas \Psi(\nabla y_\eps, w_\eps)$, our  goal is to apply Proposition~\ref{prop:chainrulepos-new} for the functions $y_\eps$ and $\theta_\eps$. To this end, we need to show that $(y_\eps,\theta_\eps) \in \mathcal{S}_{\rm chain} $. Due to Lemma~\ref{lem:deriphi}(ii), we have $\theta_\eps \in L^2(I;H^1(\Omega))$.  Then, \AAA as in \eqref{Linftytemp2}, we get \EEE that  $\essinf_{I \times \Omega }  \det \nabla y  >0$. \rb As $p > d$, Sobolev embedding and \ZZZ \eqref{approxchainruley} imply that $\nabla y_\eps \to \nabla y$ in $L^\infty (I; L^\infty(\Omega;\R^{d \times d} ) )$, and thus $\essinf_{I \times \Omega }  \det \nabla y_\eps  >0$ for $\eps>0$ sufficiently small. This yields $(\det \nabla y_\eps )^{-1} \in  L^\infty(I; L^q(\Omega;(0,\infty))) \EEE$. \EEE

Thus, \EEE all assumptions of Proposition~\ref{prop:chainrulepos-new}  \AAA hold, \EEE and \EEE the curve \BBB $(\lambda, \ZZZ y_\eps, \EEE \theta_\eps)$    \AAA satisfies \EEE the identity \eqref{formula:chainrulereg} \III with the dual pairing $\langle \cdot, \cdot \rangle$ replaced by the scalar product  \EEE in $L^2(\Omega)$. Integrating \MMM in time, this shows for every $t_1,t_2 \in \rb (0, T)$ that
\begin{align}\label{regularizedidentity000}
 & \frac{1}{2} \int_\Omega \phi(\lambda(t_2)  - \theta_\eps(t_2) )^2 \di x -   \frac{1}{2} \int_\Omega \phi(\lambda(t_1)  - \theta_\eps(t_1) )^2 \di x  \notag \\
 &\quad= \int_{t_1}^{t_2} \int_\Omega \varphi_\eps \big( \BBB c_V(\nabla \ZZZ y_\eps \EEE, \theta_\eps)    \partial_t  \lambda    +  \partial_F W^{\rm in}(\nabla \ZZZ y_\eps, \EEE \theta_\eps) : \partial_t \nabla \ZZZ y_\eps \EEE \big)  \III- \partial_t w_\eps \varphi_\eps \di x \, \di t, \EEE
\end{align}
where \MMM for convenience we \rb wrote \ee $\phi$ in place of $\phi_\beta$ \TTT and \EEE set $\varphi_\eps \defas \frac{\phi(\lambda - \theta_\eps)\phi'(\lambda - \theta_\eps)}{c_V(\nabla \ZZZ y_\eps \EEE, \theta_\eps)}$ \MMM for brevity. \EEE  

  By the continuity of $\Psi$   (see Lemma \ref{lem:deriphi}(i)), \lll \eqref{approxchainrule}, and \eqref{approxchainruley} \EEE it holds that \ZZZ $\theta_\eps(t) \to \theta  (t) \defas   \Psi(\nabla y(t), w(t))$ \EEE a.e.~in $\Omega$ for a.e.~$t \in I$.  Thus, \BBB we get that $\phi(\lambda - \theta_\eps)$, $\phi'(\lambda - \theta_\eps)$, $c_V(\nabla \ZZZ y_\eps \EEE, \theta_\eps)$, \ZZZ and $ \partial_F W^{\rm in}(\nabla  y_\eps, \theta_\eps) $ \EEE  converge \ZZZ pointwise \EEE to their respective limits with $\theta$ and \ZZZ $y$ in place of $\theta_\eps$ and $y_\eps$ \EEE \MMM a.e.~in $I \times \Omega$. \EEE  In the same way, $\varphi_\eps$ converges pointwise to $\varphi  \defas \frac{\phi(\lambda - \theta)\phi'(\lambda - \theta)}{c_V(\nabla y, \theta)}$ a.e.~in $I \times \Omega$. Since $\theta_\eps\geq0$ a.e.~in \MMM $I \times \Omega$, \EEE we have by the choice of $\lambda$ \MMM and the fact that $\phi$ is increasing, \ZZZ see Lemma~\ref{lem:propertyphibeta}, \EEE the uniform bound $\phi(\lambda-\theta_\eps)\leq \phi(\lambda) \leq \phi(\lambda_0)$ a.e.~in $I \times \Omega$, \lll where $\lambda_0 \defas \max_{t \in I} \vert\lambda(t)\vert$. \EEE \MMM Moreover, combining the first two estimates in \eqref{eq: the good inequality}, we get $\phi'(\lambda-\theta_\eps) \le 4$.  \BBB Recalling   \eqref{inten_mon}  and using \ZZZ the \EEE uniform bound on  $\partial_F W^{\rm in}(\nabla \ZZZ y_\eps \EEE, \theta_\eps)$ in \eqref{est:coupl}, the dominated convergence theorem  \ZZZ together with \eqref{approxchainruley} \EEE implies that
\begin{align}\label{eq:somelimits}
& \lim\limits_{\eps \to 0} \frac{1}{2} \int_\Omega \phi(\lambda(t)  - \theta_\eps(t) )^2 \di x = \frac{1}{2} \int_\Omega \phi(\lambda(t)  - \theta(t) )^2 \di x \quad \text{for a.e.~$t \in I$}, \notag\\
&  \BBB \lim\limits_{\eps \to 0}  \int_{t_1}^{t_2} \int_\Omega \varphi_\eps   c_V(\nabla \ZZZ y_\eps \EEE, \theta_\eps)  \partial_t  \lambda   \di x \, \di t   = \int_{t_1}^{t_2} \int_\Omega \varphi   c_V(\nabla y, \theta)  \partial_t  \lambda   \di x \, \di t, \notag\\ 
&  \BBB  \lim\limits_{\eps \to 0} \int_{t_1}^{t_2} \int_\Omega \varphi_\eps    \partial_F W^{\rm in}(\nabla \ZZZ y_\eps \EEE, \theta_\eps) : \partial_t \nabla \ZZZ y_\eps \EEE   \di x \, \di t = \int_{t_1}^{t_2} \int_\Omega \varphi    \partial_F W^{\rm in}(\nabla y, \theta) : \partial_t \nabla y   \di x \, \di t  
\end{align}
for \ZZZ $t_1,t_2 \in I$. \EEE
By \III Lemma~\ref{lemma: phii}, \eqref{what an estimate}, \eqref{approxchainrule}, and  \eqref{approxchainruley} we see that $(\varphi_\eps)_\eps$ is bounded in $L^2(I; H^1(\Omega))$, \EEE
implying that $\varphi_\eps \rightharpoonup \varphi$ weakly in $L^2(I;H^1(\Omega))$, up to selecting a subsequence. 
\III Moreover, we have
\begin{align}\label{5Z}
\int_{t_1}^{t_2} \int_\Omega \partial_t w_\eps \varphi_\eps \di x \, \di t = \int_{t_1}^{t_2} \langle \partial_t w_\eps , \varphi_\eps \rangle \, \di t,
\end{align}
where $\langle \cdot, \cdot \rangle$ denotes the dual pairing of $H^1(\Omega)$ and $(H^1(\Omega))^*$.
Hence, by the triangle \ZZZ inequality, \EEE we derive that  
\begin{align*}
&\rb\left\vert \int_ {t_1}^{t_2} \langle \partial_t  w_\eps, \varphi_\eps \rangle - \langle \partial_t  w, \varphi \rangle \di t \right\vert  \leq \rb \left\vert \int_{t_1}^{t_2} \langle \partial_t  w_\eps -   \partial_t w,    \varphi_\eps \rangle   \di t \right\vert + \left\vert \int_ {t_1}^{t_2} \langle\partial_t  w,    \varphi_\eps   - \varphi \rangle \di t \right\vert \\
&\rb\quad\leq  \Vert \partial_t  w_\eps - \partial_t  w \Vert_{L^2( (t_1,t_2); (H^{1}(\Omega))^* )} \Vert \varphi_\eps \Vert_{L^2( (t_1,t_2); H^1(\Omega) )}  + \left\vert \int_ {t_1}^{t_2}  \langle \partial_t  w, \varphi_\eps   - \varphi   \rangle \di t \right\vert.
\end{align*}
Using \eqref{approxchainrule}   and \AAA the \EEE weak convergence of $(\varphi_\eps)_\eps$ \MMM in $L^2(I;H^1(\Omega))$, \EEE the right-hand side converges to zero \ZZZ   \ZZZ for each $t_1,t_2 \in \rb (0, T)\ee$. \EEE This along with \eqref{eq:somelimits} \MMMMM and \eqref{5Z} \EEE implies that we can pass to the limit \AAA in \eqref{regularizedidentity000} \rb yielding \ZZZ
\begin{align}\label{formula:chainrulemaxproof2}  
 &\rb \frac{1}{2} \ZZZ \int_\Omega \phi(\lambda(t_2)  - \theta (t_2) )^2 \di x -   \frac{1}{2} \int_\Omega \phi(\lambda(t_1)  - \theta (t_1) )^2 \di x \notag\\ &\quad\rb= \ZZZ \int_{t_1}^{t_2} \int_\Omega \varphi  \big(  c_V(\nabla y, \theta )    \partial_t  \lambda    +  \partial_F W^{\rm in}(\nabla y, \theta ) : \partial_t \nabla y \big) \di x  - \left\langle \III \partial_t w   , \varphi \EEE \right\rangle \, \di t
 \end{align}
  for \ZZZ a.e.~$t_1,t_2 \in  \rb (0, T) \ee $. \EEE
 \MMM Since $\theta \in C(I; \ZZZ L^2(\Omega))$ by Lemma \ZZZ \ref{lem:deriphi}(iii), \ZZZ and $\phi(\lambda - \theta)^2 \in [0, \lambda_0^2]$ a.e.~in $I \times \Omega$ \MMMMM due \EEE to \eqref{eq: the good inequality} and $\theta \geq 0$ a.e.~in $I \times \Omega$, \EEE \eqref{formula:chainrulemaxproof2} holds in fact for all \ZZZ $t_1,  t_2 \in I$. \EEE   \MMM Eventually, as \eqref{formula:chainrulemaxproof2} is satisfied for arbitrary $t_1,t_2 \in I$, we conclude that  $t \mapsto  \frac{1}{2} \int_\Omega (\phi (\lambda  - \theta))^2 \di x$ lies in $W^{1,1}(I)$ and for a.e.\ $t \in I$ the chain rule \eqref{formula:chainrulereg} holds. \EEE  
\end{proof}

\subsection{\MMM Proof of auxiliary lemmas\EEE}\label{aux: lemma}

\MMM In this subsection, we prove the auxiliary results in Lemmas \ref{lem:propertyphibeta}, \ref{lemma: phii}, and  \ref{lem:deriphi}\lll.\EEE

\begin{proof}[Proof of Lemma \ref{lem:propertyphibeta}]
We start by computing several derivatives of $\phi_\beta$. For any $s>0$, it holds that
\begin{align*}
\phi_\beta'(s)&= (s^4 + \beta^4)^{-3/4} s^3, & \phi_\beta''(s) &= \frac{3 s^2\beta^4 }{(s^4 + \beta^4)^{7/4}}, & \phi_\beta'''(s) &= \frac{3\beta^4 (2\beta^4s - 5 s^5)}{(s^4+\beta^4)^{11/4}}.
\end{align*}
From the computation above, \III it follows \rb that \AAA $\phi_\beta$, $\phi_\beta'$,  $\phi_\beta'' > 0$ \EEE in $(0, \infty)$.
Moreover, notice that
\begin{equation*}
  \lim_{s \searrow 0} \phi_\beta(s)
  = \lim_{s \searrow 0} \phi_\beta'(s)
  = \lim_{s \searrow 0} \phi_\beta''(s)
  = \lim_{s \searrow 0} \phi_\beta'''(s) = 0.
\end{equation*}
This shows $\phi_\beta \in C^3(\R)$.
We also see that $\beta \mapsto \phi_\beta'(s)$ is monotonously decreasing for every $s \in \R$,   with
\begin{equation*}
  \lim_{\beta \searrow 0} \phi_\beta'(s) = (s^{4})^{-3/4} s^3 \MMM  = \EEE 1 \qquad \text{for any } s > 0.
\end{equation*}
For all $s \leq 0$, we have $\phi_\beta'(s) = 0$.
This shows that \MMM $(\phi'_\beta)_\beta$ \EEE are converging monotonously from below   towards  ${\lenni \indic \EEE }_{(0, \infty)}$ \MMM as $\beta \searrow 0$. \EEE To prove the monotone convergence of $(\phi_\beta)_\beta$, we note that
\begin{align*}
  \partial_\beta \left((s^4 + \beta^4)^{1/4} - \beta\right)
  = \frac{\beta^3}{(s^4 + \beta^4)^{3/4}}  - 1 \leq \frac{\beta^3}{(\beta^4 )^{3/4}}  - 1 = 0.
\end{align*}
This shows that the sequence of functions $(\phi_\beta)_\beta$ is monotonously increasing as $\beta \to 0$ with $\lim_{\beta \searrow 0} \phi_\beta(s) = s$ for every $s > 0$. As $\phi_\beta(s) = 0$ for every $s \leq 0$, we \III derive \EEE that $\phi_\beta \to (\cdot)_+$ pointwise.

It remains to show \eqref{eq: the good inequality}.
Notice that the inequalities stated in \eqref{eq: the good inequality} are clearly satisfied for $s \leq 0$.
\martin Therefore, the case $s > 0$ remains to be investigated. \EEE  \MMM First,  \EEE using the monotonicity of $\phi_\beta$ in $\beta$ and the fact that $\phi_\beta(s) \to s$ for any $s >0$\III, we derive \EEE that $\phi_\beta(s) \leq s$. \AAA This shows the first inequality in \eqref{eq: the good inequality}. \EEE We \rb further have
\begin{align*}
\frac{\phi_\beta(s)}{\phi_\beta'(s)} = \frac{s^4 + \beta^4 - \beta (s^4 + \beta^4)^{3/4} }{s^3} = s + \beta \cdot \frac{\beta^3 - (s^4 + \beta^4)^{3/4}}{s^3} \leq s,
\end{align*}
where in the last inequality we have used $\beta^3 = (\beta^4)^{3/4} \leq (s^4 + \beta^4)^{3/4}$.
This shows the \MMM second \EEE inequality of \eqref{eq: the good inequality}.
In order to show the \MMM third \EEE inequality, we define $u \defas \beta^4 s^{-4}$ and apply the AM-GM inequality in the version $a^{1-\nu} b^\nu \leq (1-\nu) a + \nu b$ for $a =1$, $b = 1+ u^{-1}$, and $\nu = 3/4$. \MMM This yields \EEE
\begin{align*}
  (1 + u^{-1})^{3/4} \leq \MMM   \frac{1}{4} + \frac{3}{4}\Big(1 + \frac{1}{u} \Big) = \EEE  1 + \frac{1}{u} - \frac{1}{4 u}.
\end{align*}
Multiplying with $u$ leads to
\begin{align*}
1+u-u(1+u^{-1})^{3/4} \geq \frac{1}{4}.
\end{align*}
Thus, we discover that
\begin{align*}
\frac{\phi_\beta'(s) s}{\phi_\beta(s)} = \frac{s^4}{s^4 + \beta^4 - \beta (s^4 + \beta^4)^{3/4}} = \frac{s^4}{s^4 + \beta^4 - \beta^4 (1 + u^{-1})^{3/4}} = \frac{1}{1+u-u(1+u^{-1})^{3/4}} \leq 4,
\end{align*}
which is the \MMM third \EEE inequality in \eqref{eq: the good inequality}.
Finally, the last inequality of \eqref{eq: the good inequality} follows from
\begin{align*}
\frac{\phi_\beta''(s) s}{\phi_\beta'(s)} = \frac{3 \beta^4}{s^4 + \beta^4} \leq 3.
\end{align*}
\MMM This concludes the proof. \EEE
\end{proof}

\begin{proof}[Proof of Lemma \ref{lemma: phii}]
 \MMM Note that  $\phi_\beta(s) \leq s^+$ for any $s \in \R$ by Lemma \ref{lem:propertyphibeta}. \EEE Hence, by the \MMM second estimate \EEE in \eqref{eq: the good inequality} we derive that $\phi_\beta'(s) \leq 4 \phi_\beta (s) s^{-1} \leq 4$ for any $s > 0$. This is trivially satisfied also for $s \leq 0$.
Thus, \eqref{inten_mon} and $\phi_\beta(\lambda - \theta) \leq (\lambda - \theta)_+ \leq \lambda$ \ZZZ yield \EEE
\begin{align}\label{eq: H1bound1}
\MMM \Vert \varphi_\beta \Vert_{L^\infty(I \times \Omega)} \EEE \leq 4 c_0^{-1}\MMM  \Vert \lambda \Vert_{L^\infty(I)}. \EEE 
\end{align}
As $\lambda$  is constant \MMM in space,  by applying the chain rule   
the  gradient is given by \EEE  
\begin{align*}
\nabla \varphi_\beta =  \ZZZ \phi_\beta \EEE(\lambda - \theta) \phi_\beta '(\lambda - \theta) \nabla \big( c_V(\nabla y, \theta)^{-1} \big) - \frac{\nabla \theta \big( \phi'_\beta (\lambda - \theta)^2 + \phi_\beta (\lambda - \theta)\phi''_\beta (\lambda - \theta) \big)}{ c_V(\nabla y, \theta)} 
\end{align*}
  for \TTT a.e.~$(t,x) \in I \times  \Omega$. \EEE
 \MMM Note that \EEE \eqref{eq: the good inequality} leads to   $\phi_\beta (z) \phi_\beta''(z) = \phi_\beta (z) \phi_\beta''(z) \MMM z^2 z^{-2}  \leq 12 \phi_\beta(z)^2 z^{-2} \le 12 \EEE$. Then, by $\phi_\beta(\lambda - s) \leq \lambda$, $\phi'_\beta(\lambda - s) \leq 4$ for any $s \geq 0$, \ZZZ and \eqref{inten_mon}  \MMM we get 
\begin{align*}
\Vert \nabla \varphi_\beta \Vert_{L^2(I\times \Omega)}   \leq 4  \Vert \lambda \Vert_{L^\infty(I)} \Vert   \nabla \big( c_V(\nabla y, \theta)^{-1} \big) {\lenni \indic \EEE }_{\lbrace  \nabla \varphi_\beta  \neq 0 \rbrace} \Vert_{L^2(I\times \Omega)} + 28c_0^{-1}\Vert \nabla \theta \Vert_{L^2(I\times \Omega)}. 
\end{align*}
This along with  Lemma \ref{lemma: phii-neu}, \III see \eqref{gradcVbound}, \EEE and the fact that $\theta \le \lambda$ on $\lbrace  \nabla \varphi_\beta  \neq 0 \rbrace$ yields   
\begin{align*}
\Vert \nabla \varphi_\beta \Vert_{L^2(I\times \Omega)}   
\III \leq C \Vert \lambda \Vert_{L^\infty(I)} \left(1+  \Vert \nabla y \Vert_{L^\infty(I)} \right) \left( \Vert \nabla \theta \Vert_{L^2(I\times \Omega)} + \Vert \lambda \Vert_{L^\infty(I)}   \Vert \nabla^2 y \Vert_{L^2(I\times \Omega)}    \right) 
+ C \Vert \nabla \theta \Vert_{L^2(I\times \Omega)} . \EEE
\end{align*} 
Combining \AAA this with \EEE \eqref{eq: H1bound1}, we find \eqref{eq: H1bound0}. \AAA This \EEE  shows    $\varphi_\beta \in L^2(I; H^1(\Omega))$  since    $y \in L^\infty(I;\AAA W^{2,p}\EEE (\Omega;\R^d))$,    $\theta \in \MMM L^2(I;H^1(\Omega))$, and $\lambda \in  \AAA C^1 \EEE (I)$. \EEE   
\end{proof}

\begin{proof}[Proof of Lemma \ref{lem:deriphi}]
(i) First, notice that $W^{\rm in}$ is $C^3$   \MMM by  \ref{C_Wint_regularity}. \EEE   Consequently, differentiating \eqref{Phi_identity} with respect to $w$ and using the definition of $\theta \MMM = \Psi(F,w)$ \EEE we derive that
\begin{align}\label{the first}
\partial_w \Psi(F, w)&  = \frac{1}{\partial_\theta W^{\rm in}(F, \Psi(F, w))  } = \frac{1}{\partial_\theta W^{\rm in}(F, \theta)} = \frac{1}{c_V(F, \theta)},
\end{align}
where in the last step we recall the definition  $c_V(F, \theta) = \partial_\theta W^{\rm in}(F, \theta)$ in \eqref{inten_mon}. \MMM This is the first part of \eqref{deri1}. \EEE Differentiating the identity \eqref{Phi_identity} with respect to $F$ yields
  \begin{align*}
   0 = \partial_F W^{\rm in} (F, \theta) + \partial_\theta W^{\rm in} (F, \theta) \partial_F \Psi(F,w).
  \end{align*}
  Solving for $\partial_F \Psi(F, w)$ directly leads to    \asdf \eqref{deri1}. \EEE
 \asdf Proceeding similarly, \EEE we get that $D^2 \Psi$ and $D^3 \Psi$ consist of products of derivatives of $ W^{\rm in} $ up to third order, multiplied by $c_V(F, \theta)^{-k}$ for some $k \ge 1$. By  the fact that  $W^{\rm in}$ is $C^3$  on  $GL^+(d) \times \R_+$, the continuity of $c_V(F, \theta)$ (see \eqref{inten_mon}),  and the continuity of  $\theta \MMM = \Psi(F,w)$, this shows that  $\Psi$    is $C^3$ on $GL^+(d) \times \R_+$.  

(ii) \MMM   Consider \ZZZ $y \in \AAA L^\infty \EEE (I;\AAA W^{2,p} \EEE(\Omega;\R^d))$. \EEE    \MMM We first check that $\theta \in L^2(I;H^1(\Omega))$ implies  $w \in L^2(I; H^1(\Omega))$.  \EEE In this regard, by \eqref{inten_lipschitz_bounds} and $\theta \in L^2(I;H^1(\Omega))$\ZZZ, \EEE we get $w \in L^2(I\times\Omega)$. \ZZZ The chain rule  yields \EEE
\begin{align*}
 \nabla {w} = \partial_F W^{\rm in}(\nabla y,\theta) \colon  \nabla^2 {y} + \partial_\theta W^{\rm in}(\nabla y,\theta)   \nabla \theta.
 \end{align*}
By  \eqref{est:coupl} \MMM and \EEE  \eqref{inten_mon} \MMM  we  find
$$   \Vert \nabla w \Vert_{L^2(I \times \Omega)}   \le  C\big(1 +  \Vert \nabla  y \Vert_{L^\infty(I \times \Omega)}\big) \Vert \nabla^2  y \Vert_{L^2(I \times \Omega)}   + C_0 \Vert \nabla  \theta \Vert_{L^2(I \times \Omega)}. $$
Then,  \ZZZ as $p>d$, \rb Sobolev \ZZZ embedding and Young's inequality show \EEE 
$$ {  \Vert \nabla w \Vert_{L^2(I \times \Omega)}   \le  C \Vert \nabla  \theta \Vert_{L^2(I \times \Omega)} + C\big(1+  \ZZZ \Vert   y\Vert^2_{L^{\infty}(I; \AAA W^{2,p} \EEE (\Omega))} \big). } $$
Thus, $w \in L^2(I; H^1(\Omega))$ since $\theta \in L^2(I;H^1(\Omega))$ and $y \in L^{\infty}(I; \AAA W^{2,p} \EEE (\Omega;\R^d))$. The reverse implication and the corresponding bound follow along similar lines, by using $\Psi$ in place of $W^{\rm in}$. \EEE

\ZZZ 
(iii) \MMMMM Since   $(y,\theta)\in \mathcal{S}_{\rm chain} $, we have \EEE $w=W^{\rm in}(\nabla y, \theta) \in   H^1(I;(H^1(\Omega))^*)$. \AAA By (ii) we also have $w \in L^2(I;H^1(\Omega;\R_+))$. \EEE  This immediately gives $w\in C(I;L^2(\Omega))$, \AAA see  \cite[Lemma 7.3]	{Roubicek-book}. \EEE  It remains to prove that $\theta \in C(I; L^2(\Omega))$. 
 
%
%
 \AAA 
 As  $(y,\theta)\in \mathcal{S}_{\rm chain} $, similarly to  \eqref{Linftytemp2},  we find $c_\infty = \inf_{I \times \Omega }  \det \nabla y  >0$. \MMMMM The set \III  $K_{c_\infty} = \{ F \in \R^{d \times d}: \det(F) \geq  \MMMMM c_\infty , \III \ \vert F \vert \leq  c_\infty^{-1} \}$ is a compact subset of $GL_+(d)$. \MMMMM By Lemma~\ref{lem:deriphi}(i) (see \eqref{deri1}) along with \eqref{inten_mon} and \eqref{est:coupl}   this implies \EEE  that  $\Vert \partial_w \Psi   \Vert_{L^\infty(\III K_{c_\infty}  \times \R_+)} + \Vert\partial_F \Psi\Vert_{L^\infty( \III K_{c_\infty}  \times \R_+)}  < + \infty$. \EEE
\III $K_{c_\infty}$ is a path connected  subset of $GL^+(d)$, and thus  for all $F_1, F_2 \in K_{c_\infty}$ we can find a smooth path $\gamma \colon [0,1] \to K_{c_\infty}$ with $\gamma(0) = F_1$, $\gamma(1) = F_2$ and a constant $C>0$ only depending on $K_{c_\infty}$ such that $\Vert \gamma' \Vert_{L^\infty([0, 1])} \leq C \vert F_1 - F_2 \vert $.  Let $\gamma_{\MMMMM x,t,s}$ be such a smooth path \MMMMM from $\nabla y(s,x)$ to $\nabla y(t,x)$. \III Then, the fundamental theorem of calculus and Jensen's inequality imply that
 \begin{align}\label{convonA2}
&\Vert \theta(t) - \theta(s) \Vert_{L^2({\III \Omega  \EEE })}^2 = \Vert \Psi\big(\nabla y(t), w(t) \big) - \Psi\big(\nabla y(s), w(s) \big)\Vert_{L^2({\III \Omega   })}^2 \notag  \\
 & \leq \ZZZ \int_0^1  \MMMMM \int_\Omega \Big| \EEE  \partial_F  \Psi  \big(\gamma_{\MMMMM x,t,s}(z) , z w(t)+ (1-z) w(s) \big)   :    \gamma_{x,t,s}'(z) \notag \\ & \qquad +     \partial_w  \Psi \big( \gamma_{x,t,s}(z) ,   z   w(t) + (1-z) w(s)\big)   \big( w(t)  - w(s)\rb\big) \MMMMM \Big|^2 \, {\rm d}x \, \EEE \di z \notag \\
&\leq C \Vert \partial_w \Psi   \Vert_{L^\infty(\III K_{c_\infty}  \times \R_+)}^2 \Vert w(t) - w(s) \Vert_{L^2(\Omega)}^2  + C    \Vert \partial_{F} \Psi \Vert^2_{L^\infty(\III K_{c_\infty} \times \R_+)}     \Vert \nabla y(s) - \nabla y(t) \Vert_{L^2(\Omega)}^2 .
\end{align} 
As $w\in C(I;L^2(\Omega))$ and $\nabla y \in C(I; L^2(\Omega ;   \R^{d \times d}   ))$,   we can pass to the limit $s \to t$  on the right-hand side of \III \eqref{convonA2}, respectively. This yields \ee the \rb desired \III continuity \EEE of \ZZZ $\theta$.
\end{proof}


%

\section{Strict positivity of the temperature \AAA in \EEE the nonlinear model\EEE}\label{sec:strictposregu}

In this section, we derive the positivity of \MMM the temperature for weak solutions to \ZZZ \eqref{strong_formulation}--\eqref{strong_formulation_boundary_conditions}, i.e., \EEE we show Theorem~\ref{thm:positivity_of_temperature}. We start by establishing a corresponding result for  $\nu$-regularized solutions considered in Section \ref{sec: reg sol} for the choice $\eps=1$ and $\alpha=2$, and then obtain our main result in the limit $\nu \to 0$. \AAA For notational convenience, we write $\theta_\flat$ and $\theta_0$ in place of $\theta_{\flat,1}$ and $\theta_{0,1}$ for the data in Definition \ref{def:weak_solutions_regularized}. \EEE

\MMM Let us start by proving \EEE that the temperature is positive in the $\nu$-regularized setting.  \MMM We recall the general strategy of the proof mentioned already in the introduction: \EEE Considering the  solution $\lambda\colon I \to [0,\infty)$ of the ODE
\begin{align}\label{lambda def}
\rb\frac{\di}{\di t}\ee \lambda = - \MMM \tilde D \EEE \lambda, \quad \lambda(0) = \lambda_0,
\end{align}
\MMM a suitable \EEE choice of $\tilde D>0$ and $\lambda_0>0 $ will guarantee
\begin{align*}
  \frac{{\rm d}}{{\rm d}t} \int_\Omega  \frac{1}{2}  (\lambda(t)  - \thetanu(t))_+^2 \di x \le 0.
\end{align*}
\martin This, along with the assumption $\theta_0 \ge \lambda_0$ a.e.\ at  the initial time, implies that $\thetanu(t) \ge \lambda(t)$ a.e.\ in $\Omega$ for all $t \in I$, and establishes the positivity of the temperature. \MMM The exact arguments crucially rely on \EEE Theorem~\ref{thm:chainrule}.\EEE 

\begin{proposition}[Strict positivity of the temperature \rb in the regularized setting\EEE]\label{prop:positivity_of_temp}
Assume that the initial datum \ZZZ in \eqref{strong_formulation_initial_conditions} \EEE satisfies $\theta_{0, \rm min}\defas \essinf_{x \in \Omega} \ZZZ \theta_0 \EEE >0$ and that there exists a constant \MMM $\hat{D}>0$ \EEE such that $\bt(t) \geq \theta_{0, \rm min} \exp(-\hat{D} t)$ for all $t \in I$. Moreover, 
suppose that \MMM \ref{C_third_order_bounds}--\ref{C_Wint_regularity} \EEE hold.
Then, there exist constants \rb $C, \, \nu_0, \, \lambda_0  > 0$ such that \MMM for all $\nu \le \nu_0$    \EEE the following holds true:
\MMMMM For every weak \EEE solution $(\ynu, \thetanu)$ \rb of \ee \eqref{weak_limit_mechanical_equation_nu}\rb--\ee\eqref{weak_limit_heat_equation_nu} \rb in the sense of \ZZZ Definition~\ref{def:weak_solutions_regularized}  \EEE it holds \rb that \ee \EEE
\begin{equation}\label{positivity_of_temp}
  \thetanu(t) \geq \lambda_0 \exp(- \MMM C \EEE  t) \qquad \text{for a.e.~} t \in I.
\end{equation}

\end{proposition} 

  \begin{proof}
  \emph{Step 1 (Preparations):}  Let  $\lambda\colon I \to [0,\infty)$ be the solution \MMMMM to \EEE  the ODE \MMM in \EEE \eqref{lambda def} for a constant $\tilde D \MMM \ge \hat{D} \EEE $ and an initial value \MMM $\lambda_0 \MMMMM \in (0,1) \EEE $ \EEE satisfying  $\lambda_0 \le \theta_{0,{\rm min}}$, where we will tune   the constants $\lambda_0$ and $\tilde D$ throughout the proof, \MMM see  \eqref{A_3estimate1}, \eqref{A_3estimate1.5}, \eqref{A_3estimate3}  (for $\lambda_0$) and \eqref{A_5estimate-neu} (for $\tilde D$). 
   \lll  The \III unique \EEE solution  \rb $\lambda$ \ee is given by 
\begin{align}\label{lambda def2}
\lambda(t) = \lambda_0   \exp(-\tilde D t).
\end{align}
\MMM As $\tilde D \ge \hat{D}$, \EEE it holds that  
\begin{align}\label{eq bt}
\bt(t) \ge \lambda(t) \quad \text{\AAA a.e.\ in \lll $\Gamma$ \EEE for all $t \in I$.}
\end{align}
Moreover, \MMM  we have \EEE
\begin{align}\label{timezero}
 \int_\Omega (\lambda(0) - \thetanu(0))_+^2 \di x = 0,
 \end{align}
 where we used the fact that   $\lambda(0) = \lambda_0 \le \theta_{0,{\rm min}} \leq  \ZZZ \theta_{0} = \EEE \thetanu(0) $ for a.e.~$x \in \Omega$, \AAA see Definition \ref{def:weak_solutions_regularized}. \EEE We get  $(y_\nu, \theta_\nu)  \in \ZZZ\mathcal{S}_{\rm chain} \MMM$ \ZZZ (see Remark~\ref{rem:consequencechainrule}(i)), \EEE and thus the   chain rule in Theorem \ref{thm:chainrule} is applicable.     For any $t_2 \in I$, by \eqref{final chain rule} we have 
   \begin{align*}
 \Pi  &\defas   \frac{1}{2} \int_\Omega (\lambda(t_2) - \thetanu(t_2))_+^2 \di x -  \frac{1}{2} \int_\Omega (\lambda(0) - \thetanu(0))_+^2 \di x   \notag \\ & =  \int_{0}^{t_2} \int_\Omega  (\lambda  - \thetanu )_+\big( \rb \frac{\di}{\di t} \ee \lambda   + \rb c_V(\nabla \ynu, \thetanu)^{-1}\ee \partial_F W^{\rm in}(\nabla \ynu, \thetanu) : \partial_t \nabla \ynu \big) \di x   -     \rb\Big\langle  \partial_t \wnu, \frac{(\lambda  - \thetanu )_+ }{c_V(\nabla \ynu, \thetanu)} \Big\rangle \ee \di t,
   \end{align*}
   where $\wnu \MMM \defas \EEE W^{\rm in} (\nabla \ynu, \thetanu)$. \MMM Therefore, in view of \eqref{timezero}, it suffices to check that $\Pi \le 0$ for each $t_2 \in I$.  \EEE 
  In the following, we frequently use that $\thetanu \geq 0$ \lenni a.e.~in \AAA $I \times\Omega$, \MMM which holds  by Definition \ref{def:weak_solutions_regularized}.

\emph{Step 2 (Bound on $\Pi$):}  \AAA Define \EEE 
\begin{align}\label{piphi}
\varphi \rb \defas \ee  \frac{(\lambda  - \thetanu )_+}{c_V(\nabla \ynu, \thetanu)} .
\end{align}
Notice that $\varphi {\lenni \indic \EEE }_{(0,t_2)} \in L^2(I;H^1(\Omega))$, \MMM see  \ZZZ Remark~\ref{rem:consequencechainrule}(ii). \EEE Therefore,  $\varphi{\lenni \indic \EEE }_{(0,t_2)}$ is an admissible test function  in \eqref{weak_limit_heat_equation_nu} and we derive with   \lenni  \eqref{lambda def2} \ee  
\begin{align*}
\Pi &  =  \int_{0}^{t_2} \int_\Omega -  (\lambda  - \thetanu )_+ \tilde D\lambda  +   \hcm(\nabla \ynu, \thetanu) \nabla \thetanu \cdot \nabla \varphi \di x \di t \notag \\
    &\quad - \int_{0}^{t_2} \big(
      \MMM \xi_{\nu, 2}^{\rm reg} \EEE  (\nabla \ynu, \dotnablaynu, \thetanu)
      + \pl_F \cplpot(\nabla \ynu, \thetanu) : \dotnablaynu
    \big) \varphi  \di x \di t  \notag \\
  &\quad +  \int_{0}^{t_2} \kappa\int_{\Gamma} (\thetanu -\theta_\flat) \varphi \di \haus^{d-1} \di t +   \int_{0}^{t_2} \int_\Omega \varphi \, \partial_F W^{\rm in}(\nabla \ynu, \thetanu) : \dotnablaynu \di x \di t.
 \end{align*}
Computing the gradient of $\varphi$ \MMM by the product and chain rule, \III and \EEE using $\nabla (\lambda  - \thetanu )_+  = - \nabla \thetanu  {\lenni \indic \EEE }_{\{\thetanu \leq \lambda\}}$,  \EEE this implies that
\begin{align}\label{timederivfunc2}
\Pi & =    -  \int_{0}^{t_2}  \int_\Omega (\lambda  - \thetanu )_+ \tilde D\lambda \di x \di t   -   \int_{0}^{t_2}  \int_\Omega \hcm(\nabla \ynu, \thetanu) \nabla \thetanu \cdot \nabla \thetanu \frac{ {\lenni \indic \EEE }_{\{\thetanu \leq \lambda\}}}{ c_V(\nabla \ynu, \thetanu)}
    \di x \di t    \notag\\
  &\quad   +   \int_{0}^{t_2}   \kappa\int_{\Gamma} (\thetanu -\theta_\flat) \varphi \di \haus^{d-1}  \di t +  \int_{0}^{t_2}  \int_\Omega \hcm(\nabla \ynu, \thetanu) \nabla \thetanu \cdot  (\lambda  - \thetanu )_+\nabla \big(c_V(\nabla \ynu, \thetanu)^{-1} \big)   \di x \di t   \notag \\
  &\quad  +  \int_{0}^{t_2} \int_\Omega    
   \Big( \big( \partial_F W^{\rm in}(\nabla \ynu, \thetanu) -\partial_F\cplpot(\nabla \ynu, \thetanu) \big): \dotnablaynu - 
       \MMM \xi_{\nu, 2}^{\rm reg} \EEE  (\nabla \ynu, \dotnablaynu, \thetanu) \Big) \varphi  \di x \di t \notag \\
       & \ZZZ \eqqcolon \EEE \int_{0}^{t_2}  \big( A_1(t) + A_2(t) + A_3(t) + A_4(t) + A_5(t) \big) \, \di t ,
 \end{align}
 where each $A_i$, $i = 1,...,5$, corresponds to a term involving exactly one integrand in its respective order.
 
\MMM Our goal is to show that $A_1$, $A_2$, $A_3$ are nonpositive and \EEE  that we can control $A_4$ and $A_5$ with $\rb-\ee A_1$ and $\rb - \ee A_2$ \MMM for \AAA a.e.\ \EEE $t \in (0,t_2)$ \EEE such that the sum of those terms are negative, as long as the constants \MMM $\lambda_0^{-1}$ and \EEE $\tilde D$ \rb are \ee chosen sufficiently large \MMM independently of $t$. As all following arguments are performed pointwise in time for a fixed   $t \in (0,t_2)$,       for \MMMMM notational \EEE convenience, we will drop the integration in time   and omit $t$ in the notation.   \EEE 

 Notice that $A_1$ \MMM is nonpositive, \ZZZ see \eqref{lambda def2}. \EEE For  $A_2$, we use  \eqref{spectrum_bound_K}--\eqref{hcm} and \eqref{pos_det} to \rb derive \ee that  $\hcm(\nabla y_\nu,\theta_\nu)$ is uniformly bounded from below   (in the eigenvalue sense). This along with \eqref{inten_mon} shows that  \rb there exists \ee a  constant $c >0$ such that 
\begin{align}\label{forA2}
A_2 \le - c\int_{\{\thetanu \leq \lambda\}} | \AAA \nabla \theta_\nu \EEE |^2 \di x.
\end{align} \EEE
Due to $\theta_\flat \geq \lambda$, see \eqref{eq bt}, and the nonnegativity of $(\cdot)_+$ and $c_V$, we derive that
\begin{align}\label{A_2estimate}
A_3 =  \kappa\int_{\Gamma \cap \{ \rb\thetanu \leq \lambda\ee \}} (\thetanu -\theta_\flat) \frac{(\lambda  - \thetanu )_+}{c_V(\nabla \ynu, \thetanu)}  \di \haus^{d-1} \leq 0.
\end{align}
We proceed to estimate $A_4$.
By \eqref{pos_det} \rb     together \ee with  \MMM  \eqref{spectrum_bound_K}--\eqref{hcm} \EEE  and   \eqref{gradcVbound}, we \ZZZ have  \EEE
\begin{align*}
A_4 & = \int_{\lbrace \thetanu \le \lambda\rbrace} \hcm(\nabla \ynu, \thetanu) \nabla \thetanu \cdot  (\lambda  - \thetanu )_+\nabla \big(c_V(\nabla \ynu, \thetanu)^{-1} \big)    \di x \\ &\leq C \ZZZ \frac{C_0^2}{c_0^2}  \EEE \int_{\lbrace \thetanu \le \lambda\rbrace} \AAA  (\lambda  - \thetanu )_+ \EEE \big(\vert \nabla \thetanu \vert^2 + \lambda \vert \nabla \thetanu \vert \vert \nabla^2 \ynu \vert \big)\di x. 
\end{align*}
Thus, choosing $\lambda_0$ sufficiently small,   we see by $\lambda \le \lambda_0$  \ZZZ  and \EEE \eqref{forA2} \EEE that
\begin{align}\label{A_3estimate1}
A_4\leq  - \frac{A_2}{3} + C\lambda  \int_\Omega (\lambda  - \thetanu )_+ \vert \nabla \thetanu \vert \vert \nabla^2 \ynu \vert  \di x 
\end{align}
for $C>0$ sufficiently large. \rb Up to possibly further decreasing $\lambda_0$, we get by Young's inequality that \ee
\begin{align}\label{A_3estimate1.5}
C \lambda \int_\Omega (\lambda - \thetanu)_+ \vert \nabla^2 \ynu \vert     \vert \nabla \thetanu \vert   \di x &\leq C  \lambda \int_{\lbrace \thetanu \le \lambda\rbrace}  \vert \nabla \thetanu \vert^2 \di x  + C \lambda \int_\Omega (\lambda - \thetanu)^2_+ \vert \nabla^2 \ynu \vert^2  \di x \notag \\
&\leq -\frac{A_2}{3} + C \lambda \int_\Omega (\lambda - \thetanu)^2_+ \vert \nabla^2 \ynu \vert^2  \di x.
\end{align}
By  Hölder's inequality with powers $p/2$ and $p/(p-2)$ and the fact that $\Vert \nabla^2 \ynu \Vert_{L^\infty(I;L^p(\Omega))} \leq \MMM M \EEE$, see \eqref{higherorerbounds},  we then deduce 
\begin{align}\label{A_3estimate- do not remove}
\int_\Omega (\lambda - \thetanu)^2_+ \vert \nabla^2 \ynu \vert^2  \di x\leq \MMM C \EEE  \Vert (\lambda - \thetanu)^2_+  \Vert_{L^{p/(p-2)}(\Omega) }. 
\end{align}
\MMMMM As \EEE  $p \geq 2d$, we have $p/(p-2) \le d/(d-1)$, and thus the Sobolev inequality implies   that
\begin{align}\label{A_3estimate2.5}
 \Vert (\lambda - \thetanu)^2_+  \Vert_{L^{p/(p-2)}(\Omega) } & \leq \MMM C \EEE  \Vert (\lambda - \thetanu)^2_+    \Vert_{L^{d/(d-1)}(\Omega) }  \leq  C \Vert (\lambda - \thetanu)^2_+  \Vert_{W^{1,1}(\Omega) } \notag \\
& \MMM \le \EEE C \int_\Omega (\lambda - \thetanu)^2_+  \di x + 2 C\int_\Omega (\lambda  - \thetanu )_+ \vert \nabla \thetanu \vert  \di x  \notag \\
& \leq   C  \int_\Omega (\lambda  - \thetanu )_+   \big( 1 + \vert \nabla \thetanu \vert^2 \big)  \di x,
\end{align}
where in the last step we used $\lambda_0 <1$ and the fact that $s \le 1 +s^2$ for all $s \ge 0$. Recalling \MMM \eqref{forA2} \EEE once again, we choose $\lambda_0$ even smaller,   and combine \ZZZ \eqref{A_3estimate1}--\eqref{A_3estimate2.5} \MMM  to find \EEE
\begin{align}\label{A_3estimate3}
A_4 \le  C \lambda \int_\Omega (\lambda - \thetanu)_+ \di x  -  A_2.
\end{align}
We now estimate the term $A_5$.
In view of \eqref{avoidKorn}  and \eqref{pos_det}, we obtain \AAA pointwise a.e.\ \EEE
\begin{align*}
\big(\partial_F W^{\rm in}(\nabla \ynu, \thetanu) -\partial_F\cplpot(\nabla \ynu, \thetanu) \big): \dotnablaynu \leq C  (\thetanu \wedge 1) \  \xi(\nabla \ynu, \dotnablaynu, \thetanu)^{1/2}. 
\end{align*}
 \MMM Recall \AAA  the definition of $\xi_{\nu, 2}^{\rm reg}$ in \eqref{def_xi_alpha_reg} \EEE and the fact that $\xi^{(2)} = \xi$\ZZZ, see \eqref{def_xi_alpha}. \MMM   Using  $s \wedge 1 \leq s^{1/2}$ for $s \geq 0$\rb, \ZZZ Young's inequality \ZZZ \AAA (in the case $\xi (\nabla \ynu, \dotnablaynu, \thetanu)  \leq \nu^{-1}$) and  \EEE   $s \wedge 1 \leq 1$ we derive that \AAA pointwise a.e.\ \EEE
\begin{align*}
C (\thetanu \wedge 1) \xi(\nabla \ynu, \dotnablaynu, \thetanu)^{1/2}  \leq \begin{cases}  
 C^2\thetanu + \ZZZ \xi_{\nu, 2}^{\rm reg} \EEE(\nabla \ynu, \dotnablaynu, \thetanu)   &\text{if }  \ZZZ \xi (\nabla \ynu, \dotnablaynu, \thetanu)  \EEE \leq \nu^{-1},\\
 \ZZZ C \nu^{1/2} \xi_{\nu, 2}^{\rm reg} \EEE(\nabla \ynu, \dotnablaynu, \thetanu)    &\text{else}.\end{cases}
\end{align*}
  \AAA The combination of  the aforementioned estimates \ee yields for \lll $\nu \leq \nu_0$ \EEE sufficiently small \AAA that  pointwise a.e.\   \EEE
\begin{align*}
{\lenni \indic \EEE }_{\{\thetanu \leq \lambda\}} \big(\partial_F W^{\rm in}(\nabla \ynu, \thetanu) -\partial_F\cplpot(\nabla \ynu, \thetanu) \big): \dotnablaynu \le  {\lenni \indic \EEE }_{\{\thetanu \leq \lambda\}} \MMM  \lll \left( \ZZZ \xi_{\nu, 2}^{\rm reg} \EEE(\nabla \ynu, \dotnablaynu, \thetanu) \EEE + \AAA C^2\EEE  \lambda \lll \right) \EEE.  
\end{align*}
 This along with  \eqref{inten_mon} and  \eqref{piphi} leads to
\begin{align}\label{A_5estimate}
A_5 \leq  \MMM c_0^{-1} \ZZZ C \EEE \lambda \int_\Omega   (\lambda  - \thetanu )_+  \di x.
\end{align}
By  \eqref{timederivfunc2}, \eqref{A_2estimate}, \eqref{A_3estimate3}, and \eqref{A_5estimate} we then derive that
\begin{align}\label{A_5estimate-neu}
\Pi = \int_{0}^{t_2}  \sum\nolimits_{i=1}^5 A_i \, \di t   \leq  \int_{0}^{t_2} \Big( A_1 + A_2 + 0   -A_2  +  (C + \MMM c_0^{-1} \EEE \ZZZ C \rb) \EEE \lambda \int_\Omega  (\lambda-\thetanu)_+ \di x      \Big)  \di t .
\end{align}
Thus, in view of the definition of $A_1$,  by choosing $\tilde D$  introduced in \eqref{lambda def}  large enough, namely $\tilde D \ge C + \ZZZ c_0^{-1}C \EEE$, we get $\Pi \le 0$. This concludes the proof. 
\end{proof}  

 \MMM 
 We now come to the proof of Theorem~\ref{thm:positivity_of_temperature}.

 \begin{proof}[Proof of Theorem~\ref{thm:positivity_of_temperature}]
\ZZZ Consider a sequence of solutions $(\ynu,\thetanu)_\nu$  to the $\nu$-regularized system as given in Definition~\ref{def:weak_solutions_regularized} \AAA (for $\alpha = 2$ and $\eps=1$). \rb The \ZZZ existence \rb of such a sequence \ZZZ is guaranteed by Proposition~\ref{thm:existence_positivity_regularized}(i)\rb. \ZZZ In view of \III Proposition~\ref{thm:existence_positivity_regularized}(iv), \EEE there exists a weak solution $(y,\theta)$ to the boundary value problem \eqref{strong_formulation}--\eqref{strong_formulation_boundary_conditions} in the sense of Definition~\ref{def:weak_formulation-classic} such that $\thetanu \to \theta$ pointwise a.e.~in $I \times \Omega$, up to selecting a subsequence.
Thus, Proposition~\ref{prop:positivity_of_temp} implies the result since the constants $C$ and $\lambda_0$ in \eqref{positivity_of_temp} do not depend on $\nu$. \EEE
\end{proof}
\EEE

\section{\MMM Linearization at positive temperature\rb s\ee}\label{sec:proofmaintheorem} 

\MMM This section is devoted to the proof of Theorem \ref{thm:linearization_positive_temp}. In Subsections \ref{sec:apriorireg}--\ref{sec: a priori}, we derive \EEE a priori bounds on the deformation and temperature with optimal scaling in $\eps$. \rb We remark here that t\ee he bound in \eqref{toten_bound_scheme} does not provide the desired scaling of the mechanical energy\rb. \AAA Based on the a priori estimates, in  Subsection~\ref{subsec: lin}, we prove the linearization result. \EEE

\subsection{\MMM A priori estimates on energy and dissipation}\label{sec:apriorireg}

\MMM The crucial point consists in deriving \ZZZ an a priori \EEE bound on the energy and \ZZZ the dissipation. \EEE Once this is achieved, the remaining bounds \AAA can be derived \EEE by \rb closely following \ee the reasoning in \cite[Section 3.4]{BFK} or \cite[Section 4.3]{RBMFLM}. To formulate the main statement, we need to introduce \EEE the shifted \emph{total energy functional} $\totenalpha \colon \Wid \times L^\alpha_+(\Omega) \to \R_+$ (compare with \ZZZ e.g.~\cite[Equation~(2.15)]{BFK}) \EEE   by
\begin{equation}\label{toten_shifted}
\begin{aligned}
  \totenalpha(y, \theta)
  &\defas \mechen(y) + \mathcal{W}^{\rm cpl}(y,\theta_c) + \mathcal{W}^{\rm in}_{\alpha,\theta_c}(y,\theta), \\
  \text{with } \mathcal{W}^{\rm in}_{\alpha,\theta_c}(y,\theta) &\defas \frac{\alpha}{2} \int_\Omega \vert\inten(\nabla y, \theta) - \inten(\nabla y, \theta_c)\vert^{2/\alpha} \di x,
\end{aligned} 
\end{equation}
where \ZZZ $\mechen$ and $\mathcal{W}^{\rm cpl}$ are defined in \eqref{mechanical} and \eqref{couplenergy}. \EEE
Heuristically,   \MMM including the `shifting' by $\theta_c$  in $\mathcal{W}^{\rm in}_{\alpha,\theta_c}$ ensures \EEE that \MMM an energy bound of order $\eps^2$ induces \EEE that  the deformations and temperatures are close to the identity and the critical temperature, respectively, \MMM namely \EEE \ref{W_prefers_id} implies   $\Vert\dist^2(\nabla y , SO(d))\Vert_{L^{1}(\Omega)} \leq C\eps^2$   and  \III we have \EEE $\Vert \inten(\nabla y, \theta) - \inten(\nabla y, \theta_c) \Vert_{L^{2/\alpha}(\Omega)} \leq C\eps^\alpha$ for a constant $C>0$ \rb independent of $\eps$.  \MMM The geometric rigidity result \ZZZ \cite[Theorem~3.1]{FrieseckeJamesMueller} along with the boundary condition \AAA in \eqref{strong_formulation_boundary_id} \EEE then yield \EEE a control on $\rb|\ee y - \id\rb|\ee$, and $|\theta-\theta_c|$ is controlled by  the following Lipschitz estimate, which is a consequence of  \eqref{inten_mon}: For each  $F \in GL^+(d)$ and  $0 \le \theta_1 \le \theta_2$, letting $w_i = W^{\rm in}(F,\theta_i)$, we have   $w_2 \ge w_1$ and 
\begin{align}\label{inten_mon-new}
 w_2-w_1 \le C_0(\theta_2 - \theta_1),   \quad \quad  \theta_2 - \theta_1 \le c_0^{-1}  (w_2 - w_1).
\end{align}
\EEE
 We note that $\eps^2$ is the natural energy scaling since for initial data $(y_{0,\eps}, \theta_{0,\eps})$ as in \eqref{linearization_initial_conditions} we have 
\begin{align}\label{eq: initial energy}
\totenalpha(y_{0,\eps}, \theta_{0,\eps}) \le C\eps^2
\end{align}\MMM
by \ref{W_prefers_id}, \ZZZ the second bound in \ref{H_bounds}, \EEE \ref{H_prefers_id}, and \eqref{inten_mon-new}.
We now formulate the main a priori \lll  bounds \EEE on the shifted energy and the dissipation.
\III To this end, recall that $\Lambda$ corresponds to a modeling parameter, introduced in \eqref{def_xi_alpha} and \ref{C_adiabatic_term_vanishes}. \MMMMM Consequently, \III although the statements of Propositions~\ref{lem:fineapriori-new}--\ref{lem:fineapriori} hold for \MMMMM all \III $\Lambda \gg 1$, we cannot take the limit $\Lambda \to + \infty$. \EEE
As the proof relies on the chain rule in Theorem~\ref{thm:chainrule}, it is formulated for \ZZZ regularized \EEE solutions introduced in Section \ref{sec: reg sol}.

\begin{proposition}[A priori bounds \AAA for the \EEE shifted energy and the dissipation of \ZZZ regularized \EEE solutions]\label{lem:fineapriori-new}
\MMMMM Let \EEE $(\yepsnu, \thetaepsnu)$ be a weak solution to \EEE the $\nu$-regularized evolution in the sense of Definition~\ref{def:weak_solutions_regularized}. Suppose that   \ref{C_third_order_bounds}--\ref{C_entropy_vanishes},  \ref{W_prefers_id}, and \ref{H_prefers_id}  hold.  \AAA Then, there exist some $\eps_0,\nu_0,\Lambda_0>0$ (with $\Lambda_0=1$ for $\alpha \in (1,2]$) and a constant $C>0$, independent of $\eps$, $\nu$,   such that for all $\eps \le\eps_0$, $\nu \le \nu_0$,  and  $\Lambda \ge \Lambda_0$ it holds that  \EEE   
\rb
\begin{align}
 \esssup_{t \in I}\totenalpha(  \yepsnu(t),\thetaepsnu(t)) &\leq C\eps^2, \label{toten_bound_schemeimprovedsecfinal-neu} \\
\int_I \int_{\Omega} \xi(\nabla \yepsnu, \dotnablayepsnu, \thetaepsnu) \di x \di t &\leq C \eps^2. \label{boundres:dissipation-neu}
\end{align}
\ee
\end{proposition}
Once Proposition \ref{lem:fineapriori-new} is shown, we obtain \ZZZ the \EEE remaining a priori estimates by following the strategy in \ZZZ \cite[Section~3.4]{BFK}. \MMM By passing to the limit $\nu \to 0$, the \ZZZ desired \EEE a priori bounds hold for solutions to the original nonlinear problem in Definition \ref{def:weak_formulation}, see Proposition \ref{lem:fineapriori} for details. 

Let us come to the proof strategy of Proposition \ref{lem:fineapriori-new}. In the linearization result \cite{BFK} for $\theta_c = 0$, the main idea was to suitably test the \AAA equations \eqref{weak_formulation_mechanical_eps}--\eqref{weak_limit_heat_equation_eps}. \EEE Eventually, summing both equations then resulted in an \ee energy  control of the form 
$$\mechen(y)  +  \frac{\alpha}{2} \int_\Omega  \inten(\nabla y, \theta)_+^{2/\alpha} \di x \le C\eps^2. $$
Repeating this argumentation in our setting for the shifted energy is not sufficient since it would only deliver control on the \emph{positive part} $\frac{\alpha}{2} \int_\Omega (\inten(\nabla y, \theta) - \inten(\nabla y, \theta_c))_+^{2/\alpha} \di x$. To control the \emph{negative part}, we use an argument similar to the one in Proposition \ref{prop:positivity_of_temp} with $\theta_c$ in place of $\lambda$. This leads to the following statement. 

\begin{proposition}[Lower bound on the deviation from the critical temperature]\label{prop:lowerbound}
Let $(\yepsnu, \thetaepsnu)$ be a solution \ZZZ to  the $\nu$-regularized evolution \EEE in the sense of Definition~\ref{def:weak_solutions_regularized},   and suppose that \MMM \ref{C_third_order_bounds}--\ref{C_entropy_vanishes} hold.  \AAA Then, there exist some $\eps_0,\nu_0,\Lambda_0>0$ (with $\Lambda_0=1$ for $\alpha \in (1,2]$) and a constant $C>0$, independent of $\eps$, $\nu$,   such that for all $\eps \le\eps_0$, $\nu \le \nu_0$,  and $\Lambda \ge \Lambda_0$ it holds that  \EEE   
\begin{align} 
 \Vert (\theta_c - \thetaepsnu)_+ \Vert_{L^\infty(I;L^2(\Omega))} &\leq C  \eps^\alpha  + C \MMM  \eps_{\alpha,\Lambda} \EEE   \Vert \xi \Vert_{L^1(I\times\Omega)}^{1/2} , \label{lowerboundtemperature}\\ 
 \Vert (\theta_c - \thetaepsnu)_+ \Vert_{L^2(I \times \Gamma)} &\leq C  \eps^\alpha  + C    \MMM  \eps_{\alpha,\Lambda} \EEE   \Vert \xi \Vert_{L^1(I\times\Omega)}^{1/2} , \label{lowerboundtemperatureGamma} \\
 \Vert \nabla (\theta_c - \thetaepsnu)_+ \Vert_{L^2(I\times\Omega)} &\leq C  \eps^\alpha  + C    \MMM \eps_{\alpha,\Lambda} \EEE   \Vert \xi \Vert_{L^1(I\times\Omega)}^{1/2} ,\label{lowerboundtemperaturegrad}
\end{align}
where $ \xi   \defas     \xi(\nabla \yepsnu, \dotnablayepsnu, \thetaepsnu)$ in $I\times \Omega$ and \MMM  $\eps_{\alpha,\Lambda} \defas \eps^{\alpha -1} \wedge \Lambda^{-1}$. \EEE
\end{proposition}
\MMMMM In view of  the second term on the right-hand side in \eqref{lowerboundtemperature}--\eqref{lowerboundtemperaturegrad} which depends on the fixed modeling parameter  $\Lambda$, by  \eqref{bound:dissipationepsnu}  we obtain \III a suboptimal scaling  \AAA $\eps^{\alpha-1}$. \EEE This will \lll be \EEE improved to the scaling $\eps^\alpha$ in the proof of Proposition \ref{lem:fineapriori-new}. \EEE  Using  the identity 
 $|a|^{2/\alpha} = a_+^{2/\alpha} + (-a)_+^{2/\alpha}$ for $a \in \R$ as well as  \eqref{inten_mon-new} and  \eqref{lowerboundtemperature},  as a direct consequence of Proposition \ref{prop:lowerbound}   we obtain  \EEE
%
\begin{align}
\textstyle  \hspace{-0.2cm} \Vert \frac{2}{\alpha} \mathcal{W}^{\rm in}_{\alpha,\theta_c}(\yepsnu,\thetaepsnu) - \int_\Omega  ( \inten(\nabla \yepsnu, \thetaepsnu) - \inten(\nabla \yepsnu, \theta_c) )_+^{2/\alpha}  \di x  \Vert_{L^\infty(I)}  &\leq   C\eps^2 + C    \eps_{\alpha,\Lambda}^{2/\alpha}      \Vert \xi \Vert_{L^1(I\times\Omega)}^{1/\alpha}, \label{relationtotencritical1} \\
\textstyle \hspace{-0.2cm}  \Vert  \int_{\Omega}   ( \inten(\nabla \yepsnu, \theta_c) - \inten(\nabla \yepsnu, \lll \thetaepsnu \EEE))_+^{2/\alpha} \di x  \Vert_{L^\infty(I)} &\leq   C\eps^2 + C    \eps_{\alpha,\Lambda}^{2/\alpha}   \Vert \xi \Vert_{L^1(I\times\Omega)}^{1/\alpha}. \label{relationtotencritical2} 
\end{align}
\AAA Once \EEE Proposition \ref{prop:lowerbound} is shown, we can follow the strategy in \ZZZ \cite[Sections 3.2--3.3]{BFK} to \ZZZ control \EEE the positive part which leads to the following statement.   
 
\begin{proposition}[\MMM Auxiliary \EEE bound on the   \MMM shifted \EEE  total energy]\label{prop:scalingtotalenergy}
Let $(\yepsnu, \thetaepsnu)$ be a solution \ZZZ to  the $\nu$-regularized evolution \EEE in the sense of Definition~\ref{def:weak_solutions_regularized},   and suppose that \MMM \ref{C_third_order_bounds}--\ref{C_entropy_vanishes},  \ref{W_prefers_id}, and \ref{H_prefers_id} hold.  \AAA Then, there exist some $\eps_0,\nu_0,\Lambda_0>0$ (with $\Lambda_0=1$ for $\alpha \in (1,2]$) and a constant $C>0$, independent of $\eps$, $\nu$,   such that for all $\eps \le\eps_0$, $\nu \le \nu_0$,  and   $\Lambda \ge \Lambda_0$ it holds that  \EEE   
\begin{equation}\label{toten_bound_schemeimprovedsec}
\MMM \esssup_{t \in I} \, \EEE \totenalpha(  \yepsnu(t),\thetaepsnu(t)) \leq C \eps^2 +     C \MMM  \eps_{\alpha,\Lambda}^{2/\alpha} \EEE \Vert \xi \Vert_{L^1(I\times\Omega)}^{1/\alpha},
\end{equation} 
\MMM where $ \xi   \defas     \xi(\nabla \yepsnu, \dotnablayepsnu, \thetaepsnu)$ in $I\times \Omega$ and $\eps_{\alpha,\Lambda} \defas \eps^{\alpha -1} \wedge \Lambda^{-1}$. \EEE 
\end{proposition}

From a technical point of view, Proposition \ref{prop:scalingtotalenergy} is more delicate compared to the corresponding result in \AAA \cite[Theorem~3.13]{BFK} since  in  \cite{BFK} the adiabatic term $\theta \pl_{F \theta} \cplpot(\nabla y, \theta) : \partial_t \nabla y$ in \eqref{strong_formulation_thermal} is easily handled by using  $\theta_c\partial_{F\theta} W^{\rm cpl} (F,\theta_c)  = 0$ for $\theta_c = 0$ whereas the latter does not hold any longer in the present setting $\theta_c >0$.  \EEE  Note that we call this an \emph{auxiliary} bound on the energy as  the dissipation still appears on the right-hand side of \eqref{toten_bound_schemeimprovedsec}.

We defer the proofs of Propositions \ref{prop:lowerbound}--\ref{prop:scalingtotalenergy} to Subsection \ref{sec: two pro} below and proceed with the proof of Proposition \ref{lem:fineapriori-new}. \EEE

\begin{proof}[Proof of  Proposition \ref{lem:fineapriori-new}]
We first focus on \eqref{boundres:dissipation-neu}.   By the fundamental theorem of calculus we have, \ZZZ for a.e.~$t \in I$, \EEE
$$ \int_\Omega \int_0^t   \partial_F W^{\rm cpl} (\nabla \yepsnuu(s),\theta_c)    : \dotnablayepsnuu\di s \di x  = \int_\Omega W^{\rm cpl}(\nabla \yepsnu(t),\theta_c) \di x  -  \int_\Omega W^{\rm cpl}(\nabla \yepsnu(0),\theta_c) \di x .
 $$
This, along with \EEE the energy balance \MMM in \EEE \eqref{energybalanceregularized}, implies that
\begin{align}\label{LLLLLL}
& \mathcal{M}(\yepsnu(t)) + \int_\Omega W^{\rm cpl}(\nabla \yepsnu(t),\theta_c) \di x    + \ZZZ \int_0^t \EEE \int_\Omega \xi(\nabla \yepsnu, \dotnablayepsnu, \thetaepsnu)  \di x \di s \nonumber \\
&= \mathcal{M}(\yepsnu(0))+ \int_\Omega W^{\rm cpl}(\nabla \yepsnu(0),\theta_c) \di x    + \ZZZ \int_0^t \EEE \langle   \ell_\eps(s), \dotyepsnu \AAA (s) \EEE \rangle  \di s\notag \\ &\quad \qquad  - \ZZZ \int_0^t \EEE \int_\Omega \big( \partial_F W^{\rm cpl} (\nabla \yepsnu, \thetaepsnu) - \partial_F W^{\rm cpl} (\nabla \yepsnu, \theta_c) \big) : \dotnablayepsnu\di x \di s
\end{align}
 for \ZZZ  a.e.~$t \in I$. \EEE
Since the sum of the first two terms on the left-hand side of  \MMM \eqref{LLLLLL}  is \EEE  nonnegative, see \ref{W_prefers_id}, \MMM by \eqref{eq: initial energy} \EEE we discover that
\begin{align} \ZZZ
 \quad \Vert \xi  \Vert_{L^1(I \times \Omega)}  \notag  &\leq C \eps^2 + C\int_I \langle   \ell_\eps(s), \dotyepsnu \MMMMM (s) \EEE \rangle \di s \notag \\ &\rb\phantom{\leq}\quad+\ee C \left\vert \int_I \int_\Omega \big( \partial_F W^{\rm cpl} (\nabla \yepsnu, \thetaepsnu) - \partial_F W^{\rm cpl} (\nabla \yepsnu, \theta_c) \big) : \dotnablayepsnu\di x \di s \right\vert, \label{bounds:regularizedterms}
\end{align}
\MMM where we write for shorthand $\xi =  \xi(\nabla \yepsnu, \dotnablayepsnu, \thetaepsnu)$. \EEE Our next goal is to bound the last two terms of the inequality above.  \MMM As a preparation, we control \ZZZ $ \partial_t \yepsnu$ \EEE in terms of the dissipation term.  To this end, we apply \EEE the generalized version of Korn's inequality, \ZZZ as stated in \EEE Theorem~\ref{pompe}, \MMM for  \EEE     $u= \dotyepsnu$ and $F = \nabla \yepsnu$, where   $F$ satisfies the assumptions due to \eqref{pos_det}.
 \MMM In view of \ref{D_quadratic}--\ref{D_bounds} and \eqref{diss_rate},   this shows
\begin{align}\label{pompi}
\Vert \dotyepsnu  \Vert^2_{L^2(I;H^1(\Omega))} \le C\Vert \ZZZ \partial_t \nabla  \yepsnu \EEE   \Vert_{L^2(I\times \Omega)}^2 \le C   \int_I \int_\Omega \xi(\nabla \yepsnu, \dotnablayepsnu, \thetaepsnu)  \di x \di s, 
\end{align}
where in the first step we  used Poincaré's inequality \ZZZ as $\dotyepsnu = 0$ a.e.~in $I \times \Gamma_D$. \EEE Now, on the one hand, we discover by \ZZZ \eqref{def:externalforces}, \EEE Hölder's inequality, a trace estimate, Young's inequality \ZZZ with constant $\frac{1}{3}$, and  \eqref{pompi} \EEE that  \EEE  
\begin{align}\label{forceestimate5}
\int_I \langle   \ell_\eps(s), \dotyepsnu  \AAA (s) \EEE\rangle \di s &\leq C \eps \int_I \left( \Vert  f(s) \Vert_{L^2(\Omega)}  +  \Vert  g(s) \Vert_{L^2(\Gamma_N)} \right) \Vert \dotyepsnu (s) \Vert_{H^1(\Omega)} \di s \notag \\
&\leq C \eps^2 + \frac{1}{3} \int_I \int_\Omega \xi (\nabla \yepsnu, \partial_t \nabla \yepsnu, \thetaepsnu ) \di x \di s.
\end{align}
 \ZZZ On the other hand,  
 using  \eqref{est:couplatthetacwithxi2}, \eqref{pos_det},   Young's  inequality with constant $\frac{1}{3}$,   $s \wedge 1 \leq \MMMMM ( s \wedge 1)^{1/\alpha} \EEE$ for $s \geq 0$,   and \eqref{relationtotencritical1}--\eqref{relationtotencritical2} along with the Lipschitz estimate   \eqref{inten_mon-new}, we can estimate the last term in \eqref{bounds:regularizedterms} \AAA by \EEE
\begin{align}\label{est:neededlater}
&\left\vert \int_I \int_\Omega \big( \partial_F W^{\rm cpl} (\nabla \yepsnu, \thetaepsnu) - \partial_F W^{\rm cpl} (\nabla \yepsnu, \theta_c) \big) : \dotnablayepsnu\di x \di s \right\vert \notag \\ 
&\rb\quad  \leq\ZZZ C \int_I \int_\Omega  ( |\thetaepsnu - \theta_c| \wedge 1 )^{2/\alpha} \di x \di s +  \frac{1}{3} \int_I \int_\Omega  \xi(\nabla \yepsnu, \dotnablayepsnu, \thetaepsnu) \di x \di s  \notag \\
&\rb\quad \leq \ZZZ C \int_I \int_\Omega  (\thetaepsnu - \theta_c)_+ ^{2/\alpha} + (\theta_c - \thetaepsnu)_+ ^{2/\alpha}\di x \di s +  \frac{1}{3} \int_I \int_\Omega  \xi(\nabla \yepsnu, \dotnablayepsnu, \thetaepsnu) \di x \di s  \\
&  \rb\quad\leq \ZZZ C \int_I  \mathcal{W}^{\rm in}_{\alpha,\theta_c}(\yepsnu(s),\thetaepsnu(s)) \di s  +  \frac{1}{3} \int_I \int_\Omega  \xi(\nabla \yepsnu, \dotnablayepsnu, \thetaepsnu) \di x \di s  + C \eps^{2}+ C   \MMM  \eps_{\alpha,\Lambda}^{2/\alpha} \EEE \Vert \xi \Vert^{1/\alpha}_{L^1(I\times\Omega)} .  \notag
\end{align}    
Combining \eqref{bounds:regularizedterms}, \eqref{forceestimate5}--\eqref{est:neededlater}, \ZZZ \eqref{toten_shifted}, \ref{W_prefers_id}, \EEE and Proposition~\ref{prop:scalingtotalenergy} we find that
\begin{align*}
& \frac{1}{3} \Vert \xi  \Vert_{L^1(I \times \Omega)}  \leq \MMM  C \eps^{2}+ C      \eps_{\alpha,\Lambda}^{2/\alpha} \Vert \xi \Vert^{1/\alpha}_{L^1(I\times\Omega)}.   \EEE
\end{align*}
Consider the case $\alpha \in (1,2]$, \ZZZ i.e., $\eps_{\alpha,\Lambda} = \eps^{\alpha-1}$ \AAA for $\eps$ small. \EEE \ZZZ Choosing $\eps_0>0$ small enough, \EEE Young's inequality with powers  $\alpha / (\alpha-1)$ and $\alpha$ and \MMM \rb constant \MMM $\frac{1}{6}$ \EEE  yields
\eqref{boundres:dissipation-neu}.   If $\alpha = 1$, we have \MMM $\eps_{\alpha,\Lambda} = \frac{1}{\Lambda} \rb \leq \frac{1}{\Lambda_0}$. Thus, \eqref{boundres:dissipation-neu} follows for $\Lambda_0$ large enough. Eventually, \eqref{boundres:dissipation-neu} along with \eqref{toten_bound_schemeimprovedsec} shows the energy bound \eqref{toten_bound_schemeimprovedsecfinal-neu}. \EEE 
\end{proof}

\MMM

\subsection{Proofs of Propositions \ref{prop:lowerbound}--\ref{prop:scalingtotalenergy}}\label{sec: two pro} 
\martin In this subsection, we prove the two key auxiliary statements. \EEE 
 
 \EEE

\begin{proof}[Proof of Propositions \ref{prop:lowerbound}]
The proof \rb follows along similar lines as the proof of \ee Proposition~\ref{prop:positivity_of_temp}.
\ZZZ According to Definition~\ref{def:weak_solutions_regularized}, we have $ \theta_{0,\eps} = \thetaepsnu(0)$, implying that
 by \EEE \eqref{linearization_initial_conditions}  \MMM there exists $C>0$ such that 
 \begin{align}\label{tempinitscaling}
\int_\Omega (\theta_c - \thetaepsnu(0))_+^2 \di x \leq \eps^{2\alpha} \Vert \mu_0 \Vert_{L^2(\Omega)}^2 \leq \MMM C \EEE \eps^{2\alpha}.
\end{align}
\MMM Note that by Theorem~\ref{thm:chainrule} \AAA for $\lambda = \theta_c$ and Remark \ref{rem:consequencechainrule}(i)  \EEE we have,   for any $t_2 \in I$, 
   \begin{align*}
 \Pi  &\defas  \frac{1}{2} \int_\Omega ((\theta_c - \thetaepsnu)_+(t_2))^2 \di x - \frac{1}{2} \int_\Omega ((\theta_c-\thetaepsnu)_+(0))^2 \di x \notag \\   &=  \int_{0}^{t_2} \int_\Omega  \frac{(\theta_c  - \thetaepsnu )_+ }{c_V(\nabla \yepsnu, \thetaepsnu)} \partial_F W^{\rm in}(\nabla \yepsnu, \thetaepsnu) : \dotnablayepsnu  \di x - \rb\Big\langle \partial_t \wepsnu, \frac{(\theta_c  - \thetaepsnu )_+ }{c_V(\nabla \yepsnu, \thetaepsnu)} \Big\rangle\MMM \, \di t.
   \end{align*}
\MMM The main step of the proof is to show that  there exists $C= C(M,\alpha)>0$ depending on both $M$ \rb from \ZZZ \III Proposition~\ref{thm:existence_positivity_regularized}(iii)  \EEE and $\alpha\in [1,2]$, but independent of \AAA $\nu $ and $\eps$, \EEE such that 
 \begin{align}\label{main showlin}
 \Pi   &\leq   C   \int_0^{t_2}\int_\Omega  (\theta_c-\thetaepsnu)_+^2 \di x   \di t +  C \MMM  \eps_{\alpha,\Lambda}^2 \EEE  \Vert \xi \Vert_{L^1(I \times \Omega)} +  C \eps^{2\alpha}  - \frac{\kappa} {2C_0} \int_0^{t_2}\int_{\Gamma  }    (\theta_c  - \thetaepsnu )^2_+  \di \haus^{d-1} \di s   \notag \\
&\rb\phantom{\leq}\quad - \ZZZ  \frac{1}{4} \EEE \int_{0}^{t_2}  \int_\Omega \hcm(\nabla \yepsnu, \thetaepsnu) \nabla \thetaepsnu \cdot \nabla \thetaepsnu \frac{ {\lenni \indic \EEE }_{\{\thetaepsnu \leq \theta_c\}}}{ c_V(\nabla \yepsnu, \thetaepsnu)}
    \di x \di t .
   \end{align}
Then, since the last two terms in \eqref{main showlin} are nonpositive \lll due to  \eqref{spectrum_bound_K}, \eqref{hcm}, and \eqref{inten_mon}, \EEE Gronwall's inequality  \MMM  (in integral form) \EEE   and \eqref{tempinitscaling} \III imply \EEE that  
\begin{align}\label{implicationgronwall}
\sup_{t \in I} \int_\Omega (\theta_c - \thetaepsnu(t))_+^2 \di x \leq   C e^{  C  T} \big( \MMM  \eps_{\alpha,\Lambda}^2 \EEE  \Vert \xi \Vert_{L^1(I \times \Omega)} +  \eps^{2\alpha} \big) ,
\end{align}
 where \AAA we recall that \EEE $T>0$  denotes the length of the interval $I = [0,T]$. \MMM This shows \eqref{lowerboundtemperature}. Then, combining \lll \eqref{tempinitscaling}--\eqref{implicationgronwall} \EEE we also find
\begin{align*}
  \ZZZ  &\frac{1}{4} \EEE \int_{0}^{T}  \int_\Omega \hcm(\nabla \yepsnu, \thetaepsnu) \nabla \thetaepsnu \cdot \nabla \thetaepsnu \frac{ {\lenni \indic \EEE }_{\{\thetaepsnu \leq \theta_c\}}}{ c_V(\nabla \yepsnu, \thetaepsnu)} \AAA \di x \di t  \EEE + \frac{\kappa} {2C_0} \int_0^{T}\int_{\Gamma  } (\theta_c  - \thetaepsnu )^2_+  \di \haus^{d-1} \di s \\ &\rb\quad  \le\MMM  \big( \ZZZ C+  \MMMMM C \EEE T e^{  C  T}  \EEE \big)\big(   \eps_{\alpha,\Lambda}^2 \Vert \xi \Vert_{L^1(I \times \Omega)} +  \eps^{2\alpha} \big) , 
\end{align*}
which along with \eqref{spectrum_bound_K}--\eqref{hcm},  \eqref{inten_mon},  and \eqref{pos_det} 
  \ZZZ shows \EEE \eqref{lowerboundtemperatureGamma} and \eqref{lowerboundtemperaturegrad}.  

Let us now come to the proof of  \eqref{main showlin}. \EEE \ZZZ Due to  Remark~\ref{rem:consequencechainrule}(ii),  we can test \eqref{weak_limit_heat_equation_nu} with \EEE $\varphi{\lenni \indic \EEE }_{(0,t_2)}$, \AAA where \EEE $\varphi \defas  (\theta_c  - \thetaepsnu )_+ c_V(\nabla \yepsnu, \thetaepsnu)^{-1} $.
 \MMM By repeating the argument in \eqref{timederivfunc2} for $\lambda = \theta_c$ and for $\xi_{\nu, \alpha}^{\rm reg}$ in place of $\xi_{\nu,2}^{\rm reg}$ we find\EEE 
\begin{align}\label{timederivfunc2lin}
\Pi & =   -   \int_{0}^{t_2}  \int_\Omega \hcm(\nabla \yepsnu, \thetaepsnu) \nabla \thetaepsnu \cdot \nabla \thetaepsnu \frac{ {\lenni \indic \EEE }_{\{\thetaepsnu \leq \theta_c\}}}{ c_V(\nabla \yepsnu, \thetaepsnu)}
    \di x \di t    \notag\\
  &\quad   +   \int_{0}^{t_2}   \kappa\int_{\Gamma} (\thetaepsnu -\theta_{\flat,\eps}) \varphi \di \haus^{d-1}  \di t \notag \\
  &\quad+  \int_{0}^{t_2}  \int_\Omega (\theta_c  - \thetaepsnu )_+ \hcm(\nabla \yepsnu, \thetaepsnu) \nabla \thetaepsnu \cdot   \nabla \big(c_V(\nabla \yepsnu, \thetaepsnu)^{-1} \big)    \di x \di t   \notag \\
  &\quad  +  \int_{0}^{t_2} \int_\Omega   \big(
   \big( \partial_F W^{\rm in}(\nabla \yepsnu, \thetaepsnu) -\partial_F\cplpot(\nabla \yepsnu, \thetaepsnu) \big): \dotnablayepsnu - 
        \xiregnu  (\nabla \yepsnu, \dotnablayepsnu, \thetaepsnu) \big) \varphi \di x \di t \notag \\
        &=:  \int_{0}^{t_2}   B_1(t) + B_2(t) + B_3(t) + B_4(t)   \, \di t ,
 \end{align}
 where each $B_i$, $i = 1,...,4$, corresponds to a term involving exactly one integrand in its respective order. \MMM As  in the proof of Proposition~\ref{prop:positivity_of_temp}, for \MMMMM notational \EEE convenience,   we sometimes \EEE drop the integration in time and  estimate the terms for \AAA a.e.\   \EEE fixed time $t \in (0,t_2)$.
 
 For $B_1$,    \ZZZ due to \EEE \eqref{spectrum_bound_K}--\eqref{hcm},  \eqref{inten_mon},  and \eqref{pos_det}   \EEE  \MMM we find a  constant $c >0$ such that 
\begin{align}\label{forB1}
B_1(t) \le - c\int_{\{\thetaepsnu \leq \theta_c\}} |\nabla \AAA  \theta_{\eps,\nu} \EEE (t)|^2 \di x.
\end{align} \EEE 
Our next goal is to bound $ \sum_{i=2}^4 B_i  $\lll. \EEE Due to  \eqref{def:externalforces}, \eqref{inten_mon}, and Young's inequality with a constant $\gamma_1>0$, we derive that
\begin{align*}
\int_0^{t_2} B_2(t) \di t & =   \kappa \int_0^{t_2}\int_{\Gamma  } (\thetaepsnu -\theta_{\flat,\eps}) \frac{(\theta_c  - \thetaepsnu )_+}{c_V(\nabla \yepsnu, \thetaepsnu)}  \di \haus^{d-1} \di t \notag \\
& \leq   \kappa c_0^{-1}\int_0^{t_2}\int_{\Gamma  }  \eps^\alpha \vert \mu_\flat \vert  (\theta_c  - \thetaepsnu )_+ \di \haus^{d-1} \di t - \kappa C_0^{-1} \int_0^{t_2}\int_{\Gamma  }    (\theta_c  - \thetaepsnu )^2_+  \di \haus^{d-1} \di t \notag  \\
&\leq  \frac{\kappa} {2c_0 \gamma_1} \eps^{2\alpha}\Vert \mu_\flat\Vert_{L^2([0,t_2] \times \Gamma)}^2   +  \rb\Big(\ee  \frac{\kappa}{ 2 c_0}\gamma_1  - \kappa C_0^{-1} \rb\Big)\ee \int_0^{t_2}\int_{\Gamma  } (\theta_c  - \thetaepsnu )_+^2 \di \haus^{d-1} \di t.  
\end{align*}
\MMM Choosing $\gamma_1$ such that \lll $\gamma_1  c_0^{-1} \leq   C_0^{-1}$ \EEE and using   the integrability of $\mu_\flat$, \ZZZ see \eqref{def:externalforces}, \MMM we get
\begin{align}\label{A_2estimatelin}
\int_0^{t_2} B_2(t) \di t & \leq  C \eps^{2\alpha}      - \frac{\kappa}{2 C_0}  \int_0^{t_2}\int_{\Gamma  } (\theta_c  - \thetaepsnu )_+^2 \di \haus^{d-1} \di t. 
\end{align}
 \EEE We proceed by estimating $B_3$ \MMM for fixed time $t$. Using  \eqref{gradcVinverse} we first calculate  
$$
B_3 =  -    \int_{\Omega} (\theta_c  - \thetaepsnu )_+ \hcm(\nabla \yepsnu, \thetaepsnu) \nabla \thetaepsnu \cdot   
   \frac{  \partial_\theta^2 W^{\rm in} (\nabla \yepsnu,\thetaepsnu )      \, \nabla \thetaepsnu  \AAA - \EEE \thetaepsnu \partial_{\theta \theta F} W^{\rm cpl} (\nabla \yepsnu, \thetaepsnu) : \nabla^2 \yepsnu }{c_V(\nabla \yepsnu,\thetaepsnu)^2}\di x.
$$
\lll Then, \EEE \eqref{pos_det} and \eqref{spectrum_bound_K}--\eqref{hcm} \AAA together  \EEE with \III   \ref{C_third_order_bounds}\rb, \ZZZ \ref{C_entropy_vanishes}\rb, and \eqref{inten_mon} imply that  \EEE
\begin{align}\label{A_3estimate1lin}
\frac{B_1}{2} + B_3 &\le      \int_{\lbrace \thetaepsnu \le \theta_c\rbrace}   \hcm(\nabla \yepsnu, \thetaepsnu) \nabla \thetaepsnu \cdot   \nabla \thetaepsnu \frac{-\frac{1}{2} + (\theta_c  - \thetaepsnu )_+ \frac{1}{2\theta_c} }{c_V(\nabla y,\theta)}   
     \notag\\ 
& \quad  \quad + C   \int_{\lbrace \thetaepsnu \le \theta_c\rbrace} (\theta_c  - \thetaepsnu )_+   \vert \nabla \thetaepsnu \vert \vert \nabla^2 \yepsnu \vert  \di x \notag \\
&\leq C  \int_{\lbrace \thetaepsnu \le \theta_c\rbrace} (\theta_c  - \thetaepsnu )_+  \vert \nabla \thetaepsnu \vert \vert \nabla^2 \yepsnu \vert  \di x.
\end{align}
  Employing Young's inequality \MMM and \eqref{forB1} \EEE  we derive
\begin{align*}
C  \int_\Omega (\theta_c - \thetaepsnu)_+ \vert \nabla^2 \yepsnu \vert     \vert \nabla \thetaepsnu \vert   \di x  \leq -\frac{B_1}{8} + C \int_\Omega (\theta_c - \thetaepsnu)^2_+ \vert \nabla^2 \yepsnu \vert^2  \di x.
\end{align*}
 By Hölder's inequality with exponents $p/(p-2)$ and $p/2$\lll, \EEE and \AAA by \EEE \eqref{higherorerbounds} \EEE   we then deduce 
\begin{align}\label{A_3estimate2lin}
 C \int_\Omega (\theta_c  - \thetaepsnu )_+ \vert \nabla \thetaepsnu \vert \vert \nabla^2 \yepsnu \vert  \di x \leq -\frac{B_1}{8}  +  C\Vert (\theta_c - \thetaepsnu)^2_+  \Vert_{L^{p/(p-2)}(\Omega) }. 
\end{align}
\MMMMM As \EEE $p \geq 2d$, we have $p/(p-2) \le d/(d-1)$, and thus  the Sobolev \EEE inequality implies together with Young's inequality with \rb constant \ee $\gamma_2>0$ that
\begin{align}\label{lennitrick}
  C\Vert (\theta_c - \thetaepsnu)^2_+  \Vert_{L^{p/(p-2)}(\Omega) } & \leq C \Vert (\theta_c - \thetaepsnu)^2_+    \Vert_{L^{d/(d-1)}(\Omega) }  \leq   C \Vert (\theta_c - \thetaepsnu)^2_+  \Vert_{W^{1,1}(\Omega) } \notag \\
& =  C \int_\Omega (\theta_c - \thetaepsnu)^2_+  \di x + 2   C \int_\Omega (\theta_c  - \thetaepsnu )_+ \vert \nabla \thetaepsnu \vert  \di x  \notag \\
& \leq   C \rb\Big(\ee 1+ \frac{1}{\MMMMM \gamma_2 }\rb\Big)\ee\int_\Omega (\theta_c - \thetaepsnu)^2_+ \di x  +  C \gamma_2  \int_{ \{\thetaepsnu \leq \theta_c\}} \vert \nabla \thetaepsnu \vert^2   \di x.
\end{align}
 Then, choosing   $\gamma_2$ sufficiently small and using \MMM \eqref{forB1} \EEE  we discover that
\begin{align*} C \Vert (\theta_c - \thetaepsnu)^2_+  \Vert_{L^{p/(p-2)}(\Omega) }  \leq   C \int_\Omega (\theta_c - \thetaepsnu)_+^2  \di x  - \frac{B_1}{8}.
\end{align*}
Combining this estimate with \eqref{A_3estimate1lin} and  \eqref{A_3estimate2lin} we derive  
\begin{align}\label{A_3estimate3lin}
B_3 \le   C \int_\Omega (\theta_c - \thetaepsnu  )_+^2 \di x  - \ZZZ \frac{3 B_1}{4}. \EEE
\end{align}
We now estimate the term $B_4$. In view of \ZZZ \eqref{Wint}, \EEE \eqref{est:couplatthetacwithxi}, and \eqref{pos_det} we obtain \AAA pointwise a.e.\ \EEE
\begin{align}\label{B4_estimate1}
  \begin{aligned}
&\vert \big(\partial_F W^{\rm in}(\nabla \yepsnu, \thetaepsnu) -\partial_F\cplpot(\nabla \yepsnu, \thetaepsnu) \big): \dotnablayepsnu \vert \\ &\quad\leq   \MMM C \EEE (\theta_c |\partial_{F\theta} W^{\rm cpl}(\nabla \yepsnu,\theta_c)| + \vert \thetaepsnu - \theta_c \vert \wedge 1 ) \xi(\nabla \yepsnu, \partial_t \nabla \yepsnu, \thetaepsnu)^{1/2}.
  \end{aligned}
\end{align}
Recall the definition of $\xi^{(\alpha)}$ and  $\xiregnu$ in   \eqref{def_xi_alpha} and \eqref{def_xi_alpha_reg}, respectively. \AAA  In the case  $\alpha \in (1,2]$, we choose   $\nu_0$ small enough  such that $\nu \le  \nu_0  <\Lambda^{-1}$. Then, \AAA possibly passing to a smaller $\nu_0$, \EEE  we find \AAA pointwise a.e.\ \EEE 
\begin{align*}
C  ( \vert \thetaepsnu-\theta_c \vert \wedge 1) \xi(\nabla \yepsnu, \dotnablayepsnu, \thetaepsnu)^{1/2} & \leq \begin{cases}
\MMM C^2\EEE \vert \thetaepsnu-\theta_c \vert +  \xiregnu, & \ZZZ \xi^{(\alpha)}\EEE \leq  \ZZZ \Lambda, \EEE  \\
\ZZZ \Lambda^{(\alpha -2)/ (2\alpha -2)} \EEE C^{\alpha/(\alpha-1)}    \vert \thetaepsnu-\theta_c \vert  +   \xiregnu, & \xi^{(\alpha)} \in \ZZZ (\Lambda, \EEE \nu^{-1}], \\
   \xiregnu, & \xi^{(\alpha)}>\nu^{-1},
\end{cases}
\end{align*}
 where $ \xiregnu$ is evaluated at $(\nabla \yepsnu, \dotnablayepsnu, \thetaepsnu)$. Indeed,  if $\AAA \xi^{(\alpha)} \EEE \leq \Lambda$, \ZZZ we have $\xi^{(\alpha)} = \xi \AAA =  \xiregnu$ and \AAA we \EEE use  $s \wedge 1 \leq s^{1/2}$ along with \rb Young's inequality\ee.  \EEE  If $\xi^{(\alpha)} \in (\ZZZ \Lambda \EEE,\nu^{-1}]$, we employ  $s \wedge 1 \leq s^{(\alpha -1) /\alpha}$ and   Young's inequality with \rb powers \ee $\alpha/(\alpha-1)$  and $\alpha$. The last case follows by the definition of $\xiregnu$ along with the fact that \ZZZ $C\nu^{1-1/\alpha} \Lambda^{(1-2/\alpha) /2} \leq 1$ \EEE for \AAA $\nu \le \nu_0$ and $\nu_0$ \EEE small enough\rb, where we used that $\alpha > 1$\ee.

The \MMM corresponding \EEE estimates for $\alpha = 1$ follow if we choose \MMMMM $\Lambda_0^{1/2} \geq   C$. \EEE  (Note that $C$ is independent of $\Lambda$ as it only depends on the constant $M$ in \III Proposition \ref{thm:existence_positivity_regularized}(iii).) \EEE  Indeed, we have \AAA pointwise a.e.\ \EEE
\begin{align*}
C ( \vert \thetaepsnu-\theta_c \vert \wedge 1) \xi(\nabla \yepsnu, \dotnablayepsnu, \thetaepsnu)^{1/2} & \leq \begin{cases}
 C^2\EEE \vert \thetaepsnu-\theta_c \vert +   \xiregnu, & \ZZZ \xi^{(1)} \EEE \leq \Lambda, \\
 \xiregnu, & \ZZZ \xi^{(1)} \EEE >\Lambda \\
\end{cases}
\end{align*}
\AAA for each choice $\Lambda \ge \Lambda_0$ \MMMMM in \eqref{def_xi_alpha}.  \EEE 
\MMM In all cases \ZZZ $\alpha \in [1,2]$, \EEE in \EEE view of \lll \ref{C_adiabatic_term_vanishes} and \eqref{pos_det}, \EEE we find by Young's inequality, \MMM  the definition of $\varphi$, and \eqref{inten_mon} \EEE
\begin{align}\label{B4_estimate4}
&\int_0^{t_2} \int_\Omega C \theta_c \vert \partial_{F\theta} W^{\rm cpl}(\nabla \yepsnu,\theta_c)  \vert \xi(\nabla \yepsnu, \dotnablayepsnu, \thetaepsnu)^{1/2} \varphi \di x \di t \notag \\&\qquad \leq  C \MMM  \eps_{\alpha,\Lambda}^2 \EEE  \Vert \xi \Vert_{L^1(I \times \Omega)} +   C \int_0^{t_2} \int_\Omega (\theta_c - \thetaepsnu)_+^2 \di x \di t.
\end{align}
In view of \eqref{B4_estimate1}--\eqref{B4_estimate4}, \MMM again using  \eqref{inten_mon}, \EEE we find for any $\alpha \in [1,2]$
\begin{align}\label{A_5estimatelin}
\int_0^{t_2} B_4(t) \di t \leq C   \MMM  \eps_{\alpha,\Lambda}^2 \EEE \Vert \xi \Vert_{L^1(I \times \Omega)} +  C \int_0^{t_2} \int_\Omega (\theta_c - \thetaepsnu)_+^2 \di x \di t
\end{align}
for a constant $\MMM   C  \EEE$ depending only on $M$ \MMM in \III Proposition~\ref{thm:existence_positivity_regularized}(iii) \EEE and $\alpha$, \MMM but \EEE  not on $\nu$. 
By \MMM collecting \EEE   \eqref{timederivfunc2lin}, \eqref{A_2estimatelin}, \eqref{A_3estimate3lin}, and   \eqref{A_5estimatelin}, we  
conclude the proof of \eqref{main showlin}. \MMM As seen above, this implies   \eqref{implicationgronwall}, and then eventually \eqref{lowerboundtemperature}--\eqref{lowerboundtemperaturegrad}.
\end{proof}

\MMM Having \EEE derived bounds on $\eps^{-\alpha} (\theta_c - \thetaepsnu)_+$, we address the \ZZZ auxiliary bound on $\totenalpha$ in Proposition~\ref{prop:scalingtotalenergy}. \EEE As a preparation, we relate the external forces \ZZZ (see \eqref{def:forcefunctional}) \EEE with the   shifted total energy. \EEE

\begin{lemma}\label{lem:EF}
\III Let $(\yepsnu, \thetaepsnu)$ be a solution to  the $\nu$-regularized evolution in the sense of Definition~\ref{def:weak_solutions_regularized},   and suppose that  \ref{W_prefers_id}  holds.
\ZZZ Then, 
  there exists a constant $C > 0$  such that \ZZZ for all~$t \in I$ \EEE   
  \begin{equation}\label{forceestimate}
       |\langle \ell_\eps(t), \yepsnu(t)    - \id   \rangle|
    \le \min \big\{ \totenalpha(\yepsnu(t), \thetaepsnu(t))-   \AAA \langle \ell_\eps(t), \yepsnu(t)    - \id   \rangle, \EEE \totenalpha(\yepsnu(t), \thetaepsnu(t)) \big\}
      +  C \lp^2    
  \end{equation}
  and
    \begin{equation}\label{forceestimate3}
\Vert \yepsnu  \MMM (t) \EEE -\id \Vert_{H^1(\Omega)}^2 \leq  C \totenalpha(\yepsnu(t), \thetaepsnu(t)) . 
  \end{equation}
\end{lemma}

\begin{proof}

\AAA As \EEE $W^{\rm el}(\cdot) + W^{\rm cpl}(\cdot,\theta_c)$ is nonnegative \lll due to \EEE growth condition \ZZZ  \ref{W_prefers_id}, \EEE  Poincaré's inequality, \MMMMM the fact that $ \yepsnu \in  \Wid $, \EEE   and \cite[Lemma 4.2]{FiredrichKruzik18Onthepassage} \MMM (relying on the rigidity estimate  \ZZZ \cite[Theorem~3.1]{FrieseckeJamesMueller}) \AAA imply \EEE that
\begin{align*}
\Vert \yepsnu \ZZZ (t) \EEE -\id \Vert_{H^1(\Omega)}^2 &\leq C   \Vert \nabla \yepsnu \ZZZ (t) \EEE - \Id \Vert_{L^2(\Omega)}^2 \leq C \int_\Omega \dist(\nabla \yepsnu \ZZZ (t) \EEE , SO(d))^2 \di x \\ & \leq \frac{C}{c_0} \int_\Omega W^{\rm el} (\nabla \yepsnu \ZZZ (t) \EEE ) + W^{\rm cpl} (\nabla \yepsnu \ZZZ (t) \EEE , \theta_c) \di x \leq \frac{C}{c_0}\totenalpha(\yepsnu \ZZZ (t) \EEE , \thetaepsnu \ZZZ (t) \EEE ) 
\end{align*}
for \ZZZ a.e.~$t \in I$. 
\AAA At this point,  \EEE the rest of the argument follows \rb along \MMM the lines of \cite[Lemma~3.10]{BFK}. \EEE
\end{proof}

\begin{proof}[Proof of Proposition \ref{prop:scalingtotalenergy}]
\martin For notational convenience, \EEE  we write $(y,\theta)$ in place of $(\yepsnu,\thetaepsnu)$ in the proof. \EEE The proof follows along the lines of \cite[Proposition 3.7]{RBMFLM}, where related bounds on thin domains were shown, which itself is based on \cite[Lemma 6.2]{MielkeRoubicek2020}. 
\ZZZ In contrast to the results in \cite{RBMFLM, MielkeRoubicek2020},  the internal energy density is shifted by the \ZZZ nonzero \EEE critical energy $\theta_c>0$, see \eqref{toten_shifted}, which requires nontrivial adaptations. In this regard, we frequently use \ZZZ \eqref{relationtotencritical1}--\eqref{relationtotencritical2}. \EEE The core of the proof consists in showing 
\begin{align}\label{for gronwall}
  \totenalpha (\yepsnuu(t), \thetaepsnuu(t))  
&\leq \totenalpha( \ZZZ y(0), \theta(0) \EEE)+ C \int_0^t \totenalpha (\yepsnuu(s), \thetaepsnuu(s)) \di s \notag \\ &\qquad +  \int_0^t \langle \ell_\eps(s), \partial_t y \ZZZ (s) \EEE \rangle \di s  + C \eps^2 +  C \MMM  \eps_{\alpha,\Lambda}^{2/\alpha} \EEE  \Vert \xi \Vert_{L^1(I\times\Omega)}^{1/\alpha}  
\end{align}
 for a.e.~$t \in I$. Then\rb, \ee the result follows by a Gronwall argument. We first suppose that \eqref{for gronwall} \ZZZ holds  \EEE and conclude the argument \ZZZ (Step 1). \EEE Afterwards, we show \eqref{for gronwall} by \ZZZ distinguishing \EEE the cases $\alpha = 2$ \ZZZ(Step 2) \EEE and $\alpha <2$ \ZZZ (Step 3 and 4), \EEE where as in \cite{RBMFLM} the latter is considerably more delicate.

\emph{Step 1 (Conclusion):}   For shorthand,  \EEE  we define for \ZZZ $t \in I$ \EEE
\begin{equation*}
  E^{(\alpha)}(t) \defas \totenalpha(\yepsnuu(t),\thetaepsnuu(t))  -   \langle \ell_\eps(t), \yepsnuu \ZZZ (t) \EEE -\id \rangle .
\end{equation*}
Then, by \ZZZ  an integration by parts in \eqref{for gronwall} \EEE we find that
\begin{align}\label{Ebound}
\AAA E^{(\alpha)}  \EEE (t)   &\leq  \AAA E^{(\alpha)}  \EEE(0) +  C \eps^2 +  C \MMM  \eps_{\alpha,\Lambda}^{2/\alpha} \EEE  \Vert \xi \Vert_{L^1(I\times\Omega)}^{1/\alpha}   \AAA  + C \int_0^t \totenalpha (\yepsnuu(s), \thetaepsnuu(s)) \di s \EEE   \notag\\
   &\qquad \AAA - \EEE  \int_0^t \int_\Omega \partial_s  f_\eps(s) ({\yepsnuu}  \ZZZ (s) \EEE -\id) \di x  \di s  \AAA - \EEE  \int_0^t \int_{\Gamma_N} \partial_s  g_\eps(s) ({\yepsnuu}  \ZZZ (s) \EEE -\id) \di \mathcal{H}^{d-1}   \di s .
\end{align}
\MMM By Hölder's inequality, a trace estimate, and \eqref{forceestimate3} we derive that
\begin{align}\label{forceboundcalc}
&\quad \int_0^t \int_\Omega \partial_s  f_\eps(s) ({\yepsnuu}  \ZZZ (s) \EEE -\id) \di x  \di s  + \int_0^t \int_{\Gamma_N} \partial_s  g_\eps(s) ({\yepsnuu} \ZZZ (s) \EEE -\id) \di \mathcal{H}^{d-1}   \di s \notag\\
& \leq  \int_0^t   \Vert  \partial_s  f_\eps(s) \Vert_{L^2(\Omega)} \Vert \yepsnuu(s) - \id \Vert_{L^2(\Omega)} \di s +\int_0^t   \Vert  \partial_s  g_\eps(s) \Vert_{L^2(\Gamma_N)} \Vert \yepsnuu(s) - \id \Vert_{L^2(\Gamma_N)} \di s \notag \\
&\leq C \int_0^t   \big( \Vert  \partial_s  f_\eps(s) \Vert_{L^2(\Omega)} + \Vert  \partial_s  g_\eps(s) \Vert_{L^2(\Gamma_N)} \big) \Vert \yepsnuu(s) - \id \Vert_{H^1(\Omega)} \di s \notag \\
  &\leq  C \int_0^t \big( \Vert  \partial_s  f_\eps(s) \Vert_{L^2(\Omega)} + \Vert  \partial_s  g_\eps(s) \Vert_{L^2(\Gamma_N)} \big) \sqrt{\totenalpha(\yepsnuu(s),\thetaepsnuu(s))} \di s .
\end{align}
It is elementary to check that  $\sqrt{m} \leq \eps^{-1} m + \eps$ for all  $m \geq 0$, by distinguishing the cases $m \geq \eps^2$ and  $m \leq \eps^2$. \EEE Therefore, \AAA by  \eqref{forceestimate} \EEE \III we find that \EEE
\begin{equation*}
   \sqrt{\totenalpha(\yepsnuu(s),\thetaepsnuu(s))}   \leq \frac{\totenalpha(\yepsnuu(s),\thetaepsnuu(s))}{\eps} + \eps \AAA \le   \frac{2E^{(\alpha)}  (s)}{\eps} + C\eps \EEE
\end{equation*}
for a.e.~$s \in I$. Thus,  in view of  \eqref{def:externalforces},  \eqref{forceestimate}, \eqref{Ebound}, and \eqref{forceboundcalc} we discover that  
\begin{equation*}
  E^{(\alpha)} (t) \leq  E^{(\alpha)}(0) +  C \MMM  \eps_{\alpha,\Lambda}^{2/\alpha} \EEE    \Vert \xi \Vert_{L^1(I\times\Omega)}^{1/\alpha}  + \AAA C\eps^2 \EEE  + C \int_0^t  \ZZZ \big( \AAA 1 +\EEE \Vert  \partial_s  f(s) \Vert_{L^2(\Omega)} + \Vert  \partial_s  g(s) \Vert_{L^2(\Gamma_N)} \big) \EEE (E^{(\alpha)}(s) + \eps^2) \di s .  
\end{equation*}
\lenni \AAA By \EEE 
 \eqref{forceestimate}, \eqref{eq: initial energy}, and $y(0) = y_{0,\eps}$, $\theta(0) = \theta_{0,\eps}$ a.e.~in $\Omega$ \AAA we get \EEE that $E^{(\alpha)}(0) \leq   C \eps^2$. \EEE
Then, by  \ZZZ Gronwall's inequality (in integral form), and the fact that  $ f  \in W^{1, 1}(I; L^2(\Omega; \R^d))$ \rb and \ZZZ $g  \in W^{1, 1}(I; L^2(\Gamma_N; \R^d))$,  \EEE  we derive that 
\begin{equation*}
 E^{(\alpha)} (t)
  \leq   C \eps^2  +  C \MMM  \eps_{\alpha,\Lambda}^{2/\alpha} \EEE  \Vert \xi \Vert_{L^1(I\times\Omega)}^{1/\alpha}   .
\end{equation*}
The above estimate together with \eqref{forceestimate} \MMM yields \EEE \eqref{toten_bound_schemeimprovedsec}. To conclude the proof, we need to show \eqref{for gronwall}.

\emph{Step 2 (Case $\alpha =2$):}  We first deal with the case $\alpha = 2$.
 Given $t \in I$,   we test \eqref{weak_limit_heat_equation_nu} with $\varphi(s, x) \defas \indic_{[0,t]}(s)$ resulting in
\begin{align}\label{balancereg:testwith1}
  & \mathcal{W}^{\rm in}   ( \yepsnuu(t), \thetaepsnuu(t))   - \int_0^t \int_\Omega
      \xi_{\nu, 2}^{\rm reg} (\nabla \yepsnuu, \dotnablayepsnuu, \thetaepsnuu)
      + \partial_F W^{\rm{cpl}}(\nabla \yepsnuu, \thetaepsnuu) : \dotnablayepsnuu\di x \di s  =
    \MMM \mathcal{W}^{\rm in} \EEE    (\yepsnuu(0), \thetaepsnuu(0))  + A_1, 
\end{align}  
where for convenience we have set $A_1 \defas \kappa \int_0^t \int_\Gamma (\theta_{\flat, \eps} - \thetaepsnuu)  \di \haus^{\rb d-1\ee}  \di s$ and  $\mathcal{W}^{\rm in} =  \mathcal{W}^{\rm in}_{2,0}$, \ZZZ see \eqref{toten_shifted} and \eqref{inten_lipschitz_bounds}, \EEE i.e., $   \mathcal{W}^{\rm in} ( \yepsnuu(t), \thetaepsnuu(t)) = \int_\Omega  {W}^{\rm in}   ( \yepsnuu(t), \thetaepsnuu(t)) \di x $. \ZZZ Recalling \eqref{couplenergy}, \EEE by the fundamental theorem of calculus and \eqref{Wint} we find
\begin{align}\label{fundamentalforcpl}
&\big( \mathcal{W}^{\rm cpl}(\yepsnuu(t),\theta_c) - \mathcal{W}^{\rm in}(\yepsnuu(t),\theta_c) \big) -  \big( \mathcal{W}^{\rm cpl}(\yepsnuu(0),\theta_c) - \mathcal{W}^{\rm in}(\yepsnuu(0),\theta_c)\big) \notag \\ 
 &=  \int_\Omega \int_0^t \big(  \partial_F W^{\rm cpl} (\nabla \yepsnuu(s),\theta_c)  - \partial_F W^{\rm in} (\nabla \yepsnuu(s),\theta_c)\big) : \dotnablayepsnuu \ZZZ ( s ) \EEE \di s \di x \notag \\ & = \int_0^t \int_\Omega \theta_c \partial_{F\theta} W^{\rm cpl} (\nabla \yepsnuu (s), \theta_c): \dotnablayepsnuu(s) \di x \di s  \ZZZ \eqcolon \EEE A_2.
\end{align}
Summing \eqref{balancereg:testwith1}--\eqref{fundamentalforcpl} and using \eqref{relationtotencritical1}--\eqref{relationtotencritical2} we get 
\begin{align}\label{balancereg:testwith1-new}
  & \mathcal{W}^{\rm in}_{2,\theta_c}   ( \yepsnuu(t), \thetaepsnuu(t)) + \mathcal{W}^{\rm cpl}(\yepsnuu(t),\theta_c) - \int_0^t \int_\Omega
       \xi_{\nu, 2}^{\rm reg}(\nabla \yepsnuu, \dotnablayepsnuu, \thetaepsnuu)
      + \partial_F W^{\rm{cpl}}(\nabla \yepsnuu, \thetaepsnuu) : \dotnablayepsnuu\di x \di s \nonumber\\
    &\quad \le 
    \MMM \mathcal{W}^{\rm in}_{2,\theta_c} \EEE    (\yepsnuu(0), \thetaepsnuu(0))  + \mathcal{W}^{\rm cpl}(\yepsnuu(0),\theta_c)  + A_1 + A_2 +   C\eps^2 + C \eps_{2,\Lambda}   \Vert \xi \Vert_{L^1(I\times\Omega)}^{1/2}.
\end{align} 
 Thus, we compute the sum of \eqref{energybalanceregularized} and \eqref{balancereg:testwith1-new}, and use   $\ZZZ \xi^{\rm reg}_{\nu, 2} \EEE \leq \xi$    to derive
\begin{align}
\toten_{2,\theta_c}(\yepsnuu(t), \thetaepsnuu(t))    &\leq \toten_{2,\theta_c}(\yepsnuu(0),\thetaepsnuu(0))  + \int_0^t  \langle \ell_\eps(s),   \ZZZ \partial_t y (s) \EEE \rangle \di s  + A_1  + A_2 + C\eps^2  + C \eps_{2,\Lambda}   \Vert \xi \Vert_{L^1(I\times\Omega)}^{1/2}. \label{boundsfornureg1}
\end{align}
It \ZZZ remains \EEE to bound the terms $A_1$ and $A_2$. 
In view of \eqref{def:externalforces}, \eqref{lowerboundtemperatureGamma}, the regularity of  $\mu_\flat$, \MMMMM and H\"older's inequality  \EEE we get
\begin{align}\label{dont remove}
A_1 & =  \AAA \kappa \EEE   \int_0^t \int_\Gamma (\theta_{\flat, \eps} - \thetaepsnuu)  \di \haus^{\rb d-1\ee}  \di s =  \AAA \kappa \EEE \int_0^t \int_\Gamma ( \theta_c  - \thetaepsnuu ) \di \haus^{\rb d-1\ee}  \di s  + \AAA \kappa \EEE \eps^{2}   \int_0^t \int_\Gamma \mu_\flat \di \haus^{\rb d-1\ee}  \di s \notag \\
 &\leq  \AAA \kappa \EEE  \int_0^t \int_{\Gamma  } ( \theta_c  - \thetaepsnuu )_+ \di \haus^{\rb d-1\ee}  \di s   + C \eps^{2} \Vert \mu_\flat \Vert_{L^1(I\times  \Gamma )} \leq  C \eps^2  + C \eps_{2,\Lambda} \Vert \xi \Vert_{L^1(I\times\Omega)}^{1/2} . 
\end{align} 
By \eqref{est:couplatthetacwithxi}  with $\theta=\theta_c$,  \eqref{pos_det},   \lll \ref{C_adiabatic_term_vanishes},  \EEE and Hölder's inequality   we derive that  
\begin{align}\label{strangeterm}
A_2 = \int_0^t \int_\Omega \theta_c \partial_{F\theta} W^{\rm cpl} (\nabla \yepsnuu (s), \theta_c): \dotnablayepsnuu(s) \di x \di s \leq C \eps_{2,\Lambda} \Vert \xi \Vert_{L^1(I\times\Omega)}^{1/2}.
\end{align}
Combining \eqref{boundsfornureg1}--\eqref{strangeterm}\ZZZ, we \EEE obtain \eqref{for gronwall} in the case $\alpha = 2$.

\emph{Step 3 (Cases $\alpha \in [1,2)$):}
We now show   \eqref{for gronwall}    in the case    $\alpha \in [1,2)$.
Let $ \chi(s) \defas  \frac{\alpha}{2} ( \eps^\alpha + s_+)^{2/\alpha} - \frac{\alpha}{2} \eps^2$ for $s \in \R$ and
\begin{align}\label{def:testfunctionalphanot2}
\varphi \defas \chi'(\meps) \indic_{[0,t]} = \indic_{\{\thetaepsnuu \geq \theta_c\}} \indic_{[0,t]}  (\eps^\alpha +   \meps_+ \EEE )^{2/\alpha -1}   \text{ for } \meps \defas \inten(\nabla \yepsnuu, \thetaepsnuu)-\inten(\nabla \yepsnuu, \theta_c)\ZZZ  \ \ \text{ and } t \in I, \EEE
\end{align}
\MMM where we use that $\inten$ is increasing in the temperature variable\ZZZ, see \eqref{inten_mon-new}. \EEE We show that $ \varphi$ is an admissible test function for \eqref{weak_limit_heat_equation_nu}. In this regard, \MMM we \EEE write $\chi'(\meps) = \tilde \chi(\meps) + \eps^{2-\alpha} $ for 
\begin{align*}
\tilde \chi(s) = (\eps^\alpha + s_+)^{2/\alpha-1} - \eps^{2-\alpha}.
\end{align*}
Since $\tilde \chi$ is  Lipschitz,    it suffices to show that $\meps \in L^2(I; H^1(\Omega))$.
  \MMMMM In fact, \EEE   the regularity of $(\yepsnuu,\thetaepsnuu)$ (see Definition~\ref{def:weak_solutions_regularized}),  \MMM \eqref{inten_mon}, \EEE \eqref{inten_lipschitz_bounds}, the relation 
  \begin{align}\label{def:gradmeps}
    \nabla \meps =   \partial_F W^{\rm in}(\nabla \yepsnuu, \thetaepsnuu) \nabla^2 \yepsnuu -  \partial_F W^{\rm in}(\nabla \yepsnuu, \theta_c) \nabla^2 \yepsnuu + \partial_\theta W^{\rm in} (\nabla \yepsnuu, \thetaepsnuu) \nabla \thetaepsnuu  ,
  \end{align}
   \eqref{est:coupl}, and \eqref{pos_det}  imply that $\meps \in L^2(I; H^1(\Omega))$. \AAA This shows that    $ \varphi$ \III is an \EEE admissible \EEE test function in \eqref{weak_limit_heat_equation_nu}. \AAA For later purposes, we calculate \EEE  
\begin{align}\label{def:testfunctionalphanot2grad}
 \nabla  \varphi = \indic_{[0,t]}  \nabla \chi'(\meps)
  =   \indic_{\{\thetaepsnuu \geq \theta_c\}}\indic_{[0,t]} \frac{2-\alpha}{\alpha}(\eps^\alpha +\meps_+)^{(2-2\alpha)/\alpha} \nabla \meps.
  \end{align}  
Consider the convex functional $ \mathcal{J}(m) =   \int_\Omega \chi(m) \di x   $ \MMM  on the space \EEE  $X \defas (H^{1}(\Omega))^*$. Since $W^{\rm in} (\nabla \yepsnuu,\thetaepsnuu) \in H^1(I;(H^{1}(\Omega))^* )$ and $\nabla \yepsnuu \in   L^\infty(I\times \Omega;\R^{d \times d}) \EEE \cap   H^1(I;L^2(\Omega;\R^{d \times d}))$, we get that $\meps \MMM \in \EEE H^1(I; (H^{1}(\Omega))^*)$.  
\AAA As   $\varphi \in  L^2(I; H^1(\Omega))$, we have  \EEE  $\ZZZ \chi'(\meps) \EEE \in L^2(I; X^*)$, where \MMM $X^* = H^1(\Omega)$.  Thus,  by applying the chain rule   from \cite[Proposition 3.5]{MielkeRoubicek2020} we get 
$${ \int_\Omega \chi(\meps(t)) \di x
   -  \int_\Omega \chi(\meps(0)) \di x = \ZZZ \int_0^t \EEE \big\langle  \MMMMM  \chi'(\meps  (s)) ,  \partial_t \meps (s)  \EEE \big\rangle \ZZZ \di s . \EEE } $$   
Using $\varphi = \chi'(\meps) \ZZZ \indic_{[0,t]} \EEE$ in \eqref{weak_limit_heat_equation_nu}, \ZZZ where $w$ is \EEE given by $\meps + \inten(\nabla \yepsnuu, \theta_c)$, \EEE  we discover \MMMMM by the fundamental theorem of calculus \EEE that 
\begin{align}\label{balancetestwithchi}
\int_\Omega \chi(\meps(t)) \di x
   -  \int_\Omega \chi(\meps(0)) \di x & = - \int_0^t \int_\Omega \partial_F W^{\rm in} (\nabla \yepsnuu, \theta_c ) : \dotnablayepsnuu \, \chi'(\meps) \di x \di s
 \notag  \\
   &\quad + \int_0^t \int_\Omega
            \partial_F W^{\rm{cpl}}(\nabla \yepsnuu, \thetaepsnuu) : \dotnablayepsnuu \, 
     \chi'(\meps) \di x \di s \notag \\
   &\quad + \kappa \int_0^t \int_{  \Gamma}
      (\theta_{\flat, \eps}    - \thetaepsnuu)   \chi'(\meps) \di \haus^{\rb d-1\ee} \di s   + \int_0^t \int_{\ZZZ \Omega \EEE}
    \xi^{\rm reg}_{\nu,\alpha}(\nabla \yepsnuu , \dotnablayepsnuu   , \thetaepsnuu   ) \ZZZ
    \chi'(\meps ) \EEE
       \di x \di s \notag \\
  & \quad - \int_0^t \int_\Omega \hcm(\nabla \yepsnuu, \thetaepsnuu) \nabla \thetaepsnuu \cdot \nabla \big( \chi'(\meps) \big) \di x \di s \notag \\
   &  \III \eqcolon\EEE     B_1 + B_2 + B_3 + B_4  + B_5,
 \end{align}
 where each $B_i$, $i = 1,...,5$, corresponds to \ZZZ exactly one integral \EEE in its respective order.
  
By the definition of \ZZZ $\chi$ \EEE we have  
\begin{align*} 
  \int_\Omega \chi(\meps(0)) \di x
  &= \int_\Omega \frac{\alpha}{2} \big(\eps^\alpha + \big( W^{\rm in} (\nabla \yepsnuu(0), \thetaepsnuu(0) ) - W^{\rm in} (\nabla \yepsnuu(0), \theta_c) \big)_+ \big)^{2/\alpha} -  \frac{\alpha}{2}\eps^2 \di x  \notag \\
  &\leq \lll C    \eps^2 \EEE +  \mathcal{W}^{\rm in}_{\alpha,\theta_c}(\ZZZ  y (0) , \theta (0) \EEE )    . 
\end{align*}
 In a similar fashion, using also \eqref{relationtotencritical1} we get 
\begin{align*} 
  \int_\Omega \chi(\meps(t)) \di x \ge    \mathcal{W}^{\rm in}_{\alpha,\theta_c}( \yepsnuu(t), \thetaepsnuu(t)) -  C\eps^2 - C \MMM  \eps_{\alpha,\Lambda}^{2/\alpha} \EEE   \Vert \xi \Vert_{L^1(I\times\Omega)}^{1/\alpha}
\end{align*}
for \ZZZ a.e.~$t \in I$. \EEE
Plugging this into \eqref{balancetestwithchi}, we derive 
\begin{align}\label{balancetestwithchi-neu}
\mathcal{W}^{\rm in}_{\alpha,\theta_c}(\ZZZ \yepsnuu(t), \thetaepsnuu(t) \EEE) & \le      \mathcal{W}^{\rm in}_{\alpha,\theta_c}(\ZZZ  y (0) , \theta (0) \EEE )+  C\eps^2 + C \MMM  \eps_{\alpha,\Lambda}^{2/\alpha} \EEE   \Vert \xi \Vert_{L^1(I\times\Omega)}^{1/\alpha} +   B_1 + B_2 + B_3 + B_4  + B_5
 \end{align}
 for \ZZZ a.e.~$t \in I$. \EEE
By the fundamental theorem of calculus we find for \ZZZ a.e.~$t \in I$ \EEE
\begin{align}\label{balancetestwithchi-neu2}
 \mathcal{W}^{\rm cpl}(\yepsnuu(t),\theta_c)  & =     \mathcal{W}^{\rm cpl}(\yepsnuu(0),\theta_c) + \int_0^t \int_\Omega \partial_F W^{\rm cpl} (\nabla \yepsnuu, \thetaepsnuu) : \dotnablayepsnuu\di x \di s + B_6 
 \end{align}
where
\begin{align}\label{eq: last one}
B_6 \defas - \int_0^t \int_\Omega \big( \partial_F W^{\rm cpl} (\nabla \yepsnuu, \thetaepsnuu) - \partial_F W^{\rm cpl} (\nabla \yepsnuu, \theta_c) \big) : \dotnablayepsnuu\di x \di s. 
\end{align}
In Step 4 below, we will check that
\begin{align}\label{eq: the b}
\sum_{i=1}^{ \ZZZ 6 \EEE } B_i \le C\eps^2  +  C  \MMM  \eps_{\alpha,\Lambda}^{2/\alpha} \EEE    \Vert \xi \Vert_{L^1(I\times\Omega)}^{1/\alpha} + C \int_0^t \totenalpha (\yepsnuu(s), \thetaepsnuu(s)) \di s + \int_0^t \int_{\ZZZ \Omega \EEE}
      \xi(\nabla \yepsnuu \ZZZ (s) \EEE , \dotnablayepsnuu \ZZZ (s) \EEE , \thetaepsnuu \ZZZ (s) \EEE  )
    \di x \di s.
\end{align}
\MMMMM Once this is shown, \EEE summing \eqref{balancetestwithchi-neu}, \eqref{balancetestwithchi-neu2}, and   \eqref{energybalanceregularized},  we conclude  \MMMMM that \EEE
\begin{align*}
  \totenalpha (\yepsnuu(t), \thetaepsnuu(t))  
&\leq \totenalpha(\MMMMM y(0), \theta(0)) \EEE + C \int_0^t \totenalpha (\yepsnuu(s), \thetaepsnuu(s)) \di s \\ & \qquad+  \int_0^t \langle \ell_\eps(s), \partial_t y  \ZZZ (s) \EEE \rangle \di s  + C \eps^2 +  C \MMM  \eps_{\alpha,\Lambda}^{2/\alpha} \EEE  \Vert \xi \Vert_{L^1(I\times\Omega)}^{1/\alpha}  
\end{align*}
 for a.e.~$t \in I$. This is \eqref{for gronwall} in the case $\alpha \in [1,2)$.  
 
\emph{Step 4  (Proof of \eqref{eq: the b}):}  \rb It remains to show \ee the auxiliary estimate \eqref{eq: the b} by deriving an upper bound for every term appearing on the right-hand side of \eqref{balancetestwithchi} and the term defined in \eqref{eq: last one}. More precisely, we bound $B_3$, $B_4$, $B_5$, $B_6$, and eventually $B_1 + B_2$.

We start with $B_3$. Let $C_0$ be the constant in \eqref{inten_mon}.
Given $s \in [0,t]$, consider the set $\tilde \Gamma_s \defas \{ x \in \Gamma : \meps(s,x)_+ \leq C_0  \mu_\flat (s,x) \eps^\alpha \}$ with complement $\tilde \Gamma_s^c$ such that $\Gamma = \tilde \Gamma_s \cup \tilde \Gamma_s^c$. By the definition of $\chi'$   we find
\begin{align} \label{estimatemuflat}
  \mu_\flat   \eps^\alpha \chi'(\ZZZ \meps \EEE) \leq  \MMMMM |\mu_\flat| \EEE \eps^\alpha  (\eps^\alpha + C_0  \MMMMM |\mu_\flat| \EEE \eps^\alpha)^{2/\alpha -1} \leq C (  \MMMMM |\mu_\flat| \EEE +  \MMMMM |\mu_\flat|^{2/\alpha} \EEE ) \eps^2 \quad \text{\AAA on  $\tilde \Gamma_s$.}
\end{align}
Notice that on its complement \AAA we have \EEE
\begin{align}\label{trivialestimatemuflat}
  \mu_\flat  \eps^\alpha \chi'(\ZZZ \meps \EEE) - C_0^{-1}\meps_+ \chi'(\ZZZ \meps \EEE) \leq  0 \quad \text{\AAA on  $\tilde \Gamma_s^c$.}
\end{align}
\ZZZ As shown in \eqref{inten_mon-new}, the function $W^{\rm in} (F, \cdot )$ is monotonously increasing for any $F \in GL^+(d)$ \EEE and thus $(\theta_c - \thetaepsnuu)_+ \chi'(\meps) = 0$. \MMM Moreover, by \eqref{inten_mon-new} we get \EEE  $(\thetaepsnuu - \theta_c)_+ \geq C_0^{-1} \meps_+$.   This  together with \eqref{estimatemuflat}, \eqref{trivialestimatemuflat}, and  the fact that $ \mu_\flat  \in  L^2 \lll (I\times \Gamma) \EEE$ leads to 
  \begin{align}\label{reg:est:2}
 B_3 &= \AAA \kappa \EEE   \int_0^t \int_{  \Gamma}
    (\theta_{\flat,\eps} - \thetaepsnuu)   \chi'(\meps)
  \di \haus^{\rb d-1\ee} \di s =   \AAA \kappa \EEE  \int_0^t \int_{  \Gamma  }
    \left( \eps^\alpha \mu_\flat - (\thetaepsnuu  - \theta_c)_+) \right) \chi'(\ZZZ \meps \EEE)
  \di \haus^{\rb d-1\ee} \di s \notag \\
  &\leq   \AAA \kappa \EEE   \int_0^t \int_{  \tilde \Gamma_s}
  ( \mu_\flat  \eps^\alpha  
 - C_0^{-1}    \meps_+  ) \chi'(\ZZZ \meps \EEE)
  \di \haus^{\rb d-1\ee} \di s  \leq   C  \eps^2 \int_0^t \int_{   \Gamma}
  ( \vert \mu_\flat \vert   + \vert  \mu_\flat \vert^{2/\alpha} )
  \di \haus^{\rb d-1\ee} \di s  \leq C \eps^2.
\end{align}
 Next, we address $B_4$. Recall   the definition   of $\xi^{\rm reg}_{\nu,\alpha}$ in \eqref{def_xi_alpha} and \eqref{def_xi_alpha_reg}. \lll As $\Lambda \geq 1$,  \EEE we have $(\xi^{\rm reg}_{\nu,\alpha})^{2 / \alpha} \leq (\xi^{(\alpha)})^{2 / \alpha} \leq  \Lambda \xi$. Hence,    by Young's inequality with powers $2/\alpha$ and $2/(2-\alpha)$, and constant $\frac{1}{3\Lambda}$, \ZZZ we \EEE get
\begin{align}\label{reg:est:3}
B_4 =   &\int_0^t \int_{\ZZZ \Omega \EEE}
    \xi^{\rm reg}_{\nu,\alpha}(\nabla \yepsnuu \ZZZ   \EEE , \dotnablayepsnuu \ZZZ   \EEE , \thetaepsnuu \ZZZ   \EEE )
    \chi'(\ZZZ \meps \EEE)
  \di x \di s  \notag  \\
  &   \leq  \frac{1}{3}   \int_0^t \int_{\ZZZ \Omega \EEE}
      \xi(\nabla \yepsnuu \ZZZ   \EEE , \dotnablayepsnuu \ZZZ   \EEE , \thetaepsnuu \ZZZ   \EEE )
    \di x \di s
    +   C    \int_0^t \int_{\ZZZ \Omega \EEE}
      (\eps^2+ (\meps)_+^{2/\alpha})
    \di x \di s  \nonumber  \\
  & \leq  C\eps^2  \ZZZ + \EEE C  \int_0^t
     \mathcal{W}^{\rm in}_{\alpha,\theta_c}(\yepsnuu ,\thetaepsnuu )
    \di s   + \frac{1}{3}\int_0^t \int_{\ZZZ \Omega \EEE}
      \xi(\nabla \yepsnuu \ZZZ   \EEE , \dotnablayepsnuu \ZZZ   \EEE , \thetaepsnuu \ZZZ   \EEE )
    \di x \di s,  
\end{align} 
\MMMMM where $C$ depends on $\Lambda$. \EEE  We move on to $B_5$. In view of \eqref{spectrum_bound_K}--\eqref{hcm}  and  \eqref{pos_det}, $\hcm(\nabla \yepsnuu,\thetaepsnuu)$ is uniformly bounded   from below    (in the eigenvalue sense).  
Thus, we find by \eqref {def:gradmeps}, \eqref{def:testfunctionalphanot2grad},  \eqref{est:internalatthetac},  \eqref{inten_mon}, \MMM and \eqref{pos_det} \EEE that  
\begin{align}\label{hcm_lower_bound}
  & \hcm(\nabla \yepsnuu, \thetaepsnuu) \nabla \thetaepsnuu \cdot \nabla  \big(\chi'(\meps) \big)
  \notag \\
  &\qquad \geq \indic_{[0,t] \times \{\thetaepsnuu \geq \theta_c \}}(\tfrac{2}{\alpha} - 1) (\eps^\alpha + (\meps)_+)^{2/\alpha-2}
    \big(
      C^{-1} \vert \nabla \thetaepsnuu \vert^2
      -C ( \ZZZ (\thetaepsnuu - \theta_c)_+ \EEE \wedge 1)
        \vert \nabla^2 \yepsnuu \vert
        \vert \nabla \thetaepsnuu \vert
    \big).
\end{align}
\AAA Here, we also used that $\indic_{[0,t] \times \{\thetaepsnuu \geq \theta_c \}} (\thetaepsnuu - \theta_c)_+  = \indic_{[0,t] \times \{\thetaepsnuu \geq \theta_c \}} |\thetaepsnuu - \theta_c|$. \EEE By $s \wedge 1 \leq s^{1-\MMM 2 \EEE /(\alpha p)}$ for all $s \geq 0$, \eqref{inten_lipschitz_bounds}, Young's inequality twice (firstly  with power $2$ and constant $\gamma$ and secondly  with powers $p/(p-2)$ and $p/2$), \MMMMM and \eqref{inten_mon-new} \EEE we derive that
\begin{align}\label{needed_for_imporoved_weighted_l2}
  \big((\thetaepsnuu -\theta_c)_+\wedge 1\big)
  \vert \nabla^2 \yepsnuu \vert
  \vert \nabla \thetaepsnuu \vert
 & \leq \gamma \vert \nabla \thetaepsnuu \vert^2 + \rb C_\gamma \ee  (\meps)_+^{2-\MMM 4 \EEE /(\alpha p)}  \vert \nabla^2 \yepsnuu \vert^2 \notag\\
  &\leq \gamma \vert \nabla \thetaepsnuu \vert^2
    + \AAA  C_\gamma\EEE (\meps)_+^{2(p-2)/p} (\meps)_+^{4(\alpha-\MMM 1 \EEE )/(\alpha p)}
    \vert \nabla^2 \yepsnuu \vert^2 \notag \\
  &\leq \gamma \vert \nabla \thetaepsnuu \vert^2
    + \rb C_\gamma\ee 
    \left(
      (\meps)_+^2
      + (\meps)_+^{2 - \MMM 2 \EEE /\alpha} \vert \nabla^2 \yepsnuu \vert^p
    \right),
\end{align}
where $\rb C_\gamma >0$ depends on \ZZZ $\gamma$. \EEE Choosing $\gamma \le     C^{-2}   $   with $C$ as in \eqref{hcm_lower_bound}, we can
combine \eqref{hcm_lower_bound}--\eqref{needed_for_imporoved_weighted_l2} and discover by $2/\alpha -2 \leq 0$, \AAA \eqref{toten_shifted}\asdf, and \EEE \ref{H_bounds} that 
\begin{align}\label{reg:est:5}
B_5 = &  - \int_0^t \int_\Omega
    \hcm(\nabla \yepsnuu, \thetaepsnuu) \nabla \thetaepsnuu \cdot \nabla  \big(\chi'(\ZZZ \meps \EEE) \big) \di  x \di s \notag \\
 &  \leq C \int_0^t \int_\Omega
    (\eps^\alpha + \meps_+)^{2/\alpha-2}
    (
      (\meps)_+^2
      + (\meps)_+^{2 - \MMM 2 \EEE /\alpha} \vert \nabla^2 \yepsnuu \vert^p
    )
  \di x \di s \leq C \int_0^t \int_\Omega
 (\meps)_+^{2/\alpha}
      + \vert \nabla^2 \yepsnuu \vert^p
    \di x \di s \nonumber \\
  & \leq C \int_0^t \totenalpha (\yepsnuu , \thetaepsnuu ) \di \asdf s  \EEE
\end{align}
We proceed with $B_6$. \ZZZ  As in \eqref{est:neededlater} (replacing $\int_I$ by $\int_0^t$), we derive  
\begin{align}\label{est:neededlater2}
B_6& \ZZZ \le  \left\vert \int_0^t \int_\Omega \big( \partial_F W^{\rm cpl} (\nabla \yepsnuu, \thetaepsnuu) - \partial_F W^{\rm cpl} (\nabla \yepsnuu, \theta_c) \big) : \dotnablayepsnuu\di x \di s \right\vert \notag \\ 
&\quad  \leq C \int_0^t  \mathcal{W}^{\rm in}_{\alpha,\theta_c}(\yepsnuu ,\thetaepsnuu ) \di s  +  \frac{1}{3} \int_0^t \int_\Omega  \xi(\nabla \yepsnuu \ZZZ   \EEE , \dotnablayepsnuu \ZZZ   \EEE , \thetaepsnuu \ZZZ   \EEE ) \di x \di s  \ZZZ + C  \EEE \eps^{2} \EEE + C    \eps_{\alpha,\Lambda}^{2/\alpha} \EEE \Vert \xi \Vert^{1/\alpha}_{L^1(I\times\Omega)} . \EEE
\end{align}   
We finally control $B_1 + B_2$. \rb In this regard, we \ee first show that 
\begin{align}\label{apriori_adiabatic}
\hat{B}_\alpha \defas  &\int_0^t \int_\Omega
      |( \partial_F W^{\rm cpl}(\nabla \yepsnuu, \thetaepsnuu) - \partial_F W^{\rm in}(\nabla \yepsnuu, \thetaepsnuu) ) : \dotnablayepsnuu
    \, \ZZZ  \chi'(\meps) \EEE |
    \di x \di s \notag \\
 &\leq   C    \eps^2  \ZZZ
    +  \EEE C     \int_0^t   \mathcal{W}^{\rm in}_{\alpha,\theta_c}(\yepsnuu ,\thetaepsnuu ) \di s  + \AAA \frac{1}{6} \EEE \int_0^t \int_\Omega
      \xi(  \nabla    \yepsnuu \ZZZ   \EEE  ,   \dotnablayepsnuu \ZZZ   \EEE , \thetaepsnuu \ZZZ   \EEE  )
    \di x \di s   
\end{align} 
\AAA For this, we split \EEE the proof into the cases $\alpha  \in (1,2)$ and \ZZZ $\alpha = 1$. \EEE If   $\alpha  \in (1,2)$, by \eqref{Wint}, \eqref{est:couplatthetacwithxi}, \eqref{pos_det}, \MMM \eqref{def:testfunctionalphanot2}, \EEE  Young's inequality with powers $\alpha$ and $\alpha/(\alpha-1)$, and \ref{C_adiabatic_term_vanishes}  \rb it follows \ee that  
\begin{align*}
 \hat{B}_\alpha &\leq C  \int_0^t \int_{\{ \thetaepsnuu > \theta_c \}} \left(
       ( \ZZZ (\thetaepsnuu - \theta_c)_+ \EEE \wedge 1)^{\alpha/(\alpha-1)}
      + \vert \theta_c \partial_{F\theta} W^{\rm cpl}(\nabla \yepsnuu,\theta_c) \vert^{\alpha/(\alpha-1)} 
    \right)
    (\eps^\alpha +\meps_+)^{2/\alpha-1}
  \di x \di s \\
  &  \rb\phantom{\leq}\quad + C\ee  \int_0^t \int_{\{ \thetaepsnuu > \theta_c \}}
          \xi(\nabla \yepsnuu, \dotnablayepsnuu, \thetaepsnuu)^{\alpha/2}  
       (\eps^\alpha +(\meps)_+)^{2/\alpha-1}    
    \di x \di s  \\
    & \leq C  \int_0^t \int_{\ZZZ \Omega \EEE} \left( \eps^\alpha +
         \ZZZ (m_+ \wedge 1)^{\alpha/(\alpha-1)} \EEE
      +   \xi(  \nabla    \yepsnuu,   \dotnablayepsnuu, \thetaepsnuu)^{\alpha/2}
    \right)
  (\eps^{2-\alpha} + (\meps)_+^{2/\alpha-1})
  \di x \di s,
\end{align*}
where in the last step we have also used the Lipschitz estimate in \eqref{inten_mon-new}.   Eventually, using $s \wedge 1 \leq s^{(\alpha-1)/\alpha}$ for \ZZZ $s \geq 0$, \MMMMM and \EEE
   Young's inequality with powers $2/\alpha$ and $2/(2-\alpha)$,
    we discover that  
\begin{align*}
  \hat{B}_\alpha &\leq C \int_0^t \int_{\ZZZ \Omega \EEE }
        \big( C  (\eps^2 +  \ZZZ (\meps)_+^{2/\alpha}) \EEE  
      +  \AAA \frac{1}{6} \EEE  C^{-1}\xi(  \nabla    \yepsnuu,   \dotnablayepsnuu, \thetaepsnuu)
    \big)\di x \di s \\
  &\leq   C    \eps^2
    +   C     \int_0^t   \mathcal{W}^{\rm in}_{\alpha,\theta_c}(\yepsnuu ,\thetaepsnuu ) \di s  \ZZZ + \AAA \frac{1}{6} \EEE \int_0^t \int_\Omega
      \xi(  \nabla    \yepsnuu,   \dotnablayepsnuu, \thetaepsnuu)
    \di x \di s .  
\end{align*}
This is \eqref{apriori_adiabatic} if $\alpha \in (1, 2)$.  On the other hand, \AAA for \EEE $\alpha = 1$, we get  \ZZZ by \eqref{Wint}, \eqref{est:couplatthetacwithxi}, \eqref{pos_det}, \ref{C_adiabatic_term_vanishes}, \eqref{def:testfunctionalphanot2}, \AAA and \EEE
 Young's inequality with \rb constant $ \AAA \frac{1}{6} \EEE $  that \EEE
  \begin{align*}
\hat{B}_1   & \leq C  \int_0^t \int_{\{ \thetaepsnuu > \theta_c \}} \left(
       (|\thetaepsnuu - \theta_c| \wedge 1) \xi(\nabla \yepsnuu, \dotnablayepsnuu, \thetaepsnuu)^{1/2} + \eps_{1,\Lambda} \xi(\nabla \yepsnuu, \dotnablayepsnuu, \thetaepsnuu)^{1/2}\right)
    (\eps  +\meps_+)
  \di x \di s   \\
    &  \AAA \le C  \int_0^t \int_{\{ \thetaepsnuu > \theta_c \}}  
         \xi(\nabla \yepsnuu, \dotnablayepsnuu, \thetaepsnuu)^{1/2} 
    (\eps  +\meps_+)
  \di x \di s  \EEE \leq \int_0^t \int_{\ZZZ \Omega \EEE} \left(  
         C \ZZZ (m_+)^2 \EEE  + \ZZZ C \eps^2 \EEE        +  \AAA \frac{1}{6} \EEE     \xi(  \nabla    \yepsnuu,   \dotnablayepsnuu, \thetaepsnuu) 
    \right)   \di x \di s.
\end{align*}
This  \ZZZ gives \eqref{apriori_adiabatic} \EEE in the case $\alpha =1$. 
  In \MMMMM a \EEE similar spirit to the proof of \eqref{apriori_adiabatic}, we obtain, by replacing \eqref{est:couplatthetacwithxi} with \eqref{est:couplatthetacwithxi3} in the above argument,   
  \begin{align}\label{reg:est:7}
  \lll \bar{B}_\alpha \EEE &\defas \int_0^t \int_\Omega \lll \vert \EEE \big(\partial_F W^{\rm in} (\nabla \yepsnuu, \thetaepsnuu ) - \partial_F W^{\rm in} (\nabla \yepsnuu, \theta_c ) \big) : \dotnablayepsnuu \, \ZZZ \chi'(\meps) \lll \vert \EEE \di x \di s \notag \\
   &    \leq   C    \eps^2
    +   C     \int_0^t   \mathcal{W}^{\rm in}_{\alpha,\theta_c}(\yepsnuu ,\thetaepsnuu ) \di s    \ZZZ + \EEE \AAA \frac{1}{6} \EEE \int_0^t \int_\Omega
      \xi(  \nabla    \yepsnuu,   \dotnablayepsnuu, \thetaepsnuu)
    \di x \di s .
  \end{align}
  Now, combining \eqref{apriori_adiabatic} and \eqref{reg:est:7} we get the bound 
  \begin{align}\label{b1b5}
  B_1 + B_2 \le   \lll   \hat{B}_\alpha  + \bar{B}_\alpha \leq \EEE C    \eps^2 \ZZZ
    +  \EEE  C     \int_0^t   \mathcal{W}^{\rm in}_{\alpha,\theta_c}(\yepsnuu ,\thetaepsnuu) \di s  + \frac{\AAA 1}{3}\int_0^t \int_\Omega
      \xi(  \nabla    \yepsnuu,   \dotnablayepsnuu, \thetaepsnuu) \III
    \di x   \di s   . \EEE
  \end{align}
 Eventually, collecting \eqref{reg:est:2}, \eqref{reg:est:3}, \eqref{reg:est:5}, \eqref{est:neededlater2}, and \eqref{b1b5} we get \eqref{eq: the b}, which concludes the proof. 
    \end{proof}

\MMM 
\subsection{Fine \ZZZ a priori bounds on deformation and temperature}\label{sec: a priori}
We now formulate all a priori bounds with optimal scaling in $\eps$ which are   needed \rb in order \MMM to pass to the linearized system.  \EEE

\begin{proposition}[\ZZZ Existence of solutions with fine a priori bounds\EEE]\label{lem:fineapriori} \MMM
 Suppose that   \ref{C_third_order_bounds}--\ref{C_entropy_vanishes},  \ref{W_prefers_id}, \AAA and \ref{H_prefers_id} \EEE hold.   \AAA Then, there exist some $\eps_0,\nu_0,\Lambda_0>0$ (with $\Lambda_0=1$ for $\alpha \in (1,2]$) and a constant $C>0$, independent of $\eps$, $\nu$,   such that for all $\eps \le\eps_0$, $\nu \le \nu_0$,  and  $\Lambda \ge \Lambda_0$ \EEE   there exists  a weak solution\EEE  $(\yeps, \thetaeps)$  in the sense of Definition~\ref{def:weak_formulation} satisfying  \EEE 
\begin{subequations}\label{boundres}
\begin{align}
 \esssup_{t \in I}\totenalpha(  \yeps(t),\thetaeps(t)) &\leq C\eps^2,\label{toten_bound_schemeimprovedsecfinal} \\
   \Vert \yeps - \id\Vert_{L^\infty(I;H^1(\Omega))} &\leq C\eps, \qquad  \Vert \nabla ^2\yeps \Vert_{L^\infty(I;L^{p}(\Omega))} \leq C\eps^{2/p},\label{boundres:mech} \\
     \Vert \yeps - \id \Vert_{L^\infty(I;W^{1,\infty}(\Omega))} & \leq C\eps^{2/p},\label{boundres:linfty} \\
  \Vert \thetaeps -\theta_c \Vert_{L^\infty(I;L^{2/\alpha}(\Omega))} &\leq C \eps^\alpha, \label{boundres:templ1} \\
  \int_I \int_{\Omega} \xi(\nabla \yeps, \partial_t \nabla y_\eps, \thetaeps) \di x \di t &\leq C \eps^2, \label{boundres:dissipation} \\
  \Vert \partial_t \nabla  y_\eps\Vert_{L^2(I\times \Omega)} &\leq C \eps. \label{boundres:strainrate}
\end{align}
\end{subequations}
Moreover,  for any $q \in [1,\frac{2}{\alpha} + \frac{4}{\alpha d})$ and $r \in [1,  \frac{2d+4}{\alpha d +2})$, we can find constants $C_q$ and $C_r$ independent of $\eps$ such that
\begin{subequations}\label{temp_apriori}
\begin{align}
  \Vert \thetaeps  -\theta_c \Vert_{L^q(I\times \Omega)} + \Vert m_\eps \Vert_{L^q(I\times \Omega)} &\leq C_q \eps^\alpha, \label{boundres:temperature} \\
  \Vert \nabla \thetaeps \Vert_{L^r(I \times \Omega)} + \Vert \nabla m_\eps \Vert_{L^r(I \times \Omega)} &\leq C_r \eps^\alpha,  \label{boundres:temperaturegrad} \\
  \Vert \partial_t  m_\eps \Vert_{L^1(I; H^{\MMM (d+3)/2}(\Omega)^*)} &\leq C \eps^\alpha, \label{forAubin-Lion3}
\end{align}
\end{subequations}
where $m_\eps \defas W^{\rm in}(\nabla \yeps,\thetaeps) -W^{\rm in}(\nabla \yeps,\theta_c) $.  
\end{proposition}

\begin{proof} 
\MMM It suffices to establish all a priori bounds for $\nu$-regularized solutions in the sense  Definition~\ref{def:weak_solutions_regularized}\lll, which exist due to Proposition~\ref{thm:existence_positivity_regularized}(i). \EEE
 Then, \MMMMM in view of \EEE Proposition~\ref{thm:existence_positivity_regularized}(iv), \EEE  
all bounds are preserved in the limiting passage $\nu\to 0$.
  \MMM Note, however, that by this reasoning we \III cannot \EEE guarantee that  \emph{every} weak solution in the sense of \AAA Definition \ref{def:weak_formulation} \EEE satisfies the a priori bounds, but we  only prove the existence of such a solution.

The bounds \eqref{toten_bound_schemeimprovedsecfinal} and \eqref{boundres:dissipation} have already been established in Proposition \ref{lem:fineapriori-new}, and \eqref{boundres:strainrate} has been \lll deduced \EEE in its proof, see \eqref{pompi}.   As motivated at the beginning of the section, all remaining bounds of the statement can be derived thereof by following the strategy in \cite[Lemma~6.2, Proposition~6.3]{MielkeRoubicek2020} or \cite[Section~3.4]{BFK}. \AAA We give a sketch of the proof and refer to  \cite{BFK} for details. \EEE

By \eqref{toten_bound_schemeimprovedsecfinal} and \eqref{forceestimate3}, we immediately get the first inequality in \eqref{boundres:mech} whereas the second inequality follows by \eqref{toten_bound_schemeimprovedsecfinal} \lll  and \EEE \ref{H_bounds}.
\MMM Employing \EEE Morrey's inequality \MMM we get \EEE \eqref{boundres:linfty}. \MMM Next, \EEE \eqref{boundres:templ1} follows from \eqref{toten_bound_schemeimprovedsecfinal}, \AAA \eqref{toten_shifted}, \MMMMM and  \eqref{inten_mon-new}.    \EEE Following closely the lines of  \AAA \cite[Remark~3.17 and Lemma 3.19]{BFK}, \EEE we can derive the bounds  
\begin{align*} 
\int_I \int_\Omega \frac{\vert \nabla (m_\eps)_+ \vert^2}{(1+ \eps^{-\alpha} (m_\eps)_+)^{\ell_\alpha}} \di x \di t \leq C \eps^{2\alpha} 
\end{align*}
\AAA 
with $\ell_\alpha = 1+\eta$ for $\alpha=2$ and $\ell_\alpha =   2-2/\alpha $ for $\alpha \in [1,2)$, for some $\eta>0$. \EEE    
  An \MMM interpolation \EEE  provides \eqref{boundres:temperature} and \eqref{boundres:temperaturegrad} for the positive part of the corresponding functions for $q \in [1,\frac{d+2}{d})$ and $r \in [1,\frac{d+2}{d+1})$. In the case $\alpha \in [1,2)$, we derive \rb improved bounds for a bigger range of $q$ and $r$, \AAA namely \EEE for \ee  $q \MMM \in  [1,\frac{2}{\alpha} + \frac{4}{\alpha d})\EEE$ and $r \MMM \in [1,  \frac{2d+4}{\alpha d +2})\EEE$, see \cite[Remark 3.21]{BFK} \MMM for details. \EEE Employing  \eqref{lowerboundtemperature}  and \eqref{lowerboundtemperaturegrad} together with \eqref{boundres:dissipation} \MMMMM we get \EEE  the estimates for the negative parts, first for $q = r = 2$, and then \AAA by a Sobolev embedding \EEE  also for $q \MMM \in  [1,\frac{2}{\alpha} + \frac{4}{\alpha d})\EEE$ and $r \MMM \in [1,  \frac{2d+4}{\alpha d +2})\EEE$. \MMM Finally, \EEE   \eqref{forAubin-Lion3} follows along the lines \MMM  of \cite[Theorem  3.20]{BFK} or \EEE \cite[Lemma~4.10]{RBMFLM}.
\end{proof}

\MMM \subsection{Linearization}\label{subsec: lin}  
This final subsection is entirely devoted to the proof of  Theorem~\ref{thm:linearization_positive_temp}. \EEE

\begin{proof}[Proof of Theorem~\ref{thm:linearization_positive_temp}]
The proof follows along the lines of \cite[Section~5]{BFK}, where linearized models for small temperatures have been derived. \AAA The \EEE arguments there were explicitly given by starting from the time-discrete setting. \MMM Here, we provide the  adaptations for the setting of time-continuous evolutions and  for the linearization around a positive temperature $\theta_c$.    The proof \MMM is \EEE divided into \MMM five \EEE steps. We first address the compactness properties of the rescaled temperatures and \AAA the \EEE strains. In Step 2, we derive the linearized mechanical equation \eqref{linear_evol_mech} which helps us to prove strong convergence of the rescaled strain rates in Step~3. \MMM Afterwards, \EEE we derive the linearized heat equation \eqref{linear_evol_temp} in Step~4. \MMM Eventually, uniqueness of the limit is subject of Step~5. \EEE 

\emph{Step 1 (Compactness):}
We start with a sequence of weak solutions $((y_\eps,\theta_\eps))_\eps$  \MMM satisfying the a priori bounds stated in Proposition \ref{lem:fineapriori}. \EEE  Recalling  \eqref{def_Wzero},   we first show that there exists $u \in  H^1(I; H^1_{\Gamma_D}(\Omega; \R^d))$ with $u(0) = u_0$  \lll a.e.~in $\Omega$ \EEE such that, up to possibly taking a subsequence, it holds that
  \begin{align}
    u_\eps &\to u \text{ in } L^\infty(I; L^2(\Omega; \R^d)), &
    u_\eps &\weakly u \text{ weakly in } H^1(I; H^1(\Omega; \R^d)). \label{conv:linearization1} 
  \end{align}
  By the definition of $u_\eps \MMM = \eps^{-1}(y_\eps - \id) \EEE $ and \eqref{boundres:mech}, we derive \III that \EEE
  \begin{equation}\label{up and down}
    \norm{{u_\eps}}_{\III L^\infty(I;H^1(\Omega)) \EEE } = \eps^{-1} \norm{{y_\eps}  - \id}_{\III L^\infty(I;H^1(\Omega)) \EEE} \leq C.
  \end{equation}
  Moreover, using Poincaré's inequality  \MMM and \EEE \eqref{boundres:strainrate}  we have that
  \begin{equation}\label{up and down2}
    \norm{{\partial_t u_\eps}}_{L^2(I; H^1(\Omega))}
    \leq C \norm{\partial_t\nabla { u_\eps}}_{L^2(I; L^2(\Omega))}
    = \frac{1}{\eps}  \Vert \partial_t \nabla y_\eps \Vert_{L^2(I \times \Omega)} \leq C.
  \end{equation}
  Combining \eqref{up and down}--\eqref{up and down2} we discover that $({u_\eps})_\eps$ is bounded  in $\III  H^1(I; H^1(\Omega;\R^d)) \EEE$ and thus $({u_\eps})_\eps$ is compact in $C(I; L^2(\Omega; \R^d))$ by the Aubin-Lions' theorem.
 \ZZZ  This    shows \EEE (\ref{conv:linearization1}).  
  Finally, due to (\ref{conv:linearization1}), \MMM  by the boundary condition on $\Gamma_D$ \MMM and by \EEE $u_\eps(0) = u_0$  \lll a.e.~in $\Omega$ \EEE (see \eqref{linearization_initial_conditions}),  it  follows that $u \in  H^1(I; H^1_{\Gamma_D}(\Omega; \R^d))$ with $u(0)= u_0$  \lll a.e.~in $\Omega$. \MMM We note that the convergence in \eqref{conv:linearization1} will be improved below in Step 3, see \eqref{strong_rescaled_strain_rate_comp}, which will give the desired convergence stated in \eqref{convergence:u}. \EEE

Next, we address the existence of $\mu  \in L^1(I; W^{1,1}(\Omega))$ such that, up to possibly taking a subsequence, for any  \MMM $s \in [1,\frac{2}{\alpha} + \frac{4}{\alpha d})$ and $r \in [1,  \frac{2d+4}{\alpha d +2})$ \EEE it holds that 
  \begin{align}
    {\mu_\eps} &\to \mu \text{ in }  L^s(I \times \Omega), &
    {\mu_\eps} &\weakly \mu \text{ weakly in } L^r(I; W^{1, r}(\Omega)). \label{conv:linearizationmu1} 
  \end{align} 
  The proof of \eqref{conv:linearizationmu1} relies on the a priori bounds on the internal energy in  \eqref{boundres:temperature}--\eqref{forAubin-Lion3}. The strong convergence can be derived, e.g., as in \cite[Lemma~4.2]{BFK} \MMM or \cite[Proposition 6.4]{MielkeRoubicek2020}. \EEE In particular, \MMM \eqref{conv:linearizationmu1}  implies \EEE  \eqref{convergence:mu}.

\emph{Step 2 \MMM (Linearization \EEE of the mechanical equation):}
  Let \MMM $z \in C^\infty(I \times \overline{\Omega}; \R^d)$ with $z = 0$ on $I \times \Gamma_D$. \EEE  Using the definition of $f_\eps$ \AAA and $g_\eps$ \EEE in \eqref{def:externalforces}  and dividing \AAA \eqref{weak_formulation_mechanical_eps} \EEE by $\eps$, \MMM we get \EEE 
  \begin{equation}\label{cont-mech}
\begin{aligned}
  &\eps^{-1}\intQ \pl_G
    \hypot(\nabla^2 y_\eps) \cdddot \nabla^2 z
    + \Big(
      \pl_F \felpot(\nabla y_\eps, \theta_\eps)
      + \pl_{\dot F} \disspot(\nabla y_\eps, \partial_t \nabla y_\eps, \theta_\eps)
    \Big) : \nabla z \di x \di t \\
  &\quad=  \intQ f \cdot z \di x \di t
    +  \int_I \int_{\Gamma_N} g  \cdot z \di \haus^{d-1} \di t.
\end{aligned}
\end{equation} 
        Our goal now is to show that \eqref{linear_evol_mech} arises as the limit of the above equation as $\eps \to 0$.
  By \ref{H_bounds}, \eqref{boundres:mech}, and H\"older's inequality with powers $\frac{p}{p-1}$ and $p$ we derive that
  \begin{align}\label{linmech2}
    \frac{1}{\eps} \rb\Big|\ee
      \int_I \int_\Omega \pl_G \hypot(\nabla^2 {y_\eps}) \cdddot \nabla^2 z \di x \di t
    \rb\Big|\ee
    &\leq \frac{  C}{\eps} \intQ \abs{\nabla^2 {y_\eps}}^{p-1}
      \abs{\nabla^2 z} \di x \di t \\
    &\leq  \frac{  C}{\eps} \int_I \Vert \nabla^2 {y_\eps} \Vert^{p-1}_{L^p(\Omega)}
      \norm{\nabla^2 z}_{L^p(\Omega)}\di t  
      \leq C \eps^{\frac{2(p-1)}{p} - 1} = C \eps^{1-\frac{2}{p}} \to 0, \notag
  \end{align}
  as $p \MMM \ge 2d >\EEE  2$.   We now address the elastic stress.  A Taylor expansion at \MMM $(\Id,\theta_c)$ \EEE in the spirit of \eqref{est:cpltaylorlinearize} together with \ref{W_regularity}, \ref{C_regularity}, \ref{W_prefers_id}, \eqref{boundres:linfty}, \MMM and  \eqref{est:coupl} \ZZZ  implies  that \EEE
  \begin{align}\label{all not remove0}
 &\eps^{-1}\Big\vert  \Big( \partial_F W^{\rm el}(\nabla y_\eps)+\partial_F \cplpot(\nabla y_\eps, \theta_\eps) \Big) \notag \\ & \qquad \qquad- \Big( \eps \partial_{F}^2 W^{\rm el} (\Id)  \nabla u_\eps + \eps \partial_{F}^2 W^{\rm cpl} (\Id,\theta_c) \nabla u_\eps +  \partial_{F\theta} W^{\rm cpl} (\Id,\theta_c)  (\eps^\alpha\mu_\eps \wedge 1) \Big) \Big\vert \notag \\
 &\qquad\leq C  \eps \vert \nabla u_\eps \vert^2 + C \eps^{-1} (\eps^{2\alpha} \vert\mu_\eps\vert^2 \wedge 1)
  \end{align} 
   pointwise a.e.~in $I \times \Omega$.    Due to \eqref{conv:linearization1} and \eqref{conv:linearizationmu1},
  the right-hand side of \eqref{all not remove0} converges to $0$ a.e.~in $I \times \Omega$ as $\eps \to 0$. Further\rb more\ee, $s \wedge 1 \leq s^{1/2}$ for $s \geq 0$ and \eqref{conv:linearizationmu1} imply that the right-hand side of \eqref{all not remove0} is uniformly integrable in $L^1(I \times \Omega)$. Thus, by  Vitali's convergence theorem, the left-hand side of \eqref{all not remove0} converges strongly in $L^1(I \times \Omega)$ \AAA to $0$. \EEE
  Then,  \MMM in view of \III \eqref{eq: free energy}, \EEE \eqref{Bhatt}--\eqref{alpha_dep} and \eqref{conv:linearization1}, \EEE we find that
  \begin{equation}\label{all not remove}
      \frac{1}{\eps} \intQ
      \pl_F W(\nabla {y_\eps}, \theta_\eps) : \nabla z \di x \di t 
      \to \intQ \left(\big(\partial_{F}^2 W^{\rm el} (\Id)  +   \partial_{F}^2 W^{\rm cpl} (\Id,\theta_c) \big)\nabla u  + \mathbb{B}^{(\alpha)} \mu \right): \nabla z  \di x \di t
  \end{equation}
  as $\eps \to 0 $.
For the remaining term, we note that
   by  \eqref{chain_rule_Fderiv} \III and the symmetries in \ref{D_quadratic} \EEE we have
  \begin{align}\label{eq: auch noch}
    \pl_{\dot F} \disspot(\nabla {y_\eps},   \AAA \partial_t \nabla  y_\eps, \EEE \theta_\eps) : \nabla z
    = 2 \nabla {y_\eps} (D(C_\eps, \theta_\eps) \eps \dot C_\eps) : \nabla z
    = \eps \dot C_\eps : D(C_\eps, \theta_\eps)
      (\nabla z^T \nabla {y_\eps} + (\nabla {y_\eps})^T \nabla z),
  \end{align}
  where
  \begin{align}
    C_\eps &\defas (\nabla {y_\eps})^T \nabla {y_\eps}, &
    \dot C_\eps &\defas {(\partial_t { \nabla u_\eps})}^T \nabla {y_\eps}
      + (\nabla {y_\eps})^T {\partial_t { \nabla u_\eps}}. \label{def_Ck_dotCk}
  \end{align}
  By   \eqref{boundres:mech}   and \eqref{conv:linearization1} we   see that
  \begin{align}\label{ffflater}
  \dot C_\eps \weakly 2 e(\partial_t u) \quad \text{ weakly in } L^2(I\times \Omega;\R^{d \times d}_{\rm  sym}).
  \end{align}
  Using \ref{D_bounds} we also  have   
  \begin{equation*}
    \abs{D(C_\eps, \theta_\eps) (\nabla z^T \nabla {y_\eps} + (\nabla {y_\eps})^T \nabla z)}
    \leq 2 \aC \norm{\nabla z}_{L^\infty(\Omega)} \norm{\nabla {y_\eps}}_{L^\infty(\Omega)}.
  \end{equation*}
  Up to taking a subsequence (not relabeled), we can suppose that $\nabla {y_\eps} \to \Id$ and $\theta_\eps \to \theta_c$ a.e.~in $I \times \Omega$.
  Thus, the dominated convergence theorem implies
  \begin{equation*}
    D(C_\eps, \theta_\eps) (\nabla z^T \nabla {y_\eps}
      + (\nabla {y_\eps})^T \nabla z)
    \to D(\Id, \theta_c) (\nabla z + \nabla z^T)
    = 2 D(\Id, \theta_c) \nabla z
  \end{equation*}
  strongly in $L^2(I \times \Omega;\R^{d \times d})$.
  This along with   \eqref{eq: auch noch} and   \eqref{ffflater}  leads to
  \begin{align}   \label{linmech4}
    \eps^{-1} \int_I \int_\Omega \pl_{\dot F} \disspot(\nabla {y_\eps}, \AAA \partial_t \nabla  y_\eps, \EEE \theta_\eps)
      : \nabla z \di x \di t
    \to \intQ 4 D(\Id, \theta_c) e(\partial_t u) : \nabla z \di x \di t.
  \end{align}
  Recalling the definition of $\CD$ and $\C_W$ in \eqref{def_WD_tensors}, as well as  collecting \MMM \eqref{cont-mech}, \EEE \eqref{linmech2}, \eqref{all not remove},  and \eqref{linmech4}  we conclude \MMM that \eqref{linear_evol_mech} holds. \EEE 

\emph{Step 3 (Strong convergence of the rescaled strains and strain rates):}
For the limit passage in the heat-transfer equation,  we will need  the strong convergence of the strain rates $(\partial_t \nabla u_\eps)_\eps$ in $L^2(I\times \Omega;\R^{d \times d})$ since the dissipation rate is \MMM quadratic \EEE  in $\dot F$, see \ref{D_quadratic} and \eqref{diss_rate}.
  \MMM To this end, in this step we improve \EEE the compactness in \eqref{conv:linearization1} \AAA to \EEE
 \begin{align}\label{strong_rescaled_strain_rate_comp}
    u_\eps(t) &\to u(t) \text{ in } H^1(\Omega; \R^d)
      \text{ for a.e.~} t \in I, \, \quad \text{\ZZZ and \EEE}    \quad
    {\partial_t { \nabla u_\eps}} \to \partial_t \nabla u \text{ in } L^2(I\times \Omega;\R^{d \times d}).
  \end{align} 
  \MMM Note that this, along with the Arzelà–Ascoli theorem, also shows \eqref{convergence:u}.
For convenience, for any $v \in H^1(\Omega;\R^d)$, we define
  \begin{equation*}
    \mechenl(v) \defas \frac{1}{2} \int_\Omega \CW  e(v)  : e(v) \di x,
  \end{equation*}
  where $\CW$ is as in \eqref{def_WD_tensors}.
  Let us fix an arbitrary $t \in I$.
  By the nonnegativity of $\hypot$, a Taylor expansion, \MMM \ref{W_regularity}, \AAA \ref{W_prefers_id}, \EEE \ref{C_regularity}, \EEE  and   \eqref{boundres:linfty}   we derive that
  \begin{align}\label{taylor_mechen_lpk}
  \eps^{-2} \big( &\mathcal{M}(y_\eps(t)) +   \MMM  \mathcal{W}^{\rm cpl}(\MMMMM y_\eps(t), \EEE \theta_c)  \big) \EEE     \geq \eps^{-2} \int_\Omega \elpot(\nabla {y_\eps}(t)) +W^{\rm cpl}(\nabla y_\eps(t),\theta_c)\di x \notag \\
    &\geq \frac{1}{2} \int_\Omega \big(\partial^2_F \elpot(\Id)+\partial^2_F W^{\rm cpl}(\Id,\theta_c)\big)\nabla {u_\eps}(t)
      : \nabla {u_\eps}(t) \di x  - C \int_\Omega \abs{ \MMMMM \nabla \EEE  {y_\eps}(t) - \Id} \abs{\nabla {u_\eps}(t)}^2 \di x \notag \\
    &\geq \frac{1}{2} \int_\Omega\big(\partial^2_F \elpot(\Id)+\partial^2_F W^{\rm cpl}(\Id,\theta_c)\big)\nabla {u_\eps}(t)
      : \nabla {u_\eps}(t)  \di x  - C \eps^{2/p} \int_\Omega \abs{\nabla {u_\eps}(t)}^2 \di x
  \end{align}
  for a.e.~$t \in I$.
  Consequently, \AAA by \EEE using \eqref{conv:linearization1}, by   standard lower semicontinuity arguments for integral functionals, and \AAA by \EEE the fact that $\C_W$ \ZZZ only depends on $\R^{d \times d}_{\rm sym}$ \EEE it follows that
  \begin{equation}\label{liminf_rescaled_en}
    I_1 \defas \liminf_{\eps \to 0}  \eps^{-2} \big( \mathcal{M}(y_\eps(t)) +   \MMM  \mathcal{W}^{\rm cpl}(\MMMMM y_\eps(t), \EEE \theta_c)  \big) \EEE
    \geq \liminf_{\eps \to 0} \mechenl(u_\eps(t)) \geq \mechenl(u(t))
  \end{equation}
  for a.e.~$t \in I$.
  Let $C_\eps$ and $\dot C_\eps$ be as in \eqref{def_Ck_dotCk}.
  In \eqref{ffflater} we have  seen that $\dot C_\eps \weakly 2 e(\partial_t u)$ weakly in $L^2(I\times \Omega;\R^{d \times d})$.
  This along with the definition in \eqref{diss_rate}, $\CD = 4 D(\Id, \theta_c)$, the pointwise \rb a.e.~\ee convergences of $(\nabla {y_\eps})_\eps$ and $(\theta_\eps)_\eps$, and standard lower semicontinuity arguments (see e.g.~\cite[Theorem 7.5]{FonsecaLeoni07Modern}) show
  \begin{align}
    I_2 \defas \liminf_{\eps \to 0} \eps^{-2}
      \int_0^t \int_\Omega \drate(\nabla {y_\eps},  \AAA \partial_t \nabla  y_\eps, \EEE \theta_\eps) \di x \di s
    &= \liminf_{\eps \to 0} \int_0^t \int_\Omega D(C_\eps, \theta_\eps) \dot C_\eps
      : \dot C_\eps \di x \di s \nonumber \\
    &\geq \int_0^t \int_\Omega \CD e(\partial_t u) : e(\partial_t u) \di x \di s \label{liminf_rescaled_dissrate}
  \end{align}
  for every $t \in I$.
  Our next goal is to show the reverse inequalities for the $\limsup$. Recall the definition of $\ell_\eps$ in \eqref{def:forcefunctional}.
\AAA The  \EEE energy balance in \eqref{energybalanceregularized} also holds in the setting of \AAA Definition \ref{def:weak_formulation}, see e.g.~\cite[Equation~(4.11)]{BFK}. \EEE Then, \EEE we can use the fundamental theorem of calculus to \rb derive \MMM (see also \eqref{LLLLLL} \AAA for an analogous argument) \EEE
\begin{align}\label{cont_energy_balance-eps}
& \mathcal{M}(y_\eps(t)) + \mathcal{W}^{\rm cpl}(\MMMMM y_\eps(t), \EEE \theta_c)  + \int_0^t \int_\Omega \xi(\nabla y_\eps, \partial_t \nabla y_\eps, \theta_\eps) \di x \di s\nonumber \\
&= \mathcal{M}(y_\eps(0))+ \mathcal{W}^{\rm cpl}(\MMMMM y_\eps(0), \EEE \theta_c)    +    \int_0^t
      \langle \ell_\eps(s), \partial_t  y_\eps(s) \rangle \di s \notag \\ &\quad \qquad  - \int_0^t \int_\Omega \big( \partial_F W^{\rm cpl} (\nabla y_\eps, \theta_\eps) - \partial_F W^{\rm cpl} (\nabla y_\eps, \theta_c) \big) : \partial_t \nabla y_\eps \di x \di s
\end{align}  
for a.e.~$t \in I$.   \AAA Notice that $\mathbb{B}^{(\alpha)} \neq 0$  only for $\alpha = 1$ and that in the case $\alpha = 1$ we have $\mu \in L^2(I  \times   \Omega)$, see \eqref{conv:linearizationmu1}.  Thus, by approximation (see \cite[Proposition~6.2]{virginina}), \EEE  we can use $z = \partial_t u \MMMMM \indic_{[0,t]} \EEE \in L^2(I; H^1_{\Gamma_D}(\Omega;\R^d) )$ as a test function in \eqref{linear_evol_mech}. \rb Therefore, \ZZZ using a chain rule for the convex functional $\mechenl $, \EEE  we see  for a.e.~$t \in I$ that
  \begin{align}\label{lim_energy_balance-new}
    &\mechenl(u(t)) 
    + \int_0^t \int_\Omega \big( \CD e(\partial_t u) : e(\partial_t u)  \big)  \di x \di s =  \mechenl(u_0)+ \int_0^t \langle \ell (s), \partial_t {u}(s) \rangle \di s -  \int_0^t \int_\Omega \mathbb{B}^{(\alpha)} \mu : \partial_t \nabla u \di x \di s,
  \end{align}
  where we set for $s \in I$
  \begin{align*} 
\langle \ell (s), v \rangle \defas \int_\Omega f (s) \cdot v \di x + \int_{\Gamma_N} g (s) \cdot v \di \mathcal{H}^{d-1}.
\end{align*} 
  We now address the convergence of the various terms \MMM in \eqref{cont_energy_balance-eps}. \EEE
  First of all, by \eqref{conv:linearization1} and \ZZZ \eqref{def:externalforces} \EEE we \rb have\ee
  \begin{equation}\label{lindiss3}
    \frac{1}{\eps^2} \int_0^t
      \langle \ell_\eps(s), \partial_t  y_\eps(s) \rangle \di s
    = \int_0^t \langle \ell(s), \partial_t u_\eps(s) \rangle \di s
    \to \int_0^t \langle \ell(s), \partial_t {u}(s) \rangle \di s.
  \end{equation}
By \eqref{est:cpltaylorlinearize} and \eqref{boundres:linfty} we get
\begin{align}\label{needstrongl2conv}
&\eps^{-1} \big| \big( \partial_F W^{\rm cpl} (\nabla y_\eps, \theta_\eps) - \partial_F W^{\rm cpl} (\nabla y_\eps, \theta_c) \big)  -  \partial_{F\theta} W^{\rm cpl} (\nabla y_\eps, \theta_c) ( \eps^\alpha \mu_\eps  \wedge 1) \big|   \leq \eps^{-1} C  ( \eps^{2\alpha} \mu_\eps^2 \wedge 1 )  
\end{align}
pointwise a.e.~in $I \times \Omega$. 
 \MMM Due to   \eqref{conv:linearizationmu1}, \EEE
  the right-hand side of \eqref{needstrongl2conv} converges to $0$ a.e.~in $I \times \Omega$ as $\eps \to 0$.  \MMMMM The fact that \EEE  $ \MMM t \EEE \wedge 1 \leq t^{1/(2\alpha)}$ for $\MMM t \EEE \geq 0$ and \eqref{conv:linearizationmu1} for $s \MMM  > \EEE  2 \alpha^{-1}$ imply that the right-hand side of \eqref{needstrongl2conv} is uniformly integrable in $L^2(I \times \Omega)$.   Thus, by  Vitali's convergence theorem \rb the \ee left-hand side of \eqref{needstrongl2conv} converges strongly in $L^2(I \times \Omega)$ \MMMMM to $0$. \EEE
\MMMMM We conclude \EEE by \eqref{conv:linearization1}, weak-strong convergence, \III \ref{C_adiabatic_term_vanishes},   \eqref{boundres:linfty}, \EEE and \eqref{alpha_dep} that
\begin{align}\label{lindiss2XXX}
  \lim\limits_{\eps \to 0}\eps^{-2}\int_0^t \int_\Omega \big( \partial_F W^{\rm cpl} (\nabla y_\eps, \theta_\eps) - \partial_F W^{\rm cpl} (\nabla y_\eps, \theta_c) \big) : \partial_t \nabla y_\eps \di x \di s = \int_0^t \int_\Omega \mathbb{B}^{(\alpha)} \mu : \partial_t \nabla u \di x \di s \ZZZ. \EEE
\end{align}
 We \MMM now come to the term \ZZZ  $\mathcal{M}(y_\eps(0)) + \mathcal{W}^{\rm cpl}(\MMMMM y_\eps(\rb 0\ZZZ), \EEE \theta_c)$. \EEE We   note that \MMM the second bound in \EEE \ref{H_bounds} and \ref{H_prefers_id}    lead to  
\begin{equation}\label{H_upper_bound_spec}
  \abs{H(G)} \leq  \aC  \abs{G}^p \quad \text{ for all } G \in \R^{d \times d \times d}.
\end{equation}
Hence, for the  second-gradient  term we derive by \eqref{H_upper_bound_spec}, $u_0 \in W^{2, p}(\Omega; \R^d)$, and $p > 2$ that
  \begin{equation*}
    \eps^{-2} \Big|
      \int_\Omega \hypot(\eps \nabla^2 u_0) \di x
    \Big|
    \leq C \eps^{p-2} \int_\Omega \abs{\nabla^2 u_0}^p \di x
    \leq C \eps^{p-2} \to 0 \rb \qquad \text{as } \eps \to 0.\ee
  \end{equation*}
   The convergence of the elastic energy follows similarly to the Taylor expansion in \eqref{taylor_mechen_lpk}, where we can \MMM replace all  inequalities by equalities \EEE due to the definition of the initial datum in \eqref{linearization_initial_conditions}. \rb More precisely\MMM, we get \EEE 
  \begin{equation}\label{lindiss1}
    \lim_{\eps \to 0} \eps^{-2} \big( \mathcal{M}(y_\eps(0))+ \mathcal{W}^{\rm cpl}(\MMMMM y_\eps(0), \EEE \theta_c) \big) = \mechenl(u_0).
  \end{equation}  
  Combining \eqref{cont_energy_balance-eps}--\eqref{lim_energy_balance-new}, \MMMMM and \EEE the convergences \eqref{liminf_rescaled_en}, \eqref{liminf_rescaled_dissrate}, \eqref{lindiss3},   \eqref{lindiss2XXX},  and \eqref{lindiss1}, we \ZZZ discover that \EEE 
  \begin{equation*}
  \begin{aligned}
    \mechenl(u(t))
      + \int_0^t \int_\Omega &   \CD e(\partial_t u) : e(\partial_t u)    \di x \di s
    = \mechenl(u_0)
      + \int_0^t \langle \ell(s), \partial_t {u}(s) \rangle \di s -  \int_0^t \int_\Omega \mathbb{B}^{(\alpha)} \mu : \partial_t \nabla u \di x \di s \\
    &\ge I_1 + I_2    
    \ge \mechenl(u(t))
         + \int_0^t \int_\Omega    \CD e(\partial_t u) : e(\partial_t u)    \di x \di s. 
  \end{aligned}
  \end{equation*}
  Thus, all inequalities in \eqref{liminf_rescaled_en} and \eqref{liminf_rescaled_dissrate} are equalities.
  In particular, we derive for a.e.~$t \in I$
  \begin{align}
    \lim_{\eps \to 0} \frac{1}{2} \int_\Omega
      \CW e({u_\eps}(t)) : e({u_\eps}(t)) \di x
    &= \frac{1}{2} \int_\Omega \CW e(u(t)) : e(u(t)) \di x, \label{rescaled_diss_conv1} \\
    \lim_{\eps \to 0}  \frac{1}{\eps^2}  \int_0^t \int_\Omega
      \drate(\nabla {y_\eps}, \AAA \partial_t \nabla  y_\eps, \EEE \theta_\eps) \di x \di s
    &= \int_0^t \int_\Omega 4 D(\Id,0) e(\partial_t u) : e(\partial_t u) \di x \di s, \label{rescaled_diss_conv2}
  \end{align}
  where we also used the definition of $\CD$ in \eqref{def_WD_tensors}.

We now address \eqref{strong_rescaled_strain_rate_comp}.
  Strong convergence \AAA of \EEE $(u_\eps(t))_\eps$ in $H^1(\Omega;\R^d)$ for a.e.~$t \in I$, i.e.,  the first part of \eqref{strong_rescaled_strain_rate_comp},  follows directly from \eqref{rescaled_diss_conv1}, Korn's and Poincar\'e's inequality, and the fact that $\CW$ is positive definite on $\R^{d \times d}_{\rm sym}$, \ZZZ see \eqref{positivedefiniteness}. \EEE  For the second part of  \eqref{strong_rescaled_strain_rate_comp}, we will first show strong \ZZZ $L^2(I\times \Omega)$-convergence  \EEE of $(\dot C_\eps)_\eps$ defined in \eqref{def_Ck_dotCk}:
  by \ref{D_bounds} we estimate
  \begin{align*}
    &\ac \intQ \abs{\dot C_\eps - 2 e(\partial_t u)}^2 \di x \di t
    \leq \intQ D(C_\eps, \theta_\eps) (\dot C_\eps - 2e(\partial_t u))
      : (\dot C_\eps - 2 e(\partial_t u)) \di x \di t \\
    &\quad=  \eps^{-2}  \intQ \drate(\nabla {y_\eps}, \AAA \partial_t \nabla  y_\eps, \EEE \theta_\eps) \di x \di t
      - 2 \intQ 2 D(C_\eps, \theta_\eps) e(\partial_t u) : \dot C_\eps \di x \di t \\
    &\phantom{\quad=}\quad + \intQ 4D(C_\eps, \theta_\eps) e(\partial_t u) : e(\partial_t u) \di x \di t.
  \end{align*}
  By \eqref{rescaled_diss_conv2} for $t = T$, the pointwise convergence of $(\nabla {y_\eps})_\eps$ and $(\theta_\eps)_\eps$ to $\Id$ and $\theta_c$, respectively \AAA (see \eqref{boundres:mech} and \eqref{boundres:temperature}), \EEE and the already shown weak convergence of $\dot C_\eps$ towards $2 e(\partial_t u)$ \III (see~\eqref{ffflater}), \EEE we see that the above derived upper bound converges to $0$ as $\eps \to 0$.
  Then, the desired strong convergence of $({\partial_t { \nabla u_\eps}})_\eps$ is \rb shown \ee as follows: by \MMMMM using  Korn's inequality, \EEE  \eqref{conv:linearization1}, and \AAA  \eqref{boundres:linfty}  \EEE we get
  \begin{align*}
    &\intQ \abs{{\partial_t { \nabla u_\eps}} - \partial_t \nabla u}^2 \di x \di t
    \leq C \intQ \abs{\sym({\partial_t { \nabla u_\eps}} - \partial_t \nabla u)}^2 \di x \di t \\
    &\quad\leq C \intQ \abs{\dot C_\eps - 2 e(\partial_t u)}^2 \di x \di t
      + C \intQ \abs{\nabla {y_\eps} - \Id}^2 \abs{{\partial_t { \nabla u_\eps}}}^2 \di x \di t \\
    &\quad\leq C \intQ \abs{\dot C_\eps - 2 e(\partial_t u)}^2 \di x \di t
      + C \eps^{4/p} \intQ \abs{{\partial_t { \nabla u_\eps}}}^2 \di x \di t \to 0.
  \end{align*}
  This concludes the proof \MMM of  \eqref{strong_rescaled_strain_rate_comp}. \EEE

\emph{Step 4 \MMM (Linearization \EEE of the \MMM heat-transfer \EEE equation):} Let now $\vphi \in C^\infty(I \times \overline \Omega)$ with $  \varphi(T) = 0$. We first note that the regularity of $y_\eps$, see Definition~\ref{def:weak_formulation}, implies that $\nabla y_\eps \in C(I;L^2(\Omega))$ and we have $\nabla y_\eps(0) = \nabla y_{0,\eps}$ a.e.~in $\Omega$.  \MMM An integration by parts implies \AAA that
\begin{align*}
\intQ \eps^{-\alpha}\inten(\nabla y_\eps, \theta_c) \partial_t \vphi   \di x \di t & = - \intQ \eps^{-\alpha} \frac{\rm d}{\rm dt}  \inten(\nabla y_\eps, \theta_c)  \vphi   \di x \di t   - \int_\Omega \eps^{-\alpha}\inten(\nabla y_\eps(0), \theta_c)  \vphi  (0) \di x  ,
\end{align*}
where, using \eqref{Wint}, the integrand of the first term on the right-hand side is given by
$$\eps^{-\alpha}\frac{\rm d}{\rm dt}  \inten(\nabla y_\eps, \theta_c)    =   \eps^{1-\alpha}\partial_F W^{\rm cpl}(\nabla y_\eps, \theta_c) :\partial_t \nabla u_\eps  -   \eps^{1-\alpha} \theta_c \partial_{F\theta} W^{\rm cpl}(\nabla y_\eps, \theta_c) :\partial_t \nabla u_\eps  .  $$
Here, we note that    $ \frac{\rm d}{\rm dt}   W^{\rm in} (\nabla y_\eps,\theta_c) \in L^2(I \times \Omega)$ by \eqref{est:coupl}, \MMMMM \ref{C_adiabatic_term_vanishes}, \EEE \eqref{boundres:mech}, \MMMMM \eqref{boundres:linfty}, \EEE and \eqref{boundres:strainrate}. \MMM Define \EEE $m_\eps \defas W^{\rm in}(\nabla y_\eps,\theta_\eps) -W^{\rm in}(\nabla y_\eps,\theta_c) $.
Using $\varphi$ as a test function \AAA in \eqref{weak_limit_heat_equation_eps},  \EEE dividing the equation by $\eps^\alpha$, \MMM and plugging in the \EEE previous \MMMMM equations \MMM we deduce \EEE
\begin{align}\label{cont-heat-proof}
  &\rb\intQ \ee \hcm(\nabla y_\eps, \theta_\eps) \nabla \mu_\eps \cdot \nabla \vphi - \eps^{-\alpha}m_\eps\partial_t \vphi   \di x \di t  \MMM - \EEE  \intQ \eps^{1-\alpha} \theta_c \partial_{F\theta} W^{\rm cpl}(\nabla y_\eps, \theta_c) :\partial_t \nabla u_\eps \vphi  \di x \di t \notag \\
    &\rb\qquad - \ee \intQ\big(
     \eps^{-\alpha} \drate^{(\alpha)}(\nabla y_\eps, \partial_t \nabla y_\eps, \theta_\eps)
      + \eps^{1-\alpha} \big( \pl_F \cplpot(\nabla y_\eps, \theta_\eps)  - \partial_F W^{\rm cpl}(\nabla y_\eps,\theta_c) \big): \partial_t \nabla u_\eps
    \big) \vphi \di x \di t  
     \notag \\
  & \rb\quad
    = \ee\kappa \int_I \int_{\Gamma} (\mu_\flat - \mu_\eps) \vphi \di \haus^{d-1} \di t
     +\int_\Omega \eps^{-\alpha} \big(\inten(\nabla y_{0,\eps}, \theta_{0,\eps})  - \inten(\nabla y_{0,\eps}, \theta_c) \big)\, \vphi(0) \di x.
\end{align}
  We will now pass to the limit $\eps \to 0$ in each term above.
  Recall \eqref{inten_mon}, i.e., that $c_V(F, \theta) = \partial_\theta W^{\rm in} (F,\theta)= -\theta \pl_\theta^2 \cplpot(F, \theta)$ for any $F \in  GL^+(d) $ and  \MMM $\theta > 0$.  \EEE 
 By a change of variables and \eqref{inten_mon} we find that
  \begin{align*} 
    \rb\Big\vert\ee\eps^{-\alpha} \big(\inten(\nabla {y_\eps}, \theta_\eps) - \inten(\nabla {y_\eps}, \theta_c )\big)-
      \int_0^{\mu} c_V(\nabla {y_\eps}, \theta_c +\eps^{\alpha} s) \di s
      \rb\Big \vert\ee \leq C|{\mu_\eps} - \mu|
  \end{align*}
  pointwise a.e.~in $I \times \Omega$.  
    Due to the \rb pointwise \ee convergence of $(\nabla {y_\eps})_\eps$ to $\Id$ (see \eqref{boundres:linfty}) and  \MMM the \EEE $L^1(I\times \Omega)$-convergence of $({\mu_\eps})_\eps$ to $ \mu$ (see \eqref{conv:linearizationmu1}), \lenni we derive that \EEE
  \begin{equation*} 
   \lim\limits_{\eps \to 0} \int_I \int_\Omega \eps^{-\alpha} m_\eps \partial_t \vphi \di x \di t
   = \intQ c_V(\Id,\theta_c) \mu \partial_t \vphi \di x \di t
    = \intQ \bar c_V \mu \partial_t \vphi \di x \di t.
  \end{equation*}  
   \martin In a similar fashion, \III by \eqref{linearization_initial_conditions} \EEE we get \EEE
  \begin{equation*}
  \lim\limits_{\eps \to 0} \int_\Omega \eps^{-\alpha} \big(\inten(\nabla y_{0,\eps}, \theta_{0,\eps})  - \inten(\nabla y_{0,\eps}, \theta_c) \big)\, \vphi(0) \di x = \int_\Omega \bar c_V \MMM \mu_0  \EEE \vphi(0) \di x.
  \end{equation*}
Notice that $\rb(\ee\hcm(\nabla {y_\eps}, \theta_\eps)\rb)\ee_\eps$  is uniformly bounded due to \eqref{spectrum_bound_K}--\eqref{hcm} and \eqref{boundres:linfty}, and that $(\nabla {y_\eps})_\eps$ and $(\theta_\eps)_\eps$ converge to $\Id$ and $\theta_c$ for a.e.~$(t,x)\in I \times \Omega$, respectively (see \eqref{boundres:linfty} \AAA  and \eqref{boundres:temperature}). \EEE
Combining these facts with  \eqref{conv:linearizationmu1} \AAA and a trace \III estimate, \EEE we find that 
  \begin{align*}
    &\int_I  
      \int_\Omega \hcm(\nabla {y_\eps}, \theta_\eps) \nabla {\mu_\eps} \cdot \nabla \vphi \di x \di t
    + \kappa\int_I  \int_{ \Gamma} {\mu_\eps} \vphi \di \haus^{d-1} \di t  \to \int_I   \int_\Omega  \mathbb{K}(\theta_c)  \nabla \mu \cdot \nabla \vphi \di x \di t
      + \kappa  \MMMMM \int_I \EEE \int_\Gamma   \mu \vphi \di \haus^{d-1} \di t,
  \end{align*}
  \rb as $\eps \to 0$, \ee where  $\mathbb{K}(\theta_c)$   is defined in \eqref{spectrum_bound_K}.
By \eqref{est:couplatthetac2},  \eqref{boundres:mech},  \eqref{boundres:temperature},  \eqref{conv:linearization1},  $\rb t \ee \wedge 1 \le t^{s/2}$ for some
 $s \in (1,  \frac{d+2}{d})$, and the Cauchy-Schwarz   inequality we \MMM derive \rb that \EEE
  \begin{align*}
 &  \Big| \intQ \eps^{1-\alpha} \big( \pl_F \cplpot(\nabla y_\eps, \theta_\eps)  - \partial_F W^{\rm cpl}(\nabla y_\eps,\theta_c) \big): \partial_t \nabla u_\eps
    \, \vphi \di x \di t \Big| \\   
   &   \leq \eps^{1-\alpha} \int_I \int_\Omega
      C (\vert\theta_\eps -\theta_c \vert\wedge 1) (1 + \MMM \abs{\nabla  {y_\eps}  }) \EEE
      \abs{{\partial_t { \nabla u_\eps}}} |\varphi| \di x \di t \nonumber \\
    &\leq C \eps^{1-\alpha} \Vert \vert \theta_\eps-\theta_c\vert^{s/2} \Vert_{L^2(\AAA I \times \Omega)}
      \Vert {\partial_t { \nabla u_\eps}} \Vert_{L^2(\AAA I \times \Omega)} \Vert \varphi \Vert_{L^\infty(\AAA I \times \Omega)}
  \\&   \leq C \eps^{1 - \alpha + \alpha s / 2} \Vert {\mu_\eps} \Vert^{s/2}_{L^s(\AAA I \times \Omega)}   \Vert {\partial_t { \nabla u_\eps}} \Vert_{L^2(\AAA I \times \Omega)} \Vert \varphi \Vert_{L^\infty(\AAA I \times \Omega)}  \to \rb 0 \ee
  \end{align*}
  \rb as $\eps \to 0$, \ee
  where we have used that $s > \frac{2(\alpha-1)}{\alpha}$.
\MMM  Next,  \EEE by \eqref{def_Ck_dotCk},  \III the \EEE second convergence in \eqref{strong_rescaled_strain_rate_comp},  \eqref{diss_rate}, \MMM \eqref{def_WD_tensors}, \eqref{def_xi_alpha}, \EEE   and the continuity of $D$ one can show for $\alpha = 2$ that  
  \begin{equation*}
      \int_I \int_\Omega
      \eps^{-\alpha} \MMM \xi^{(\alpha)} \EEE (\nabla {y_\eps},  \AAA \partial_t \nabla  y_\eps, \EEE \theta_\eps) \vphi =  \int_I \int_\Omega
      \eps^{-\alpha} \drate(\nabla {y_\eps},  \AAA \partial_t \nabla  y_\eps, \EEE \theta_\eps) \vphi
    \to \intQ \CD e(\partial_t u) : e(\partial_t u) \vphi \di x \di \rb t\ee
  \end{equation*}
  \rb as $\eps \to 0$. \ee
  For $\alpha < 2$  instead, the term vanishes as $\eps \to 0$ due to $\rdrate \leq \drate$, \eqref{diss_rate}, and \eqref{boundres:dissipation}.
\MMM Lastly, \EEE using \eqref{boundres:linfty}, \MMM \eqref{Bhatt}, \EEE and \ref{C_adiabatic_term_vanishes}, we see by a Taylor expansion 
\begin{align*}
\vert \eps^{1-\alpha} \theta_c \partial_{F\theta} W^{\rm cpl}(\nabla y_\eps, \theta_c) - \theta_c \hat{\mathbb{B}} \vert \leq C \vert \nabla y_\eps - \Id \vert \to 0
\end{align*}
for a.e.\ $(t,x) \in I \times \Omega$.  
Thus, \MMM  recalling \EEE \eqref{boundres:linfty}, we can use the dominated convergence theorem and weak-strong convergence to conclude that
\begin{align*}
 \intQ \eps^{1-\alpha} \theta_c \partial_{F\theta} W^{\rm cpl}(\nabla y_\eps, \theta_c) :\partial_t \nabla u_\eps \vphi  \di x \di t \to  \intQ \theta_c \hat{\mathbb{B}} :\partial_t \MMM \nabla u \EEE \vphi  \di x \di \rb t \ee \AAA = \intQ \theta_c \hat{\mathbb{B}} :\AAA e(\partial_t u) \EEE \vphi  \di x \di \rb t \ee \EEE
\end{align*}
\rb as $\eps \to 0$. \AAA Here, we used the symmetry of $\hat{\mathbb{B}}$, see the discussion before \eqref{positivedefiniteness}. \EEE
  Collecting all convergences and recalling the definition of  $\CD^{(\alpha)}$   in \eqref{alpha_dep}, \ZZZ the limit of \eqref{cont-heat-proof} is precisely \eqref{linear_evol_temp}. \EEE

  \MMM 
\emph{Step 5 (Uniqueness):} In this final step, we \MMMMM prove   the \EEE uniqueness  of the limiting evolution  $(u,\mu)$. Once this is shown, every subsequence of $(u_\eps,\mu_\eps)_\eps$ considered above converges to the same limit, so that \eqref{convergence:u} and \eqref{convergence:mu} actually hold for any sequence by Urysohn's subsequence principle.

\MMM 
In the case $\alpha \in (1,2]$, uniqueness follows by \cite[Lemma 5.6]{BFK} which relies on the observation that \EEE  the mechanical equation in Definition~\ref{def:weak_form_linear_evol} \MMMMM does not depend on $\mu$, \EEE and \MMM allows for applying \EEE   standard theory for parabolic equations. \MMM The case $\alpha = 1$, however,  is \MMMMM more subtle \EEE due to the nontrivial coupling of \eqref{linear_evol_mech} and \eqref{linear_evol_temp}.  \MMM We therefore provide a detailed argument. As a preparation, we notice  \EEE that $\CD^{(\alpha)}= 0$ \MMM by \EEE \eqref{alpha_dep}, and that we can use $z \in L^2(I; H^1_{\Gamma_D}(\Omega;\R^d)\III ) \EEE$ and $\varphi \in C^\infty(I; H^1(\Omega) )$ with $\varphi(T)=0$  as test functions in \eqref{linear_evol_mech} and \eqref{linear_evol_temp}, respectively, due to \MMM the density of $C^\infty(\Omega)$ in $H^1(\Omega)$, \EEE and the regularity $u\in H^1(I; H^1_{\Gamma_D}(\Omega;\R^d))$ and $\mu \in L^2(I; H^1(\Omega))$, \AAA see also \cite[Proposition~6.2]{virginina}. \EEE

\MMM As an auxiliary step, we address the regularity of $\mu$ by showing  \EEE $\partial_t \mu \in L^2(I; (H^1(\Omega))^*)$. \MMM  To this end, we define for \III a.e.~$s \in I$ \MMMMM and $v \in H^1(\Omega)$ \EEE
\begin{align}\label{eqBochnerderivative}
 \langle
\sigma(s),v \rangle \defas  \int_\Omega  (\bar c_V)^{-1}\left( \mathbb{K}(\theta_c) \nabla \mu(s) \cdot \nabla v
      -  \theta_c \hat{\mathbb{B}} : e(\partial_t u(s)) \, v \right)\di x  -   \int_\Gamma (\bar c_V)^{-1} \kappa (\mu_\flat(s) - \mu(s)) v \di \haus^{d-1} .
\end{align}
 Then, $\sigma \in L^2(I;(H^1(\Omega))^*)$ since   $\mu \in L^2(I; H^1(\Omega))$ and $\partial_t \nabla u \in L^2(I \times \Omega;\R^d)$. Now, choosing $\varphi(s,x) \defas \tilde \varphi(s) v(x)$ for $\AAA \tilde\varphi \EEE \in C^\infty_c(\rb (0, T) \ee ) $ and $v \in C^\infty(\III \overline{\Omega} \EEE)$ in \eqref{linear_evol_temp}, and dividing the latter equation by $\bar c_V>0$ (see \eqref{inten_mon} and \eqref{linearized_heat_capacity}) implies that
\begin{align*}
  \int_I  \langle \sigma(s), v \rangle \tilde{\varphi}(s)    \di s   = \MMM  \int_I  \int_\Omega   \mu \,  v \di x  \, \partial_t \tilde \vphi  \di s .
\end{align*}
By the definition of the derivative in Bochner spaces, this shows  $\mu \in H^1(I;(H^1(\Omega))^*)$ with $\partial_t \mu = -\sigma$ for a.e.\ $t \in I$. \EEE Notice that an interpolation implies that $\mu \in L^2(I;H^1(\Omega)) \cap H^1(I;(H^1(\Omega))^*) \subset C(I;\MMM L^2(\Omega)) \EEE $, \AAA see  \cite[Lemma 7.3]	{Roubicek-book}. \EEE \MMM Using test functions $\III \hat \varphi  \EEE \in C^\infty(\AAA I \times \overline{\Omega} \EEE )$ with $\III \hat \varphi  \EEE (T) =0$, \MMMMM by \EEE the integration by parts $\int_I \langle \mu, \partial_t \III \hat \varphi  \EEE \rangle \di t = \int_I \langle \sigma , \III \hat \varphi  \EEE \rangle \di t - \AAA \int_\Omega \mu(0) \III \hat \varphi  \EEE(0)\di x \EEE$, and the weak formulation \eqref{linear_evol_temp} we \lll find \EEE $\mu(0) = \mu_0$  \lll a.e.~in $\Omega$. \EEE
We are now in the position to prove uniqueness in the case $\alpha = 1$. To this end, we consider two weak solutions $(u_i,\mu_i)$ for $i = 1,2$ with the same  \MMM initial datum $u_i(0) = u_0$  \lll a.e.~in $\Omega$ \EEE for $i=1,2$.   As shown above, we also have $\mu_i(0) = \mu_0$  \lll a.e.~in $\Omega$ \EEE for $i=1,2$. \AAA We plug  $v=\mu_2(s) - \mu_1(s)$  into \EEE \eqref{eqBochnerderivative},  where $(u,\mu)$ is replaced by $(u_i, \mu_i)$ for $i = 1,2$. \EEE  Subtracting the corresponding identities from each other\III, using $\sigma = - \partial_t \mu$, \EEE  and taking the integral from $a$ to $b$ for   $0<a<b<T$, we deduce 
\begin{align*}
 &\int_a^b \langle \partial_t \mu_2  -\partial_t \mu_1   , \mu_2 - \mu_1 \rangle \di t   +  \int_a^b  \int_\Gamma (\bar c_V)^{-1} \kappa (\mu_2 - \mu_1) (\mu_2 - \mu_1) \di \haus^{d-1} \di s \\
  &\rb\quad= \ee \int_a^b \int_\Omega
     (\bar c_V)^{-1}\left( - \mathbb{K}(\theta_c) \nabla (\mu_2 - \mu_1 )\cdot \nabla (\mu_2 - \mu_1)
      + \theta_c \hat{\mathbb{B}} : (e(\partial_t u_2)- e(\partial_t u_1) )\, (\mu_2 - \mu_1) \right)\di x \di s.
\end{align*}  
Using  $\mu \in L^2(I;H^1(\Omega)) \cap H^1(I;(H^1(\Omega))^*) \rb \subset \ee C(I; \MMM L^2(\Omega))\EEE$ once again, we can employ the chain rule as well as the   \MMM sign \AAA of \EEE the  second and the third term, \EEE which  \MMM in the limit $a \to 0$ and $b \to t$ yields \EEE  
 \begin{align}\label{uniquenessproof1}
 \frac{\bar c_V}{2 \theta_c}\int_\Omega \vert \mu_2(\AAA t \EEE ) - \mu_1( \AAA t \EEE ) \vert^2 \di x - \frac{\bar c_V}{2 \theta_c}\int_\Omega \vert \mu_2(0) - \mu_1(0) \vert^2 \di x  
  \MMM \le \EEE \int_0^t \int_\Omega
     \hat{\mathbb{B}} : (e(\partial_t u_2)- e(\partial_t u_1) )\, (\mu_2 - \mu_1)  \di x \di s.
\end{align}
We proceed similarly with the mechanical equation. Testing \eqref{linear_evol_mech} with the function $\varphi(s,x) = {\lenni \indic \EEE }_{[0,t]}(s) \partial_t (u_2 - u_1)$, where $(u,\mu)$ are replaced by $(u_i,\mu_i)$, subtracting the latter identities from each other, \MMM using \EEE the chain rule, and  
  $\mathbb{B}^{(1)} = \hat{\mathbb{B}}$ (see \eqref{alpha_dep}),  \MMMMM we obtain \EEE
\begin{align}\label{uniquenessproof2}
&\frac{1}{2} \int_\Omega \CW \MMMMM (e(u_2(t)- u_1(t)))  :  (e(u_2(t) -u_1(t)))  \EEE \di x - \frac{1}{2} \int_\Omega \CW (e(u_2(0)- u_1(0))) \MMM : \EEE (e(u_2(0)- u_1(0))) \di x \notag \\
 &\rb\quad= \ee \int_0^t \int_\Omega \big( \CW (e(u_2) - e(u_1)) \big) :  e (\partial_t (u_2 - u_1) ) \di x \di s \notag \\
   &\rb\quad= \ee -  \int_0^t \int_\Omega  \CD  e(\partial_t (u_2- u_1))  :  e (\partial_t (u_2 - u_1) ) \di x \di s - \int_0^t \int_\Omega  \hat{\mathbb{B}} (\mu_2 -\mu_1)  :  e (\partial_t (u_2 - u_1) ) \di x \di s  \notag \\
   &\rb\quad\leq \ee  - \int_0^t \int_\Omega  \hat{ \mathbb{B}} (\mu_2 -\mu_1)  :  e (\partial_t (u_2 - u_1) ) \di x \di s .
\end{align}
Using  $u_1(0) = u_2(0)$, $\mu_1(0) = \mu_2(0)$  \lll a.e.~in $\Omega$ \MMM and summing \eqref{uniquenessproof1}--\eqref{uniquenessproof2} yields \EEE  the uniqueness.
\end{proof}

\noindent \textbf{Acknowledgements} This work was funded by  the DFG project FR 4083/5-1 and  by the Deutsche Forschungsgemeinschaft (DFG, German Research Foundation) under Germany's Excellence Strategy EXC 2044 -390685587, Mathematics M\"unster: Dynamics--Geometry--Structure. \martin The work was further supported by the DAAD project 57600633 /  DAAD-22-03. M.K.~acknowledges support by  GA\v{C}R-FWF project 21-06569K and GA\v{C}R project  23-04766S.      
 The authors gratefully thank Tom\'{a}\v{s} Roub\'{\i}\v{c}ek for helpful discussions.

\appendix
\section{\MMM Auxiliary \AAA estimates \EEE}\label{sec:appendix}

\AAA In this section, we provide some estimates used throughout the paper. We assume the setting of \MMMMM Subsection \ref{sec:setting}. \EEE The following \AAA generalized version of Korn's inequality has been shown in   \cite[Theorem 3.3]{MielkeRoubicek2020}, and is based on \cite {Pompe03Korns, Neff}. \EEE
\begin{theorem}[Generalized  Korn's inequality]\label{pompe}
Given fixed constants $\rho>0$ and $\lambda \in (0,1]$, there exists a constant $C>0$ depending on $\Omega$, $\rho$, and $\lambda$ such that for all $u \in H^1(\Omega;\R^d)$ 
  with $u = 0$ on \ZZZ $\Gamma_D$ \EEE  and $F \in C^{0,\lambda} (\Omega; \R^{d \times d})$ satisfying $\det F \geq \rho $ in $\Omega $ and $\Vert F \Vert_{C^{0,\lambda}(\Omega)} \leq \rho^{-1}$ it holds that
  \begin{equation*}
    \left\Vert \nabla u \right\Vert_{L^2(\Omega)} \leq	 C \Vert \sym (F^T \nabla u ) \Vert_{L^2(\Omega)}.
  \end{equation*}
\end{theorem}

\ZZZ The following lemma contains \MMM helpful estimates on the \lll densities of the coupling energy and internal energy \MMMMM defined in \eqref{couplenergy} and \eqref{Wint}, respectively. \EEE   
\begin{lemma}[\ZZZ Estimates on the coupling potential\EEE]
Assume that \ZZZ \ref{C_regularity}--\ref{C_third_order_bounds} \EEE hold true.
Then, there exists a constant $C>0$ such that for all $F \in GL^+(d)$, $\dot F \in \R^{d \times d}$, and $\theta \geq 0$ it holds that
\begin{align} 
\vert \partial_F W^{\rm cpl} (F,   \theta   ) \vert + \vert \partial_F W^{\rm in}(F,  \theta   ) \vert &\leq C (  \theta    \wedge 1 ) (1 + \vert F \vert), \label{est:coupl} \\
      \vert \partial_F W^{\rm cpl} (F,  \theta   ): \dot F \vert +  \vert \partial_F W^{\rm in} (F,  \theta   ): \dot F \vert
    &\leq C (  \theta    \wedge 1) |F^{-1}| (1 + |F|) \xi(F, \dot F,   \theta   )^{\ZZZ 1/2 \EEE}, \label{avoidKorn}
  \end{align}
where $\xi$ is as in \eqref{diss_rate}.
Moreover, if we additionally assume \ref{C_more_third_order_bounds}, the following bounds \MMMMM hold \EEE true:
\begin{align}
\vert \theta\partial_{F\theta}\cplpot(F, \theta) - \theta_c \partial_{F\theta}\cplpot(F, \theta_c) \vert &\leq  C (1+ \vert F \vert) ( \vert \theta - \theta_c \vert \wedge 1), \label{est:couplatthetac} \\
\vert \partial_{F} \cplpot(F, \theta) - \partial_{F} \cplpot(F, \theta_c) \vert &\leq  C  (1+ \vert F \vert) ( \vert \theta - \theta_c \vert \wedge 1), \label{est:couplatthetac2} \\
\vert \partial_{F} W^{\rm in}(F, \theta) - \partial_{F} W^{\rm in}(F, \theta_c) \vert &\leq  C  (1+ \vert F \vert) ( \vert \theta - \theta_c \vert \wedge 1), \label{est:internalatthetac}\\
\vert \partial_{F} W^{\rm cpl}(F, \theta) - \partial_{F} W^{\rm cpl}(F, \theta_c) - \partial_{F\theta} W^{\rm cpl} (F,\theta_c)( (\theta - \theta_c) \wedge 1) \vert &\leq  C  (1+ \vert F \vert) ( \vert \theta - \theta_c \vert^2 \wedge 1), \label{est:cpltaylorlinearize}
\end{align}
as well as
\begin{align}
\vert \MMM \theta \partial_{F\theta} W^{\rm cpl}(F, \theta)  \EEE : \dot F \vert   \leq   C \vert F^{-1} \vert (1+ \vert F \vert) (\theta_c |\partial_{F\theta} W^{\rm cpl}(F,\theta_c)| + & \vert \theta - \theta_c \vert \wedge 1 ) \xi(F, \dot F, \theta)^{1/2}, \label{est:couplatthetacwithxi} \\
\vert \big( \partial_F W^{\rm cpl} (F, \theta ) - \partial_F W^{\rm cpl} (F, \theta_c) \big) : \dot F \vert \leq   C \vert F^{-1} \vert (1+ \vert F \vert) ( & \vert \theta - \theta_c \vert \wedge 1 ) \xi(F, \dot F, \theta)^{1/2}, \label{est:couplatthetacwithxi2} \\
\vert \big( \partial_F W^{\rm in} (F, \theta ) - \partial_F W^{\rm in} (F, \theta_c) \big) : \dot F \vert  \leq   C \vert F^{-1} \vert (1+ \vert F \vert) ( & \vert \theta - \theta_c \vert \wedge 1 ) \xi(F, \dot F, \theta)^{1/2}. \label{est:couplatthetacwithxi3}
\end{align}
\end{lemma}

\begin{proof}
\AAA First, \EEE \eqref{est:coupl}--\eqref{avoidKorn}  have already been shown in \cite[Lemma 3.4]{BFK} and \cite[Lemma~4.4 and Lemma~4.5]{RBMFLM}.
\AAA The \EEE proof of the \AAA other \EEE  bounds follows similarly. \EEE 
\ZZZ In particular, we will employ the fundamental theorem of calculus for $\partial_{F\theta} W^{\rm cpl}(F,\theta)$ and $\partial_{F} W^{\rm cpl}(F,\theta)$ at $\theta = 0$ which \rb is well-defined due to \ref{C_third_order_bounds}. \EEE
We now show \eqref{est:couplatthetac}.
Consider first the case \AAA $|\theta - \theta_c| \le  1$. \EEE
Then, \AAA by \EEE \ref{C_regularity}, the fundamental theorem of calculus, the second bound in \ref{C_bounds}, \MMMMM and  \ref{C_more_third_order_bounds} \EEE we have 
\begin{align*}
&\quad \vert \theta\partial_{F\theta}\cplpot(F, \theta) - \theta_c \partial_{F\theta}\cplpot(F, \theta_c) \vert
\leq  \Big|\int_{\theta_c}^{\theta} \vert \partial_{F\theta}\cplpot(F, s)| + s|\partial_{F\theta\theta} W^{\rm cpl} (F, s) \vert \di s \Big| \\
&\leq \III C_0 (1+\vert F \vert) \left( \Big|\int_{\theta_c}^\theta \frac{1}{s \vee 1} \di s\Big| + \lll \Big| \III \int_{\theta_c}^\theta \frac{s }{(s \vee 1)^2} \di s \Big| \right) \EEE \leq 2 C_0 (1 + |F|) |\theta - \theta_c|.
\end{align*}
If $\AAA |\theta -\theta_c| >  1 \EEE$, we derive from the second bound in \ref{C_bounds} that
\begin{align*}
\vert \theta\partial_{F\theta}\cplpot(F, \theta) - \theta_c \partial_{F\theta}\cplpot(F, \theta_c) \vert
&\leq \theta\vert \partial_{F\theta}\cplpot(F, \theta)| + \theta_c |\partial_{F\theta} \cplpot(F, \theta_c) \vert \\
&\leq C_0 (1 + \vert F \vert) \left(\frac{\theta}{\theta \vee 1} + \frac{\theta_c}{\theta_c \vee 1}\right)
\AAA \le \EEE \lll 2   C_0 \EEE (1 + \vert F \vert).
\end{align*}
Combining the previous two estimates \AAA gives \EEE \eqref{est:couplatthetac} for any \AAA choice \EEE $C \geq 2C_0$.
\ZZZ The proof of \eqref{est:couplatthetac2} follows along the lines of the previous \rb argument\ee, where we only have to use the second bound in \ref{C_bounds} \MMMMM and not \ref{C_more_third_order_bounds}. \EEE
\AAA Then, \EEE \eqref{est:internalatthetac} is a consequence of \eqref{est:couplatthetac}, \eqref{est:couplatthetac2}, and \eqref{Wint}. \EEE
 \AAA We proceed with \EEE \eqref{est:cpltaylorlinearize}: if $\AAA |\theta - \theta_c| \le \theta_c/2\EEE$, we find by \ref{C_regularity}, the fundamental theorem of calculus \ZZZ (twice), \EEE \ZZZ and the fact that $\partial_{F\theta\theta} W^{\rm cpl}$ can be continuously extended to $\R_+$, see \rb \ref{C_third_order_bounds}, \ee that
\begin{align*}
\vert \partial_{F} W^{\rm cpl}(F, \theta) - \partial_{F} W^{\rm cpl}(F, \theta_c) - \partial_{F\theta} W^{\rm cpl} (F,\theta_c) (\theta - \theta_c) \vert &\leq \int_{\theta_c}^\theta \int_{\theta_c}^s \vert \partial_{F\theta\theta} W^{\rm cpl} (F, t) \vert \di t \di s \\ &
\III \leq \lll C_0 \EEE (1+ \vert F \vert) \vert \theta - \theta_c \vert^2.
\end{align*}
If $\AAA |\theta - \theta_c| > \theta_c/2$, \EEE we use \AAA \eqref{est:couplatthetac2} and \EEE \lll the second bound in \EEE \ref{C_bounds} to derive that
 \begin{align*}
 \vert \partial_{F} W^{\rm cpl}(F, \theta) - \partial_{F} W^{\rm cpl}(F, \theta_c) - \partial_{F\theta} W^{\rm cpl} (F,\theta_c) \vert \leq C (1+ \vert F \vert ).
 \end{align*}
The proof of \eqref{est:couplatthetacwithxi} follows similarly to the proof of \eqref{avoidKorn}.
Along the lines of \ZZZ \cite[Equation~(3.17)]{BFK}, \EEE we can derive from the frame indifference of $W^{\rm cpl}$ (see \ref{C_frame_indifference}) that
\begin{equation}\label{Lennasagtja}
  \partial_{F\theta} W^{\rm cpl}(F, \theta) : \dot F
  = \frac{1}{2} F^{-1} \partial_{F\theta} W^{\rm cpl}(F, \theta) : (\dot F^T F + F^T \dot F).
\end{equation}
Hence, \ZZZ using \eqref{est:couplatthetac}, \ref{D_bounds}, and \eqref{diss_rate}, \EEE we derive that
\begin{align*}
 \vert \theta \partial_{F\theta}\cplpot(F, \theta) : \dot F \vert &= \vert \frac{1}{\NNN 2\EEE} \theta F^{-1} \partial_{F\theta} \cplpot(F, \theta) : \big( \dot F^T F + F^T \dot F \big) \vert \\
&\quad\leq |F^{-1}| \big(\theta_c |\partial_{F\theta} W^{\rm cpl}(F,\theta_c)| + C (1 + |F|)(|\theta - \theta_c| \wedge 1)\big) \vert   \dot F^T F + F^T \dot F \vert \\
&\quad\leq   C |F^{-1}|(1+ \vert F \vert) \big(\theta_c |\partial_{F\theta} W^{\rm cpl}(F,\theta_c)| + \vert \theta - \theta_c \vert \wedge 1 \big) \xi(F, \dot F, \theta)^{1/2}.
\end{align*}
This shows \eqref{est:couplatthetacwithxi}. \lll As in \eqref{Lennasagtja}, \EEE we obtain the identity
\begin{align*}
\big( \partial_F W^{\rm cpl} (F, \theta ) - \partial_F W^{\rm cpl} (F, \theta_c) \big) : \dot F =   \frac{1}{2}   F^{-1} \big( \partial_F W^{\rm cpl} (F, \theta ) - \partial_F W^{\rm cpl} (F, \theta_c) \big)  : \big( \dot F^T F + F^T \dot F \big) .
\end{align*}
Using \eqref{est:couplatthetac2}, \ref{D_bounds}, and \eqref{diss_rate}, we can conclude \ZZZ \eqref{est:couplatthetacwithxi2}. \EEE
\ZZZ Finally, we derive \eqref{est:couplatthetacwithxi3} simply by replacing $\partial_F W^{\rm cpl}$  with $\partial_F W^{\rm in}$ in the calculation above and by using \eqref{est:internalatthetac} instead of \eqref{est:couplatthetac2}. \EEE
\end{proof}

\begin{lemma}\label{lemma: phii-neu}
\MMM  Assume that  \AAA \ref{C_third_order_bounds}--\ref{C_Wint_regularity} \MMM hold true. \III Recall   $c_V$ defined in \eqref{inten_mon} and
let  $\theta \in \rb H^1_+(\Omega)$ and  $y \in H^2(\Omega;\R^d)$. \EEE Then,  we have that
\begin{align}
\nabla \big( c_V(\nabla y,\theta)^{-1} \big) = \rb c_V(\nabla y,\theta)^{-2} \left(\theta  \partial_{F\theta \theta} W^{\rm cpl} (\nabla y, \theta) : \nabla^2 y -  \partial_\theta^2 W^{\rm in} (\nabla y, \theta) \nabla \theta  \right)\ee, \label{gradcVinverse}
\end{align}
and \AAA it holds that \EEE
\begin{align}\label{gradcVbound}
  \vert \nabla \big( c_V(\nabla y, \theta)^{-1} \big) \vert \leq \frac{C_0}{ c_0^2 \EEE} \left(\vert \nabla \theta \vert + \theta \vert \nabla^2 y \vert 
\right) \AAA (1 + |\nabla y|) \EEE
\end{align}
\MMM a.e.~in $\Omega$, \EEE where $C_0$  and $c_0$ are \EEE as in  \MMM Subsection \EEE \ref{sec:setting}.
\end{lemma}

\begin{proof}
\III By \EEE the chain rule we have
\begin{align*}
\nabla \left( c_V(\nabla y,\theta)^{-1} \right) &= \frac{\AAA -1}{c_V(\nabla y,\theta)^2} \nabla (c_V(\nabla y,\theta)) = \frac{\AAA 1}{c_V(\nabla y,\theta)^2} \nabla \left( \theta \partial_\theta^2 W^{\rm cpl} (\nabla y, \theta) \right) \\
&= \frac{\AAA 1}{c_V(\nabla y,\theta)^2} \Big( \left( \partial_\theta^2 W^{\rm cpl} (\nabla y, \theta)   + \theta  \partial_{\theta}^3 W^{\rm cpl} (\nabla y, \theta)  \right) \nabla \theta  + \theta  \partial_{F \theta \theta} W^{\rm cpl} (\nabla y, \theta) : \nabla^2 y   \Big).
\end{align*}
Then, \ZZZ \eqref{gradcVinverse} and \eqref{gradcVbound} follow \EEE from \ZZZ \eqref{Wint},  \ref{C_third_order_bounds}--\ref{C_Wint_regularity}, \EEE and \eqref{inten_mon}.
\end{proof}

\section{Example on phase transformation in shape-memory alloys}\label{expl:shapememory}

This section is devoted to the example mentioned at the end of Section~\ref{sec:model}.   More precisely,    we discuss the free energy potential   \eqref{freeenergyexample}   modeling austenite-martensite \MMMMM transformations \EEE in shape-memory alloys. 
 In view of \eqref{eq: free energy}, the elastic and coupling energy densities are then given by 
\begin{align}\label{defelandcpl}
  W^{\rm el}(F) &= W_M(F), \quad \quad \quad  W^{\rm cpl}(F, \theta) = a(\theta) (W_A(F) - W_M(F)) + C_1 \theta (1 -  \log\theta  ),
\end{align}
  for some fixed constant $C_1 > 0$.   Here, we consider   $a \in C^3(\R_+; [0, 1])$ with $a(0) = 0$ 
such that $a$ satisfies the bounds
\begin{equation}\label{a_requirements}
  |a'(\theta)| + |a'''(\theta)| \leq \frac{C_2}{\theta \vee 1}, \quad
  |a''(\theta)| \leq \frac{C_2}{(\theta \vee 1)^2}  
\end{equation}
for all $\theta \geq 0$, where $C_2 \geq 1$ denotes a constant.
\AAA Moreover, \EEE  $W_M \in C^3(GL^+(d))$ is a frame indifferent multi-well potential modeling the martensite state while $W_A \in C^3(GL^+(d))$ denotes a frame indifferent single-well potential with minimum on $SO(d)$, corresponding to the austenite state.
 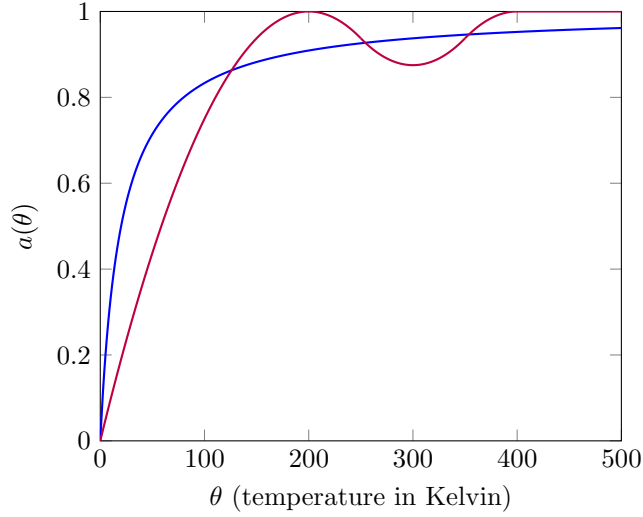
\begin{figure}
 \centering
 \begin{tikzpicture}
\begin{axis}[
    xmin = 0, xmax = 500,
    ymin = 0, ymax = 1.0,
    xlabel = {$\theta$ (temperature in Kelvin)},
    ylabel = {$a(\theta)$},]
    \addplot[
        domain = 0:500,
        samples = 200,
        smooth,
        thick,
        blue,
    ] {1- 1/(1+0.05*x)};
    \addplot[
    domain = 0:250,
    samples = 200,
    smooth,
    thick,
    purple,
] {-(x-200)*(x-200)/40000 +1};
    \addplot[
    domain = 250:350,
    samples = 200,
    smooth,
    thick,
    purple,
] {+(x-300)*(x-300)/40000 +0.875};
    \addplot[
    domain = 350:400,
    samples = 200,
    smooth,
    thick,
    purple,
] {-(x-400)*(x-400)/40000 +1};
    \addplot[
    domain = 400:500,
    samples = 200,
    smooth,
    thick,
    purple,
] { 1};
\end{axis}
\end{tikzpicture} 
\caption{Possible choices for $a$ in \eqref{a_requirements}. The blue \MMMMM curve \EEE indicates the function $a(\theta) = 1 - (1+0.05\theta)^{-1}$. The purple \MMMMM curve \EEE shows a function $a$ which \MMMMM satisfies \EEE  $  a(\theta_c) = 1$ \AAA and \EEE  $a'(\theta_c) = 0$ for  $\theta_c = 200$.  } 
 \label{fig:admissible_a}
 \end{figure} 
Furthermore, we assume that there exists a constant $C_3>0$ such that
\begin{align}\label{W_M_requirements1}
  W_M(F) \geq \frac{1}{C_3} (|F|^2 + \det(F)^{-q}) - C_3 \quad \text{for all $F \in GL^+(d)$,}
\end{align} 
where $q$ is given as in \ref{W_lower_bound}, \AAA and \EEE  we impose the existence of constants $C_4>0$ and $C_5>0$ such that
\begin{equation}\label{W_A_requirements_2}
  |W_A(F) - W_M(F)| \leq C_4 \quad \quad  \text{for all $F \in GL^+(d)$,}
\end{equation}
\begin{align}\label{W_M_requirements2}
  |\partial_F^2 (W_M(F) - W_A(F))| &\leq C_5 \quad \quad  \text{for all $F \in GL^+(d)$.}
\end{align}
Eventually, we require that there exists a constant $C_6>0$ such that
\begin{align}\label{W_A_requirements3}
  W_A(F) + C_1 \theta_c (1 - \log \theta_c) \geq \frac{1}{C_6} \dist(F, SO(d))^2, \text{ and } W_A(F) + C_1 \theta_c (1 - \log \theta_c) = 0 \quad \text{for all $F \in SO(d)$.}
\end{align}  
In the next lemma, we will appropriately tweak the constants $C_1, \dots, C_6$ in order to confirm  the compatibility of the above choices with the assumptions \III \ref{C_regularity}--\ref{C_entropy_vanishes} \EEE and  \ref{W_regularity}--\ref{W_prefers_id}.

\begin{lemma}[Compatibility with modeling assumptions]\label{lem:compatibilityshapememory}
Consider $a \in C^3(\R_+; [0, 1])$, $W_M\in C^3(GL^+(d) )$, and $W_A\in C^3(GL^+(d) )$ which satisfy the conditions \lll \eqref{a_requirements}--\eqref{W_M_requirements2} \EEE introduced above. \AAA Then, for suitable choices of $\rbb C_1\AAA, \ldots, \lll C_5 \EEE $ the following holds. \EEE

\begin{itemize}
\item[(i)]  \AAA The potentials \EEE $W^{\rm el}$ and $W^{\rm cpl}$ given in \eqref{defelandcpl} comply with \ref{W_regularity}--\ref{W_lower_bound} and \ref{C_regularity}--\ref{C_Wint_regularity}.

\item[(ii)] Assume additionally \MMMMM \eqref{W_A_requirements3} and  $  a(\theta_c) = 1$. \EEE 
Then, $W^{\rm el}$ and $W^{\rm cpl}$ also satisfy \ref{C_adiabatic_term_vanishes}--\ref{C_entropy_vanishes} and \ref{W_prefers_id}.
\end{itemize}

\end{lemma}
Figure~\ref{fig:admissible_a} provides \AAA graphs of functions \EEE $a$ \MMMMM for \EEE both cases (i) and (ii) of the previous lemma.

\begin{proof}
\AAA (i) Properties \EEE \ref{W_regularity}--\ref{W_frame_invariace} and \ref{C_regularity}--\ref{C_frame_indifference} follow from the assumed frame indifference and the imposed regularity \AAA on \EEE $W_A$, $W_M$, and $a$.
Due to \eqref{W_M_requirements1}, the lower bound in
\ref{W_lower_bound} is satisfied for the choice $c_0 = 1/C_3$ and $C_0 = C_3$. As $a(0) = 0$, \AAA \ref{C_zero_temperature} holds, \EEE where we remark that $\theta \mapsto \theta (1 - \log \theta )$ can be continuously extended \rbb by \ee $0$ \rbb at \ee $\theta = 0$.
 Notice that from  \eqref{W_M_requirements2}  we derive that
\begin{equation}\label{W_M_W_A_grad_bound}
  |\partial_F (W_M - W_A)(F)| \leq C_8 (1 + |F|) \quad \quad \text{\MMMMM for all $F \in GL^+(d)$}
\end{equation}
for \AAA some \EEE $C_8>0$.
\MMMMM From this, \EEE we show a Lipschitz bound on $W_M - W_A$: the fundamental theorem of calculus and \eqref{W_M_W_A_grad_bound} imply that there exists a constant $C_9>0$ such that
\begin{align}\label{Lipschitzboundwmwa}
 \big|(W_M(F) - W_A(F)) - (W_M(\tilde F) - W_A (\tilde F))\big| \leq C_9 (1 + |F| + |\tilde F|) |F - \tilde F|
\end{align}
for all $F, \, \tilde F \in GL^+(d)$.
\rbb Possibly increasing $C_0$ such that $C_0 \geq C_9$ holds, \ee \ref{C_lipschitz} follows from \eqref{Lipschitzboundwmwa} and the fact that $a(\theta) \AAA \in [0,1] \EEE$\ee.
We now verify \ref{C_bounds}:
  by \eqref{W_M_requirements2} and $a(\theta) \AAA \in [0,1] \EEE $, we have
\begin{align}\label{CC5}
  |\partial_F^2 W^{\rm cpl}(F, \theta)| &\leq |a(\theta)| |\partial_F^2 (W_M(F) - W_A(F) )| \leq C_5. 
  \end{align}
\rbb Moreover\ee,
\eqref{W_M_W_A_grad_bound} and \eqref{a_requirements} \rbb lead to \ee
\begin{align}\label{CC52}
  |\partial_{F \theta} W^{\rm cpl}(F, \theta)| &\leq |a'(\theta)| |\partial_F ( W_M(F) - W_A(F))|  
  \leq \frac{C_2 C_8}{\theta \vee 1} (1 + |F|)\rbb, \ee
\end{align}
for all $F \in GL^+(d)$ and $\theta \geq 0$.
\rbb Moreover, an elementary computation yields \ee
\begin{equation}\label{derivativewcpl}
\begin{aligned}
  \partial_\theta W^{\rm cpl} (F,\theta) &= a'(\theta) ( W_A(F) - W_M(F) ) - C_1 \log \theta, \\ \theta \partial^2_{ \theta} W^{\rm cpl} (F,\theta) &= \theta a''(\theta) ( W_A(F) - W_M(F) ) - C_1 .
\end{aligned}
\end{equation}
Hence, \rbb by \ee \eqref{a_requirements} and \eqref{W_A_requirements_2} \rbb it follows that \ee
\begin{align}\label{heat_capacity_lower_bound}
  - \theta \partial_\theta^2 W^{\rm cpl}(F, \theta) &= C_1 - \theta a''(\theta) (W_A(F) - W_M(F))
  \AAA \geq C_1 \EEE - \theta \frac{C_2}{(\theta \vee 1)^2} \cdot C_4 \geq C_1 - C_2 C_4
\end{align}
for all $F \in GL^+(d)$ and $\theta \geq 0$.
Similarly, we can also show that $  - \theta \partial_\theta^2 W^{\rm cpl}(F, \theta) \leq C_1 + C_2 C_4$
for all $F \in GL^+(d)$ and $\theta \geq 0$.
Consequently, \AAA using \eqref{CC5}, \eqref{CC52}, and \eqref{heat_capacity_lower_bound}, we see that \EEE \ref{C_bounds} is satisfied, as long as we choose $C_1 > C_2 C_4$, $c_0 < C_1 - C_2 C_4$, and $C_0 \geq \max\{\AAA C_5, \III C_2C_8 , \EEE  \, C_1 + C_2 C_4\}$. \AAA Next, by \EEE  \eqref{a_requirements} and \eqref{W_M_W_A_grad_bound} we find
\begin{align}\label{CC9}
  |\partial_{F\theta \theta} W^{\rm cpl}(F, \theta)|
  &\leq |a''(\theta)| |\partial_F  (W_A(F) - W_M(F))|  
  \leq \frac{C_2 C_8}{(\theta \vee 1)^2} (1 + |F|).
\end{align}
This shows that \III $\partial_{F\theta \theta} W^{\rm cpl}(F, \theta)$ \EEE can be continuously extended to $\theta = 0$ and \ref{C_third_order_bounds} holds \AAA if $C_0 \ge C_2 C_8$. \EEE In view of \eqref{inten_mon} and \eqref{derivativewcpl}, we have
\begin{align}\label{computationderivative4}
\partial^2_\theta W^{\rm in} (F,\theta) =    - (\theta a''' (\theta)+ a''(\theta) ) (W_A(F) - W_M(F)).
\end{align}  
Thus, \eqref{a_requirements} and \eqref{W_A_requirements_2} imply that
\begin{align*}
  |\partial^2_\theta W^{\rm in} (F,\theta)|
  \leq C_4 \cdot \left(\frac{C_2}{(\theta \vee 1)^2} + \theta \frac{C_2}{\theta \vee 1}\right)
  \leq 2 C_2 C_4.
\end{align*}
In particular, the  bound  in \ref{C_Wint_regularity}   holds true, after possibly increasing $C_0$ such that $C_0 \geq 2 C_2 C_4$. This concludes the proof of (i).

\AAA (ii) \EEE As $a(\theta_c) = 1$, we have $a'(\theta_c) = 0$ since $\theta_c$ is maximum point of $a$.
This immediately gives
$ \vert \partial_{F\theta} W^{\rm cpl} (F,\theta_c) \vert = 0$  and thus \ref{C_adiabatic_term_vanishes} holds.
  The bound in \ref{C_more_third_order_bounds} has already beed verified in \eqref{CC9}. \EEE
\MMMMM Using \EEE \eqref{derivativewcpl}, \eqref{computationderivative4}, and the definition of $c_V$ in \eqref{inten_mon}
we deduce that  
\begin{align*}
  \partial^2_\theta W^{\rm in} (F,\theta)  \AAA = \EEE - c_V(F,\theta) \frac{ (\theta a''' (\theta)+ a''(\theta) ) (W_A(F) - W_M(F)) } {  C_1 -  \theta a'' (\theta)   (W_A(F) - W_M(F)) }.
\end{align*}  
 Taking   the absolute values \III on \EEE both sides above\ee, \AAA by \EEE \eqref{a_requirements}, \eqref{W_A_requirements_2}, \rbb and \ee \eqref{heat_capacity_lower_bound}   it follows that  
\begin{align*}
  \vert \partial^2_\theta W^{\rm in} (F,\theta) \vert
  &\leq \frac{2 C_2 C_4}{\rbb(C_1 - C_2 C_4) \rbb (2\theta_c)^{-1}\ee} c_V(F, \theta) \rbb\frac{1}{2\theta_c}\ee.
\end{align*}  
Consequently, \ref{C_entropy_vanishes} follows, as long as we choose $C_1 \geq \rbb ( 1 + 4 \theta_c )\ee C_2 C_4$.
 In view of \eqref{freeenergyexample} and $a(\theta_c) = 1$, we have
 \begin{align*}
 W(F,\theta_c) = W_A(F) + \AAA C_1 \theta_c (1- \log \theta_c). \EEE
 \end{align*}
 Thus, \ref{W_prefers_id} is derived by \eqref{W_A_requirements3} for $c_0 \leq C_6^{-1}$.
 This concludes the proof of (ii).
\end{proof}


\begin{lemma}[Specific choices of $W_A$ and $W_M$]\label{lem:specificchoice}
\III An admissible \EEE example for the densities $W_M$ and $W_A$ \AAA is \EEE given by  
\begin{align*}
W_M (F) &\defas \tilde W(F) + \left( \frac{1}{\det(F)} -1 \right)^q \III - \EEE  C_1 \theta_c (1 - \log \theta_c), \\ W_A(F) &\defas\hat W(F) + \left( \frac{1}{\det(F)} - 1 \right)^q  \III - \EEE C_1 \theta_c (1 - \log \theta_c)
\end{align*}
for any \AAA frame indifferent \EEE $\hat W \in C^\infty(GL^+(d))$ such that $\hat W(F) \geq \dist^2 (F, SO(d))$ for all $F \in GL^+(d)$, \III $ \hat W(F) = 0 $  for $F \in SO(d)$ and \AAA any frame indifferent \EEE $\tilde W \in C^\infty(GL^+(d))$ such that  $\tilde W-\hat W \in C^\infty_c(GL^+(d))$.
 \end{lemma}
 \begin{proof}
The proof follows \AAA by \EEE elementary computations and Lemma~\ref{lem:compatibilityshapememory}.  
 \end{proof}

\typeout{References}

\end{document}